\documentclass[11pt]{article}
\usepackage[a4paper,margin=2.5cm]{geometry}
\usepackage{amsmath}
\usepackage{amsthm}
\usepackage{makecell}
\usepackage{empheq}
\usepackage{amsfonts}
\usepackage{enumitem}
\usepackage{amssymb}
\usepackage{booktabs}
\usepackage{longtable}
\usepackage[ruled]{algorithm}
\usepackage{algpseudocode}
\usepackage{multirow}
\usepackage{rotating}
\usepackage{color}
\usepackage{graphics}
\usepackage{graphicx}

\usepackage{epsfig}
\usepackage{float}
\usepackage{subfigure}
\usepackage{latexsym,bm}
\usepackage{hhline}
\usepackage[greek,english]{babel}
\usepackage{mathrsfs}
\usepackage{threeparttable}
\usepackage[strings]{underscore}
\usepackage{authblk}
\usepackage{setspace}
\usepackage{xcolor}
\usepackage{appendix}
\usepackage{float}
\usepackage{array}
\usepackage{booktabs}
\usepackage{tabularx}
\usepackage{diagbox}
\usepackage{comment}
\usepackage{siunitx}

\usepackage{fix-cm} 
\usepackage{tikz}
\usetikzlibrary{calc,arrows.meta,fadings,decorations.markings}
\usepackage{listofitems}
\usepgflibrary{shapes.geometric}

\numberwithin{equation}{section}
\usepackage[]{hyperref}
\hypersetup{
     colorlinks=true,
     linkcolor=blue,
     filecolor=gray,
     urlcolor=blue,
     citecolor=blue,
 }
\usepackage{bookmark}

\newtheoremstyle{mystyle}
  {3pt}
  {3pt}
  {\itshape}
  {}
  {\bfseries}
  {}
  {0.5em}
  {}

\theoremstyle{mystyle}
\newtheorem{theorem}{Theorem}
\newtheorem{assumption}{Assumption}
\newtheorem{lemma}{Lemma}

\newtheorem{remark}{Remark}
\newtheorem{example}{Example}

\newlist{fwisteps}{enumerate}{1}
\setlist[fwisteps]{
  label=\arabic*:,
  leftmargin=2.2em,
  labelsep=0.55em,
  align=right,
  itemsep=5pt,
  parsep=0pt,
  topsep=6pt,
  partopsep=0pt
}

\makeatletter
\newenvironment{breakablealgorithm}[1]
{%
  \par
  \begingroup
  \normalsize
  \setlength{\parindent}{0pt}
  \setlength{\parskip}{0pt}
  \setlength{\abovedisplayskip}{3pt}
  \setlength{\belowdisplayskip}{3pt}
  \setlength{\abovedisplayshortskip}{2pt}
  \setlength{\belowdisplayshortskip}{2pt}
  \setlength{\jot}{2pt}

  \refstepcounter{algorithm}%
  \addcontentsline{loa}{algorithm}
  {\protect\numberline{\thealgorithm}#1}%

  \hrule height 1.2pt
  \kern 3pt
  {\raggedright
   \textbf{Algorithm \thealgorithm:} #1\par}%
  \kern 3pt
  \hrule height 0.4pt
  \nopagebreak[4]
  \kern 6pt
}
{%
  \par
  \kern 5pt
  \hrule height 0.8pt
  \endgroup
  \par
}
\makeatother

\begin{document}
\renewcommand*{\Authsep}{, }
\renewcommand*{\Authand}{, }
\renewcommand*{\Authands}{, }
\renewcommand*{\Affilfont}{\small\it}
\newcommand{\filliation}[5]{\affil[#1]{\textbf{#2}\vskip 0pt \textbf{#3}\vskip 0pt #4\vskip 0pt #5\vspace{10pt}}}

\title{\textbf{Bayesian Geometric Inversion for Wave Equation: Well-Posedness and Multiscale Multimodal Posterior Approximation}}

\author[a]{Xiong-Bin~Yan}
\author[a,$\ast$]{Ming-Jia~Pei}
\author[b]{Guang-Hui~Zheng}
\author[a]{Ting~Wei}
\affil[a]{School of Mathematics and Statistics, Lanzhou University, Lanzhou 730000, China.}
\affil[b]{School of Mathematics, Hunan University, Changsha 410082, China.}

\date{}
\maketitle
\begingroup
\renewcommand{\thefootnote}{\fnsymbol{footnote}}
\footnotetext[1]{Corresponding author.}
\endgroup
\let\thefootnote\relax\footnote{Email addresses:
\begin{minipage}[t]{0.8\linewidth}
\href{mailto:yanxb2015@163.com}{yanxb2015@163.com }(Xiong-Bin~Yan)
\href{mailto:zhenggh2012@hnu.edu.cn}{zhenggh2012@hnu.edu.cn}(Guang-Hui~Zheng),\\
\href{mailto:tingwei@lzu.edu.cn}{tingwei@lzu.edu.cn }(Ting~Wei),
\href{mailto:pmja1998@163.com}{pmja1998@163.com }(Ming-Jia~Pei) .
\end{minipage}}

\begin{abstract}
Geometric full-waveform inversion aims to recover sharp subsurface interfaces from noisy waveform observations. We formulate this problem within a Bayesian level-set framework, where the unknown interface is represented as zero level set of a latent function endowed with a Gaussian prior. Although the level-set-to-medium-parameter map is discontinuous, we establish that the level-set-to-observation map is continuous almost surely with respect to the prior, thereby guaranteeing that the posterior measure is well-posed and locally Lipschitz continuous with respect to the data in the Hellinger metric. To enable tractable posterior computation under the severe nonlinearity of the inverse problem and the resulting multimodality, we develop a multiscale multimodal Laplace approximation. A frequency-continuation strategy incorporates waveform data progressively from low to high frequencies, identifying and propagating multiple candidate modes across successive scales. At the final frequency level, local Gaussian approximations are constructed around the retained modes and combined into a Gaussian mixture, whose weights are determined by least-squares fitting to the unnormalized posterior density. Numerical experiments demonstrate that the method accurately reconstructs anomalous interfaces while effectively capturing posterior multimodality and the local uncertainty associated with distinct plausible geometries.
\end{abstract}

\section{Introduction}
Seismic imaging is a fundamental technique for imaging subsurface structures from recorded wavefield data. Various approaches have been developed for this purpose, including traveltime tomography \cite{bording1987applications,clement2001migration,zelt1992seismic},  the Born approximation \cite{hudson1981use, gelis2007two}, and full-waveform inversion (FWI) \cite{fichtner2013multiscale,xu2012full}. Among these methods, FWI reconstructs subsurface properties by iteratively minimizing the misfit between observed and simulated seismic waveforms. By utilizing the complete waveform rather than selected data attributes, FWI simultaneously exploits traveltime, phase, and amplitude information, thereby yielding high-resolution models of subsurface structures.

Conventional full-waveform inversion (FWI) typically represents the subsurface using spatially distributed parameters, such as wave speed, slowness, or density, defined at individual grid points \cite{tarantola1984inversion, liu2015improved}. This representation is natural for smoothly varying media. However, when the subsurface consists of a limited number of distinct geological regions separated by sharp interfaces, the main reconstruction targets are the location, geometry, and boundaries of the anomalous bodies rather than the parameter value at every grid point. In this situation, a pointwise parameterization may not preserve sharp material contrasts \cite{bunks1995multiscale} effectively, because the limited resolution of finite-frequency data and the use of regularization tend to produce overly smooth interfaces. Shape-based FWI provides an alternative by representing the subsurface through the geometry of the anomalous region and directly updating its shape and boundary.

The level-set method \cite{aghasi2011parametric,de2024single,zheng2024new,deng2025level,osher1988fronts} represents the unknown geometry implicitly via a scalar function $\phi:\Omega\rightarrow\mathbb{R}$. The anomalous region and its interface are defined, respectively, as
\begin{equation*}
D_\phi=\{x\in\Omega:\phi(x)\geq 0\},
\qquad \partial D_\phi=\{x\in\Omega:\phi(x)=0\}.
\end{equation*}
For a medium consisting of two materials, the material coefficient is assigned based on the sign of $\phi$. Consequently, interface reconstruction reduces to estimating a scalar function on the fixed computational domain  $\Omega$, rather than the direct optimization of the interface itself. Unlike traditional shape-based inversion methods that explicitly parameterize the boundary and iteratively estimate its parameters \cite{kirsch1993domain,kress1994quasi}, the level-set approach implicitly represents the interface as the zero level set of a scalar function, thereby eliminating explicit interface tracking \cite{osher1988fronts}. A key advantage of this implicit representation is its natural ability to handle irregular geometries and topological changes (e.g., the merging or splitting of anomalies), which renders it highly effective for geometric reconstruction in FWI.

While the level-set method elegantly represents the unknown geometry implicitly, the inverse problem remains inherently ill-posed. Most level-set approaches for FWI \cite{lanznaster2021level, de2024single, emmendoerfer2023full} are framed deterministically, minimizing a regularized data-misfit functional to yield a single estimated geometry. This estimate is highly sensitive to measurement noise, the amount and distribution of available data, and the initial level-set function. Furthermore, non-uniqueness arises because different geometries can generate similar waveform observations. As a result, a single deterministic estimate fails to quantify the uncertainty in the recovered geometry. These limitations motivate a Bayesian formulation, in which the level-set function, or its parameterization, is treated as an unknown random quantity and the available prior information and waveform observations are combined to obtain its posterior distribution.

Applying Bayesian inference \cite{stuart2010inverse,dashti2013bayesian} to level-set FWI poses specific theoretical challenges. \emph{The map from the level-set function to the corresponding piecewise-constant medium is not continuous. Since standard well-posedness results for Bayesian inverse problems typically assume continuity of the parameter-to-data map, they cannot be directly applied here. To address this, we adopt the almost-sure continuity framework developed for Bayesian level-set inversion \cite{iglesias2016bayesian} and demonstrate that the exceptional set has zero prior probability under the chosen Gaussian prior. This establishes that the posterior measure is well-defined and depends continuously on the observed waveform data.}

Even with a well-defined posterior, numerical characterization remains computationally prohibitive. Evaluating the likelihood typically requires solving the forward wave equation for multiple seismic sources. Sampling methods such as Markov chain Monte Carlo (MCMC) may therefore involve a large number of forward simulations \cite{hoang2013complexity,chen2018dimension,hu2024mcmc,bui2016fem}. The Laplace approximation provides a less expensive alternative by approximating the posterior near a maximum a posteriori (MAP) point with a Gaussian distribution. The covariance of this Gaussian is given by the inverse Hessian of the negative log-posterior evaluated at the MAP point \cite{helin2022non,rue2009approximate,ding2021determination}. This approximation describes the uncertainty only in the neighborhood of the selected MAP point. If the FWI posterior has several well-separated modes associated with different plausible geometries, a single Gaussian approximation may ignore other modes and consequently underestimate the overall uncertainty.

In this work, we develop a Bayesian level-set framework for geometric shape reconstruction in full-waveform inversion. The unknown interface is represented as the zero level set of a latent function that characterizes the subsurface medium, while the wavefield, governed by the time-domain wave equation, yields the observed seismic waveforms. We establish the well-posedness of the resulting posterior measure and propose a multiscale multimodal Laplace approximation to enable efficient posterior computation. The main contributions of this paper are summarized as follows.

\begin{itemize}
\item We formulate geometric FWI within a Bayesian level-set framework, in which the unknown geometry is represented implicitly by a level-set function. The zero level set of this function defines the interface, while its sign distinguishes the anomalous region from the background.
The posterior distribution of the level-set function then provides a probabilistic description of this geometry, quantifying the reconstruction uncertainty rather than reducing it to a single estimate.

\item We establish the well-posedness of the resulting Bayesian inverse problem. The level-set-to-medium map is not globally continuous, so the standard well-posedness theory for Bayesian inverse problems cannot be applied directly. To address this, we show that the level-set-to-observation map is almost surely continuous with respect to the chosen Gaussian prior, and that the associated potential satisfies local boundedness and local Lipschitz conditions. Consequently, the posterior measure is well-defined and depends locally Lipschitz continuously on the observed data in the Hellinger distance.

\item We propose a multiscale, multimodal Laplace approximation for characterizing the posterior. A frequency-continuation strategy progressively introduces waveform information from low to high frequencies. At each frequency level, multiple candidate modes are identified, duplicate solutions are removed, and the retained modes are propagated to the next level. At the final level, a local Gaussian approximation is built around each retained mode, and these components are assembled into a Gaussian-mixture representation of the posterior, whose nonnegative weights are obtained by least-squares fitting to the unnormalized target posterior and then rescaled to sum to one. Unlike the standard single-mode Laplace approximation, this construction captures both the local uncertainty around individual modes and the variability among distinct plausible geometries.

\item We validate the proposed method using four two-dimensional examples. The two-parameter test confirms its ability to retain separated posterior modes, while the remaining examples demonstrate the reconstruction of a single nonconvex inclusion, multiple disconnected inclusions, and a BP2004-derived salt-body geometry. In the comparison tests, the method provides more accurate boundary localization and more informative uncertainty estimates than the standard single-component Laplace approximation, while requiring substantially less computational time than gPCN-based MCMC.
\end{itemize}

The remainder of this paper is organized as follows. Section~\ref{sec：wellposed_wave} introduces the time-domain wave propagation model and establishes the well-posedness, regularity, and uniform estimates of the forward problem needed for the boundary observations. Section~\ref{sec:bayesian-level-set} formulates the geometric inverse problem within a Bayesian level-set framework and proves that the resulting posterior measure is well defined and stable with respect to the observed data. Section~\ref{sec:Inversion_algorithm} introduces a finite-dimensional representation of the level-set function and presents the proposed multiscale multimodal Laplace approximation, including frequency continuation, mode selection and propagation, local Gaussian approximations, and mixture-weight estimation. Section~\ref{sec:Numerical experiments} reports numerical experiments that evaluate the geometric reconstructions and the associated uncertainty estimates. Finally, Section~\ref{sec:conclusions} summarizes the main conclusions. The appendices provide the regularized parameterization, likelihood derivatives, prior and posterior-mixture implementation, and evaluation metrics.

\section{Well-posedness, regularity, and energy estimates for the wave equation}\label{sec：wellposed_wave}
Let $\Omega\subset\mathbb{R}^2$ be a bounded Lipschitz domain. Let
$\Gamma_{D}\subset\partial\Omega$ be a connected subset on which we impose the homogeneous Dirichlet boundary condition, and let $\Gamma_N:=\partial\Omega\setminus\Gamma_{D}$, where we impose the homogeneous Neumann boundary condition. We denote $H_{\Gamma_{D}}^1(\Omega)$ the space of functions in $H^1(\Omega)$ with vanishing trace on $\Gamma_{D}$. 

We consider the space-time field $u(\boldsymbol{x},t)$ in a spatial domain $\Omega\subset\mathbb{R}^2$, satisfying the following initial--boundary value problem, which models wave propagation in a heterogeneous medium with damping: 
\begin{align}
\left\{
\begin{aligned}
    &\mu(\boldsymbol{x})u_{tt} + \sigma(\boldsymbol{x}) u_{t} = \Delta u + S, \quad \forall (\boldsymbol{x},t)\in\Omega\times(0,T], \\
    &u(\boldsymbol{x},0)=0, \quad \forall \boldsymbol{x}\in\Omega, \\
    &u_t(\boldsymbol{x},0)=0, \quad \forall \boldsymbol{x}\in\Omega, \\
    &u(\boldsymbol{x},t)=0, \quad \forall (\boldsymbol{x},t)\in\Gamma_D\times(0,T], \\
    &\nabla u(\boldsymbol{x},t)\cdot\mathbf{n}=0, \quad \forall (\boldsymbol{x},t)\in\Gamma_N\times(0,T],
\end{aligned}
\right.
\label{eq:forward_problem}
\end{align}

To ensure the well-posedness of the forward problem, we impose the following assumptions on the medium parameters.
\begin{assumption}\label{assumption1}
    Let $\mu,~\sigma\in L^{\infty}(\Omega)$, $\sigma\ge 0$, and  $\mu_{\min}\le \mu(x)\le \mu_{\max}$, here $\mu_{\min}$ and
    $\mu_{\max}$ are positive constants.
\end{assumption}

Following \cite{albuquerque2021level}, we adopt the mild solution framework for our theoretical analysis. A function $u\in C([0,T]; H_{\Gamma_D}^1(\Omega))\cap C^1([0,T];L^2(\Omega))$ is named a mild solution to the hyperbolic 
forward problem \eqref{eq:forward_problem} if and only if

\begin{empheq}[left=\empheqlbrace]{equation}
\begin{aligned}
    &\frac{d}{dt} \int_{\Omega} \mu u_t(t)\psi \,d\boldsymbol{x} + \int_{\Omega} \bigl( \nabla u(t) \cdot \nabla \psi + \sigma u_t(t)\psi \bigr) \,d\boldsymbol{x} = \int_{\Omega} S(t)\psi \,d\boldsymbol{x}, \\[1ex]
    &\qquad \text{for all } \psi \in H^1_{\Gamma_D}(\Omega) \text{ and a.e. } t \in (0,T), \\[1ex]
    &\qquad u(0) = u_t(0) = 0,
\end{aligned}\label{eq:mild_eq}
\end{empheq}
and for $\psi\in H_{\Gamma_{D}}^1(\Omega)$, the mapping $t \to (\mu u_t(t),\psi)_{L^2(\Omega)}$ is absolutely continuous on
$[0,T]$.

The following lemma, which provides the existence, uniqueness, and energy estimates for the mild solution, is taken from \cite{albuquerque2021level}.
\begin{lemma}[\cite{albuquerque2021level}]\label{lem:forward_wellposedness}
Suppose that Assumption~\ref{assumption1} holds. Then, for every
$S \in L^{1}\bigl(0,T;L^{2}(\Omega)\bigr)$,
the forward hyperbolic problem~\eqref{eq:forward_problem} admits a unique
mild solution $u \in C\bigl([0,T];H_{\Gamma_D}^{1}(\Omega)\bigr)
    \cap C^{1}\bigl([0,T];L^{2}(\Omega)\bigr).$
Moreover, for every \(t\in[0,T]\), the solution satisfies
\begin{equation}\label{eq:forward_L2_estimate}
\int_{\Omega} \mu(\boldsymbol{x})\lvert u(\boldsymbol{x},t)\rvert^{2}\,\mathrm{d}\boldsymbol{x}
\leq t\int_{0}^{t} \left\| \mu^{1/2}S(\cdot,\tau)
\right\|_{L^{2}(\Omega)}\,\mathrm{d}\tau,
\end{equation}
and
\begin{equation}\label{eq:forward_energy_estimate}
\left(\int_{\Omega} \mu(\boldsymbol{x})\lvert u_t(\boldsymbol{x},t)\rvert^{2}\,\mathrm{d}\boldsymbol{x}
+\int_{\Omega} |\nabla u(\boldsymbol{x},t)|^2\,\mathrm{d}\boldsymbol{x}
\right)^{1/2}
\leq\int_{0}^{t}\left\|\mu^{1/2}S(\cdot,\tau)
\right\|_{L^{2}(\Omega)}\,\mathrm{d}\tau.
\end{equation}

If, in addition, $S\in W^{1,1}\bigl(0,T;L^{2}(\Omega)\bigr)$,
then the mild solution possesses the higher regularity
$u\in C^{2}\bigl([0,T];L^{2}(\Omega)\bigr)$,
and satisfies
\begin{equation}\label{eq:forward_strong_equation}
\mu u_{tt}(t) -\Delta u(t)+\sigma u_t(t)
=S(t),\qquad t\in[0,T],
\end{equation}
in $L^{2}(\Omega)$. Consequently, $u$ is a solution of the
forward problem \eqref{eq:forward_problem}.
\end{lemma}
\begin{assumption}\label{assum:domain}
    We consider a specific rectangular domain 
    $\Omega = (0, L_x) \times (0, L_y) \subset \mathbb{R}^2$. 
    The boundary $\partial\Omega$ is partitioned into a Neumann boundary 
    $\Gamma_N = (0, L_x) \times \{0\}$ and a Dirichlet boundary 
    $\Gamma_D = \partial\Omega \setminus\Gamma_N$ (which consists of the remaining three sides). 
\end{assumption}

\begin{theorem}\label{thm:spatial-regularity}
 Let Assumption \ref{assum:domain} hold, and let
$S\in W^{1,1}\bigl(0,T;L^2(\Omega)\bigr).$
Then the unique mild solution $u$ of problem~\eqref{eq:forward_problem}
satisfies
\begin{equation*}
    u\in C\bigl([0,T];H^2(\Omega)\bigr)
    \cap C^2\bigl([0,T];L^2(\Omega)\bigr).
\end{equation*}
Moreover, there exists a constant $C_\Omega>0$, depending only on
$\Omega$, such that, for every $t\in[0,T]$,
\begin{equation}\label{eq:H2-estimate-u}
    \|u(t)\|_{H^2(\Omega)}
    \leq C_\Omega\Bigl(
        \|S(t)\|_{L^2(\Omega)}
        +\mu_{\max}\|u_{tt}(t)\|_{L^2(\Omega)}
        +\|\sigma\|_{L^\infty(\Omega)} \|u_t(t)\|_{L^2(\Omega)}\Bigr).
\end{equation}
In particular,
\begin{equation*}
    u\in C\bigl([0,T];C^{0,\alpha}(\overline{\Omega})\bigr),\qquad 0<\alpha<1.
\end{equation*}
\end{theorem}

\begin{proof}
Set $V:=H^1_{\Gamma_D}(\Omega)$.
By Lemma~\ref{lem:forward_wellposedness}, the unique mild solution satisfies
\begin{equation*}
u\in C([0,T];V)\cap C^2([0,T];L^2(\Omega)).
\end{equation*}
In particular, $u_t, u_{tt} \in C([0,T]; L^2(\Omega))$. Since $W^{1,1}(0,T; L^2(\Omega))$ embeds continuously into $C([0,T]; L^2(\Omega))$, the source term satisfies $S \in C([0,T]; L^2(\Omega))$. Thus, we may define
\begin{equation*}
g(t):=S(t)-\mu u_{tt}(t)-\sigma u_t(t).
\end{equation*}
Since $\mu, \sigma \in L^\infty(\Omega)$ are independent of $t$, it follows that
\begin{equation}\label{eq:reflected-source-g}
    g\in C([0,T]; L^2(\Omega)).
\end{equation}

At each fixed $t\in[0,T]$, the variational formulation of problem \eqref{eq:forward_problem} can be written as
\begin{equation}\label{eq:original-mixed-variational}
\int_\Omega\nabla u(t)\cdot\nabla v\,dx\,dy
=\int_\Omega g(t)v\,dx\,dy,
\qquad v\in V.
\end{equation}
Thus, for each fixed $t$, the function $u(t)$ is the weak solution of mixed elliptic problem
\begin{equation}\label{eq:mixed-elliptic-problem}
\begin{cases}
    -\Delta u(t)=g(t),
        & \text{in }\Omega,\\
    u(t)=0,
        & \text{on }\Gamma_D,\\
    \partial_{\boldsymbol n}u(t)=0,
        & \text{on }\Gamma_N.
\end{cases}
\end{equation}

We next establish the $H^2$-regularity for problem \eqref{eq:mixed-elliptic-problem}. 
 Let $h\in L^2(\Omega)$ and
$w\in V$ satisfy
\begin{equation}\label{eq:generic-mixed-variational}
\int_\Omega\nabla w\cdot\nabla v\,dx\,dy
=\int_\Omega hv\,dx\,dy,
\qquad v\in V.
\end{equation}
To this end, we perform an even reflection across the Neumann boundary on $y = 0$. Specifically, we define

\begin{equation*}
\widetilde\Omega=(0,L_x)\times(-L_y,L_y),\qquad
\widetilde w(x,y)=w(x,|y|),\qquad
\widetilde h(x,y)=h(x,|y|).
\end{equation*}
The even extension maps $H^1(\Omega)$ continuously into
$H^1(\widetilde\Omega)$. Since $w$ has zero trace on the left, right, and top
sides of $\Omega$, its reflection has zero trace on every side of
$\partial\widetilde\Omega$. Therefore
\begin{equation*}
\widetilde w\in H^1_0(\widetilde\Omega),
\qquad \widetilde h\in L^2(\widetilde\Omega).
\end{equation*}

For any
$\psi\in H^1_0(\widetilde\Omega)$, let
\begin{equation*}
v_\psi(x,y)=\psi(x,y)+\psi(x,-y),\qquad (x,y)\in\Omega.
\end{equation*}
Since $\psi \in H_0^1(\widetilde{\Omega})$, the restriction satisfies $v_\psi \in V$, and almost everywhere in $\Omega$, we have
\begin{equation*}
\partial_x v_\psi(x,y)=\partial_x\psi(x,y)+\partial_x\psi(x,-y),
\qquad
\partial_y v_\psi(x,y)=\partial_y\psi(x,y)-\partial_y\psi(x,-y).
\end{equation*}
Splitting the integral over $\widetilde{\Omega}$ into its upper and lower halves, and changing variables $y \mapsto -y$ in the lower half, we obtain
\begin{align*}
\int_{\widetilde{\Omega}} \nabla \widetilde{w} \cdot \nabla \psi \, \mathrm{d}x\mathrm{d}y 
&= \int_\Omega \left[ w_x (\psi_x(x, y) + \psi_x(x, -y)) + w_y (\psi_y(x, y) - \psi_y(x, -y)) \right] \mathrm{d}x\mathrm{d}y \notag \\
&= \int_\Omega \nabla w \cdot \nabla v_\psi \, \mathrm{d}x\mathrm{d}y = \int_\Omega h v_\psi \, \mathrm{d}x\mathrm{d}y \notag \\
&= \int_{\widetilde{\Omega}} \widetilde{h} \psi \, \mathrm{d}x\mathrm{d}y.
\end{align*}
Hence, $\widetilde{w}$ is the unique weak solution to the homogeneous Dirichlet problem
\begin{equation*}
\begin{cases}
-\Delta \widetilde{w} = \widetilde{h}, & \text{in } \widetilde{\Omega}, \\
\widetilde{w} = 0, & \text{on } \partial\widetilde{\Omega}.
\end{cases}
\end{equation*}
Since $\widetilde{\Omega}$ is a bounded convex polygon, by the $H^2$-regularity theory for elliptic problems (see, e.g., Grisvard~\cite[Theorem~3.2.1.2]{grisvard2011elliptic}), there exists a constant $C_{\widetilde{\Omega}} > 0$ such that $\widetilde{w} \in H^2(\widetilde{\Omega})$ and
\begin{equation*}
\|\widetilde{w}\|_{H^2(\widetilde{\Omega})} \le C_{\widetilde{\Omega}} \|\widetilde{h}\|_{L^2(\widetilde{\Omega})}.
\end{equation*}
By symmetry, we derive
\begin{equation}\label{eq: estimate_H2}
\|w\|_{H^2(\Omega)} \le C_\Omega \|h\|_{L^2(\Omega)}.
\end{equation}

Applying \eqref{eq: estimate_H2} to $w=u(t)$ and $h=g(t)$
in \eqref{eq:original-mixed-variational}, we obtain
\begin{equation*}
\|u(t)\|_{H^2(\Omega)}
\leq C_\Omega\|g(t)\|_{L^2(\Omega)}
\leq C_\Omega\bigl(
\|S(t)\|_{L^2(\Omega)}
+\mu_{\max}\|u_{tt}(t)\|_{L^2(\Omega)}
+\|\sigma\|_{L^\infty(\Omega)}\|u_t(t)\|_{L^2(\Omega)}
\bigr),
\end{equation*}
which proves \eqref{eq:H2-estimate-u}.

For any $s, t \in [0, T]$, the difference $w = u(t) - u(s)$ satisfies the variational problem \eqref{eq:original-mixed-variational} with the source term $h = g(t) - g(s)$. Applying the elliptic regularity estimate \eqref{eq: estimate_H2}, we obtain
\begin{equation*}
\|u(t) - u(s)\|_{H^2(\Omega)} \le C_\Omega \|g(t) - g(s)\|_{L^2(\Omega)}.
\end{equation*}
Since $g \in C([0, T]; L^2(\Omega))$ by \eqref{eq:reflected-source-g}, the right-hand side vanishes as $s \to t$, which implies that $u \in C([0, T]; H^2(\Omega))$. Finally, the continuous Sobolev
embedding $H^2(\Omega)\hookrightarrow
C^{0,\alpha}(\overline\Omega)$ in two dimensions, valid for every
$0<\alpha<1$, gives
$u\in C([0,T];C^{0,\alpha}(\overline\Omega))$.
\end{proof}

Having established the spatial regularity of the solution, we now derive energy estimates for the wave problem \eqref{eq:forward_problem}. In particular, we establish uniform bounds with respect to the medium coefficient $\mu$, which are crucial for proving the well-posedness of the posterior measure. These estimates are summarized in the following theorem.

\begin{theorem}\label{thm:uniform-energy-estimates}
Let $\Gamma_D$ non-empty and Assumption \ref{assumption1} hold.
Let $S\in W^{1,1}(0,T;L^2(\Omega))$, and let $u$ solve the wave problem
\eqref{eq:forward_problem}.
Then
\begin{equation*}
    u\in C^1([0,T];V)\cap C^2([0,T];L^2(\Omega)),
\end{equation*}
and the following estimates hold:
\begin{align}
\label{eq:estimate-ut}
    \sup_{t\in[0,T]}\|u_t(t)\|_{L^2(\Omega)}
    &\leq \frac{1}{\mu_{\min}} \|S\|_{L^1(0,T;L^2(\Omega))},\\ \label{eq:estimate-grad-u}
    \sup_{t\in[0,T]} \|\nabla u(t)\|_{L^2(\Omega)}
    &\leq \frac{1}{\sqrt{\mu_{\min}}} \|S\|_{L^1(0,T;L^2(\Omega))},\\
    \label{eq:estimate-utt}
    \sup_{t\in[0,T]}\|u_{tt}(t)\|_{L^2(\Omega)}
    &\leq \frac{1}{\mu_{\min}}\bigl(
        \|S(0)\|_{L^2(\Omega)} +\|S_t\|_{L^1(0,T;L^2(\Omega))}
\bigl).
\end{align}
Consequently, the estimates are uniform for all coefficients $\mu$ satisfying
the prescribed bounds $\mu_{\min}\leq\mu\leq\mu_{\max}$; in particular, their
constants do not depend on the particular choice of $\mu$.
\end{theorem}

\begin{proof}
We first estimate $u_t$ and $\nabla u$. Define
\begin{equation*}
    E_u(t):=\frac{1}{2}\bigl(
        \|\mu^{1/2}u_t(t)\|_{L^2(\Omega)}^2+\|\nabla u(t)\|_{L^2(\Omega)}^2\bigl).
\end{equation*}
Testing the equation \eqref{eq:forward_problem} by $u_t$ and using the homogeneous boundary conditions
give
\begin{equation}\label{eq:basic-energy-identity}
    E_u'(t)+\|\sigma^{1/2}u_t(t)\|_{L^2(\Omega)}^2
    =(S(t),u_t(t))_{L^2(\Omega)}.
\end{equation}
Indeed, since $\mu$ is independent of time,
\begin{equation*}
    \int_\Omega \mu u_{tt}u_t\,dx
    =\frac{1}{2}\frac{d}{dt}\|\mu^{1/2}u_t\|_{L^2(\Omega)}^2,
\end{equation*}
whereas integration by parts gives
\begin{equation*}
    -\int_\Omega \Delta u\,u_t\,dx
    =\int_\Omega\nabla u\cdot\nabla u_t\,dx
    =\frac{1}{2}\frac{d}{dt}\|\nabla u\|_{L^2(\Omega)}^2.
\end{equation*}

By the Cauchy--Schwarz inequality,
\begin{align*}
    (S(t),u_t(t))_{L^2(\Omega)}
    &=\bigl(\mu^{-1/2}S(t),\mu^{1/2}u_t(t)\bigr)_{L^2(\Omega)}\\
    &\leq
    \|\mu^{-1/2}S(t)\|_{L^2(\Omega)}
    \|\mu^{1/2}u_t(t)\|_{L^2(\Omega)}\\
    &\leq \|\mu^{-1/2}S(t)\|_{L^2(\Omega)}\sqrt{2E_u(t)}.
\end{align*}
It follows from \eqref{eq:basic-energy-identity} that
\begin{equation}\label{eq:energy-differential-inequality}
    E_u'(t)\leq\|\mu^{-1/2}S(t)\|_{L^2(\Omega)}\sqrt{2E_u(t)}.
\end{equation}

To avoid dividing by zero when $E_u(t)=0$, let
\begin{equation*}
    Y_\varepsilon(t):=\sqrt{2E_u(t)+\varepsilon},
    \qquad \varepsilon>0.
\end{equation*}
Then \eqref{eq:energy-differential-inequality} yields
\begin{equation*}
    Y_\varepsilon'(t)
    =\frac{E_u'(t)}{\sqrt{2E_u(t)+\varepsilon}}
    \leq \|\mu^{-1/2}S(t)\|_{L^2(\Omega)}.
\end{equation*}
Since $E_u(0)=0$, integration over $(0,t)$ and passage to the limit
$\varepsilon\downarrow0$ give
\begin{equation}\label{eq:energy-u-bound}
    \sqrt{2E_u(t)}\leq\int_0^t\|\mu^{-1/2}S(s)\|_{L^2(\Omega)}\,ds
    \leq\frac{1}{\sqrt{\mu_{\min}}}\int_0^t\|S(s)\|_{L^2(\Omega)}\,ds.
\end{equation}
On the other hand,
\begin{equation*}
    \sqrt{\mu_{\min}}\|u_t(t)\|_{L^2(\Omega)}
    \leq\|\mu^{1/2}u_t(t)\|_{L^2(\Omega)}\leq\sqrt{2E_u(t)},
\end{equation*}
and
\begin{equation*}
    \|\nabla u(t)\|_{L^2(\Omega)}\leq\sqrt{2E_u(t)}.
\end{equation*}
Combining these inequalities with \eqref{eq:energy-u-bound} and taking the
supremum over $t\in[0,T]$ prove \eqref{eq:estimate-ut} and
\eqref{eq:estimate-grad-u}.

We next estimate $u_{tt}$. Let $v$ be the solution of
\begin{equation}\label{eq:differentiated-problem}
\begin{cases}
    \mu v_{tt}+\sigma v_t-\Delta v=S_t, &\text{in }\Omega\times(0,T),\\
    v=0,&\text{on }\Gamma_D\times(0,T),\\
    \partial_{\boldsymbol n}v=0,&\text{on }\Gamma_N\times(0,T),\\
    v(\cdot,0)=0,\qquad v_t(\cdot,0)=\mu^{-1}S(0),&\text{in }\Omega.
\end{cases}
\end{equation}
The value $S(0)$ is well defined because
$S\in W^{1,1}(0,T;L^2(\Omega))$. To see that $v=u_t$, we set
\begin{equation*}
    U(t):=\int_0^t v(s)\,ds.
\end{equation*}
Integrating the equation for $v$ from $0$ to $t$ and using the initial data in \eqref{eq:differentiated-problem}, we obtain
\begin{equation*}
    \mu U_{tt}+\sigma U_t-\Delta U=S,
    \qquad U(0)=U_t(0)=0,
\end{equation*}
and
\begin{equation*}
    \quad U=0~\text{on}~\Gamma_{D},\quad \partial_{\boldsymbol{n}}U=0 \text{ on } \Gamma_{N}.
\end{equation*}
Thus, $U$ solves \eqref{eq:forward_problem}. By the uniqueness of the solution to \eqref{eq:forward_problem}, we have $U = u$, which in turn implies that $v = u_t$ and $v_t = u_{tt}$.

Define
\begin{equation*}
    E_v(t):=\frac{1}{2}\bigl(\|\mu^{1/2}v_t(t)\|_{L^2(\Omega)}^2
        +\|\nabla v(t)\|_{L^2(\Omega)}^2\bigl).
\end{equation*}
Repeating the preceding energy argument with $v$ and source term $S_t$ yields
\begin{equation}\label{eq:energy-v-bound}
    \sqrt{2E_v(t)}\leq\sqrt{2E_v(0)}
    +\int_0^t\|\mu^{-1/2}S_t(s)\|_{L^2(\Omega)}\,ds.
\end{equation}
The initial conditions in \eqref{eq:differentiated-problem} give
\begin{equation*}
    \sqrt{2E_v(0)}=\|\mu^{-1/2}S(0)\|_{L^2(\Omega)}
    \leq\frac{1}{\sqrt{\mu_{\min}}}\|S(0)\|_{L^2(\Omega)}.
\end{equation*}
Therefore,
\begin{equation*}
    \sqrt{2E_v(t)}\leq\frac{1}{\sqrt{\mu_{\min}}}
    \bigl(\|S(0)\|_{L^2(\Omega)}
        +\int_0^t\|S_t(s)\|_{L^2(\Omega)}\,ds
    \bigl).
\end{equation*}
Finally, since $v_t=u_{tt}$,
\begin{equation*}
    \sqrt{\mu_{\min}}\|u_{tt}(t)\|_{L^2(\Omega)}
    \leq\|\mu^{1/2}v_t(t)\|_{L^2(\Omega)}
    \leq\sqrt{2E_v(t)}.
\end{equation*}
Taking the supremum over $t\in[0,T]$ proves
\eqref{eq:estimate-utt}. The regularity stated in the theorem follows from the standard energy theory applied first to $u$ and then to
\eqref{eq:differentiated-problem}.
\end{proof}

For later use, we record a uniform $H^2$-bound \eqref{eq:H2-estimate-u} for the solution. Set
\begin{equation}\label{eq:A_B}
 A:=\|S\|_{L^1(0,T;L^2(\Omega))},
 \qquad
 B:=\|S(0)\|_{L^2(\Omega)}
    +\|\partial_t S\|_{L^1(0,T;L^2(\Omega))}.
\end{equation}
Since $S\in W^{1,1}((0,T);L^2(\Omega))$, we have
\begin{equation*}
 \|S\|_{C([0,T];L^2(\Omega))}\le B.
\end{equation*}
Combining the preceding energy estimates with the $H^2$-estimate
\eqref{eq:H2-estimate-u} yields
\begin{equation}\label{eq:uniform-field-bound}
 \sup_{t\in[0,T]}\|u(t)\|_{H^2(\Omega)}
 \le C_\Omega\left[
    \left(1+\frac{\mu_{\max}}{\mu_{\min}}\right)B
    +\frac{\|\sigma\|_{L^\infty(\Omega)}}{\mu_{\min}}A
 \right]
 =:M.
\end{equation}
Here $C_\Omega>0$ depends only on the fixed domain and its prescribed
boundary decomposition. In particular, the estimate is uniform over
all coefficients $\mu$ satisfying Assumption~\ref{assumption1}, since its right-hand side depends on $\mu$ only through the fixed bounds
$\mu_{\min}$ and $\mu_{\max}$.
\section{Bayesian level-set formulation and posterior well-posedness}
\label{sec:bayesian-level-set}

This section formulates geometric full-waveform inversion within a Bayesian
level-set framework. We first introduce the geometric inverse problem and its
observation model. We then represent the unknown inclusion by a level-set
function and define the corresponding parameter-to-observation map. Finally,
we introduce the prior and posterior measures and establish the well-posedness
of the posterior.

\subsection{Geometric inverse problem and observation model}
\label{subsec:geometric-inverse-problem}

Let $D\subset\Omega$ denote an unknown inclusion. The material parameters
inside and outside $D$ are prescribed distinct positive constants
$\mu_{\mathrm{in}}$ and $\mu_{\mathrm{out}}$, respectively. The resulting
medium coefficient is
\begin{equation}\label{eq:binary-medium}
\mu_D=\mu_{\mathrm{in}}\mathbf{1}_D
+\mu_{\mathrm{out}}\mathbf{1}_{\Omega\setminus D}.
\end{equation}
Thus, the unknown is the geometry of $D$, whereas the material values on the two sides of its boundary are fixed. If $\mu_D$ represents squared slowness, the corresponding wave speed is $v_D=\mu_D^{-1/2}$. The damping coefficient $\sigma\in L^\infty(\Omega)$ is assumed to be known, nonnegative, and independent of time.

For each prescribed source $S^{(s)}$, $s=1,\ldots,N_s$, let $u_D^{(s)}$
denote the solution of the forward wave problem with coefficient $\mu_D$.
In particular,
\begin{equation}\label{eq:forward-multisource}
\begin{cases}
\mu_D u_{tt}^{(s)}+\sigma u_t^{(s)}-\Delta u^{(s)}=S^{(s)},
& \text{in }\Omega\times(0,T),\\
u^{(s)}=0,& \text{on }\Gamma_D\times(0,T),\\
\partial_{\boldsymbol n}u^{(s)}=0,& \text{on }\Gamma_N\times(0,T),\\
u^{(s)}(\cdot,0)=u_t^{(s)}(\cdot,0)=0,& \text{in }\Omega.
\end{cases}
\end{equation}
We assume that
\begin{equation*}
S^{(s)}\in W^{1,1}(0,T;L^2(\Omega)), \qquad s=1,\ldots,N_s.
\end{equation*}

Let $\boldsymbol{x}_r\in\overline{\Omega}$, $r=1,\ldots,N_r$, be the receiver locations, and let $0<t_1<\cdots<t_{N_t}=T$ be the recording times. Define the
coefficient-to-observation map by
\begin{equation*}\label{eq:coefficient-observation-map}
\mathcal G(\mu)
=\bigl(u_\mu^{(s)}(\boldsymbol{x}_r,t_n)\bigr)_{s,r,n}\in\mathbb R^M,
\qquad M=N_sN_rN_t,
\end{equation*}
where the entries are stacked in a fixed order.

The observation model is
\begin{equation}
\label{eq:data-model-coefficient}
d=\mathcal G(\mu_D)+\eta,
\end{equation}
where $\eta\in\mathbb R^M$ denotes measurement noise. The inverse problem is
to recover the inclusion $D$ from these finitely many noisy boundary
measurements.

\subsection{Level-set representation of the unknown geometry}
\label{subsec:level-set-representation}

We represent the inclusion implicitly by a level-set function
$\phi\in U:=C(\overline\Omega)$,
where $U$ is equipped with the supremum norm. Define
\begin{equation*}\label{eq:level-set-domain}
D_\phi:=\{\boldsymbol{x}\in\Omega:\phi(\boldsymbol{x})\geq0\}.
\end{equation*}
The corresponding interface satisfies
$ \partial D_\phi\cap\Omega \subseteq \{\boldsymbol{x}\in\Omega:\phi(\boldsymbol{x})=0\}.$
When the zero level set has empty interior, the two sets agree.

Define the level-set-to-medium map
\begin{equation*}\label{eq:level-set-medium-map}
F:U\longrightarrow L^\infty(\Omega),
\qquad F(\phi)=\mu_\phi,
\end{equation*}
where
\begin{equation*} \label{eq:mu-level-set}
\begin{aligned}
\mu_\phi
&=\mu_{\mathrm{in}}\mathbf{1}_{D_\phi}
+\mu_{\mathrm{out}}\mathbf{1}_{\Omega\setminus D_\phi}\\
&=\mu_{\mathrm{in}}\mathbf{1}_{\{\phi\geq0\}}
+\mu_{\mathrm{out}}\mathbf{1}_{\{\phi<0\}}.
\end{aligned}
\end{equation*}
Notice that different level-set functions may define the same inclusion.
The level-set function is an auxiliary representation; the physical quantity to be reconstructed is the inclusion $D_\phi$ and its boundary. 

By construction, $\mu_\phi$ satisfies the uniform coefficient bounds
in Assumption~\ref{assumption1} for every $\phi\in U$. More precisely,
\begin{equation*}
 0<\mu_{\min}\le \mu_\phi(\boldsymbol{x})\le\mu_{\max}
 \qquad\text{for a.e. }\boldsymbol{x}\in\Omega,
\end{equation*}
where
\begin{equation*}
 \mu_{\min}:=\min\{\mu_{\mathrm{in}},\mu_{\mathrm{out}}\},
 \qquad
 \mu_{\max}:=\max\{\mu_{\mathrm{in}},\mu_{\mathrm{out}}\}.
\end{equation*}
Consequently, all forward estimates established in the preceding
section apply to $\mu_\phi$, with constants independent of
$\phi$.

The parameter-to-observation map for the level-set formulation is the
composition
\begin{equation*}\label{eq:observation-map}
G:=\mathcal G\circ F,\qquad
G(\phi)=\bigl(u_\phi^{(s)}(\boldsymbol{x}_r,t_n)\bigr)_{s,r,n}\in\mathbb R^M.
\end{equation*}
Accordingly, the observation model becomes
\begin{equation*}
\label{eq:data-model}
d=G(\phi)+\eta.
\end{equation*}

By Theorem~\ref{thm:spatial-regularity},
\begin{equation*}
u_\phi^{(s)}\in C([0,T];H^2(\Omega)).
\end{equation*}
Since $\Omega\subset\mathbb R^2$ and
$H^2(\Omega)\hookrightarrow C(\overline\Omega)$, all receiver evaluations in
\eqref{eq:observation-map} are well defined. Furthermore, the uniform
estimate obtained in the preceding section implies that there exists a
constant $C_G>0$, independent of $\phi$, such that
\begin{equation}\label{eq:uniform-observation-bound}
\sup_{\phi\in U}\|G(\phi)\|_{\mathbb R^M}\leq C_G.
\end{equation}
This bound will be used below to control the likelihood potential.

\subsection{Bayesian formulation and posterior well-posedness}

We place a Gaussian prior
\begin{equation*}
 \pi_0=\mathcal N(m_0,\mathcal C_0)
\end{equation*}
on the level set function, where $m_0\in U$ is the prior mean.
The covariance is chosen so that the Gaussian field has continuous
sample paths on $\overline\Omega$ and
\begin{equation}\label{eq:prior-nondegeneracy}
 \operatorname{Var}_{\pi_0}[\phi(\boldsymbol{x})]>0
 \quad\text{for Lebesgue-a.e. }\boldsymbol{x}\in\Omega.
\end{equation}
Thus $\pi_0$ is regarded as a Gaussian probability measure on $U$.
The covariance operator remains general in this formulation; it specifies
the amplitude and spatial correlation of the prior. We use the
$\pi_0$-completion of the Borel sigma-algebra on $U$.

Assume that the noise is independent of $\phi$ and satisfies
$\eta\sim\mathcal N(0,\Gamma)$, where $\Gamma\in\mathbb{R}^{M\times M}$ is a known symmetric positive definite covariance matrix. Define
\begin{equation*}
 \lvert z\rvert_\Gamma:=(z^\top\Gamma^{-1}z)^{1/2}.
\end{equation*}
The likelihood and its negative logarithm, up to a constant independent
of $\phi$, are
\begin{equation}\label{eq:likelihood}
 p(d\mid\phi)\propto\exp[-\Phi(\phi;d)],
 \qquad \Phi(\phi;d):=\frac{1}{2}| d-G(\phi)|_\Gamma^2.
\end{equation}
In the case of independent noise with common variance $\delta^2$,
$\Gamma=\delta^2 I_M$ and
\begin{equation*}
 \Phi(\phi;d)=\frac{1}{2\delta^2}
  \sum_{s=1}^{N_s}\sum_{r=1}^{N_r}\sum_{n=1}^{N_t}
  \left|u_\phi^{(s)}(\boldsymbol{x}_r,t_n)-d_{srn}\right|^2.
\end{equation*}

Let $\pi_0$ denote the prior measure on $U$, and define
\begin{equation}
 Z(d):=\int_U\exp[-\Phi(\phi;d)]\,\pi_0(d\phi).
 \label{eq:normalizing-constant}
\end{equation}
Provided that $0<Z(d)<\infty$, Bayes' formula defines the posterior
measure $\pi^y$ through its Radon--Nikodym derivative
\begin{equation} \label{eq:posterior}
 \frac{d\pi^d}{d\pi_0}(\phi)
 =\frac{1}{Z(d)}\exp[-\Phi(\phi;d)].
\end{equation} 

The main analytical difficulty is that the level-set-to-medium map
$F$ is not continuous on all of $U$. In particular, continuity may
fail at a function $\phi$ whose zero level set
\begin{equation*}
 \{\boldsymbol{x}\in\Omega:\phi(\boldsymbol{x})=0\}
\end{equation*}
has positive Lebesgue measure. Consequently, global continuity of the
parameter-to-observation map $G$ cannot be assumed. To apply the
posterior well-posedness theory in \cite{dashti2013bayesian,stuart2010inverse}, we instead show
that this exceptional set has zero measure under the Gaussian prior
specified below. It follows that $G$, and hence
$\Phi(\,\cdot\,;d)$, is continuous at $\pi_0$-almost every
$\phi\in U$. The conditions required for the subsequent
well-posedness analysis are summarized in the following assumption.

\begin{assumption}[\cite{iglesias2016bayesian}]\label{assumption-3}
The function $\Phi : U \times \mathbb{R}^M \to \mathbb{R}$ and probability measure $\pi_0$ on the measure space $(U, \Sigma)$ satisfy the following properties:
\item[1.] for every $r > 0$ there is a $K = K(r)$ such that, for all $\phi \in U$ and all $d \in \mathbb{R}^M$ with $|d|_\Gamma < r$,
\begin{equation*}
0 \leq \Phi(\phi;d) \leq K;
\end{equation*}
\item[2.] for any fixed $d \in \mathbb{R}^M$, $\Phi(\cdot;d) : U \to \mathbb{R}$, is continuous $\pi_0$-almost surely on the complete probability space $(U, \Sigma, \pi_0)$;
\item[3.] for $d_1, d_2 \in \mathbb{R}^M$ with $\max\{|d_1|_\Gamma, |d_2|_\Gamma\} < r$, there exists a $C = C(r)$ such that, for all $u \in U$,
\begin{equation*}
| \Phi(\phi;d_1) - \Phi(\phi;d_2) | \leq C |d_1 - d_2|_\Gamma.
\end{equation*}
\end{assumption}

To quantify the stability of the posterior measure with respect to data perturbations, we employ the Hellinger distance. Since all posterior measures considered herein are absolutely continuous with respect to the prior $\pi_0$, we choose $\pi_0$ as a common reference measure. For any two probability measures $\pi, \pi' \ll \pi_0$ on $U$, their Hellinger distance is defined by
\begin{equation}\label{eq:Hellinger-distance}
 d_{\mathrm{Hell}}(\pi,\pi')
 :=\left[\frac12\int_U
 \left(\sqrt{\frac{d\pi}{d\pi_0}(\phi)}
 -\sqrt{\frac{d\pi'}{d\pi_0}(\phi)}\right)^2
 d\pi_0(\phi)\right]^{1/2}.
\end{equation}
Although formulated in terms of $\pi_0$, this distance is independent of the choice of the common dominating measure. The Hellinger distance provides a natural metric for assessing posterior stability and, under suitable integrability conditions, also bounds differences between posterior expectations of polynomial observables.

We are now in a position to state the main well-posedness result, which guarantees that the posterior measure is well defined and depends locally Lipschitz continuously on the observed data.

\begin{theorem}\label{thm: well_posed}
Let $\pi_0$ be a Gaussian prior measure on the Banach space $U$ satisfying $\pi_0(U) = 1$. Let $\Phi: U \times \mathbb{R}^M \to \mathbb{R}$ be the potential defined in \eqref{eq:posterior}.
Then $\pi^d \ll \pi_0$ with Radon--Nikodym derivative
\begin{equation*}
\frac{d\pi^d}{d\pi_0} = \frac{1}{Z} \exp\big(-\Phi(\phi;d)\big)
\end{equation*}
where, for $d$ almost surely,
\begin{equation*}
Z := \int_U \exp\big(-\Phi(\phi;d)\big) \mu_0(d\phi) > 0.
\end{equation*}
Furthermore $\pi^d$ is locally Lipschitz with respect to $d$, in the 
Hellinger distance: for all $d, d'$ with 
$\max\{|d|_\Gamma, |d'|_\Gamma\} < r$, there exists a $C = C(r) > 0$ such that
\begin{equation*}
d_{\mathrm{Hell}}(\pi^d, \pi^{d'}) \leq C |d - d'|_\Gamma.
\end{equation*}
This implies that, for all $f \in L^2_{\pi_0}(U; S)$ for separable Banach space $S$,
\begin{equation*}
\|\mathbb{E}^{\pi^d} f(\phi) - \mathbb{E}^{\pi^{d'}} f(\phi)\|_S 
\leq C |d - d'|.
\end{equation*}
\end{theorem}

\begin{proof}
By the abstract well-posedness framework of Stuart~\cite[Theorem~2.3]{iglesias2016bayesian}, it suffices to verify that the forward potential $\Phi(\phi; d)$ and prior $\pi_0$ satisfy the three conditions specified in Assumption~\ref{assumption-3}:

\medskip
\noindent\textbf{(i) Local boundedness of $\Phi$.} By formula \eqref{eq:uniform-field-bound}, we know 
$\sup_{\phi\in U}\sup_{t\in[0,T]}\|u_\phi^{(s)}(t)\|_{H^2(\Omega)}\le C_{\Omega}M_s$, 
 where $M_s=\left(1+\frac{\mu_{\max}}{\mu_{\min}}\right)B_s
+\frac{\|\sigma\|_{L^\infty(\Omega)}}{\mu_{\min}}A_s$, and $A_s,~B_s$ given by \eqref{eq:A_B}.
Since $H^2(\Omega)\hookrightarrow C(\overline\Omega)$ in two dimensions,
with sobolev embedding constant $C_{\mathrm{emb}}$, we obtain
\begin{equation}\label{eq:guniform}
 \sup_{\phi\in U}|G(\phi)|_\Gamma
 \leq  C_{\mathrm{emb}} \|\Gamma^{-1/2}\|
\left(N_rN_t\sum_{s=1}^{N_s}M_s^2\right)^{1/2}
 =:K_\Gamma.
\end{equation}
Consequently, for $|d|_\Gamma\leq r_d$,
\begin{equation}\label{eq:potential-bound}
 0\leq\Phi(\phi;d)\leq
 (K_\Gamma+r_d)^2=:K_r,\qquad \phi\in U.
\end{equation}

\medskip
\noindent\textbf{(ii) $\pi_0$-almost-sure continuity in $\phi$.} Let
\begin{equation*}
 U_0=\{\phi\in U: \text{meas}_2(\{\boldsymbol{x}\in\Omega:\phi(\boldsymbol{x})=0\})=0\},
\end{equation*}
where $\text{meas}_2(\cdot)$ denotes the two-dimensional Lebesgue measure.
For the Gaussian prior used here, assume that sample paths belong to
$U$ and that $\operatorname{Var}_{\pi_0}[\phi(\boldsymbol{x})]>0$ for
Lebesgue-almost every $\boldsymbol{x}\in\Omega$. Tonelli's theorem then yields
\begin{align*}
 &E_{\pi_0}\!\left[\text{meas}_2(\{\boldsymbol{x}:\phi(\boldsymbol{x})=0\})\right]\\
 &=\int_{U}\int_{\Omega}\mathbf{1}_{\{\phi(\boldsymbol{x})=0\}}d\boldsymbol{x}\,\pi_0(d\phi)\\
 &=\int_\Omega\pi_0\{\phi:\phi(\boldsymbol{x})=0\}\,d\boldsymbol{x}=0.
\end{align*}
By Markov's inequality and straightforward computations, we obtain $\pi_0(U_0)=1$. 

Fix $\phi\in U_0$ and let $\{\phi_k\}_{k\ge1}\subset U$ satisfy
$\phi_k\to\phi$ in $C(\overline\Omega)$. To establish the continuity of
$G$ at $\phi$, we first show that
$\mu_{\phi_k}\to\mu_\phi$ almost everywhere in $\Omega$.

Set
\begin{equation*}
 Z_\phi:=\{\boldsymbol{x}\in\Omega:\phi(\boldsymbol{x})=0\},
 \qquad
 \varepsilon_k:=\|\phi_k-\phi\|_{C(\overline\Omega)}.
\end{equation*}
Since $\phi\in U_0$, we have
$\operatorname{meas}_2(Z_\phi)=0$, while the convergence gives
$\varepsilon_k\to0$. For any $\boldsymbol{x}\in\Omega\setminus Z_\phi$, we have
$\phi(\boldsymbol{x})\ne0$. Hence, there exists $k_0=k_0(\boldsymbol{x})$ such that
\begin{equation*}
 \varepsilon_k<\frac{|\phi(\boldsymbol{x})|}{2},
 \qquad k\ge k_0.
\end{equation*}
If $\phi(\boldsymbol{x})>0$, then, for every $k\ge k_0$,
\begin{equation*}
 \phi_k(\boldsymbol{x})\ge \phi(\boldsymbol{x})-|\phi_k(\boldsymbol{x})-\phi(\boldsymbol{x})|>\frac{\phi(\boldsymbol{x})}{2}>0.
\end{equation*}
Similarly, if $\phi(\boldsymbol{x})<0$, then
\begin{equation*}
 \phi_k(\boldsymbol{x})
 \le \phi(\boldsymbol{x})+|\phi_k(\boldsymbol{x})-\phi(\boldsymbol{x})|
 <\frac{\phi(\boldsymbol{x})}{2}<0.
\end{equation*}
Thus, $\phi_k$ and $\phi$ have the same sign at $x$ for all sufficiently
large $k$. Consequently,
\begin{equation*}
 \mathbf 1_{\{\phi_k\ge0\}}(\boldsymbol{x})
 \longrightarrow
 \mathbf 1_{\{\phi\ge0\}}(\boldsymbol{x}),
 \qquad \boldsymbol{x}\in\Omega\setminus Z_\phi.
\end{equation*}
Since $\operatorname{meas}_2(Z_\phi)=0$, this convergence holds almost
everywhere in $\Omega$. Recalling that
\begin{equation*}
 \mu_\phi=\mu_{\mathrm{out}} +(\mu_{\mathrm{in}}-\mu_{\mathrm{out}})
  \mathbf 1_{\{\phi\ge0\}},
\end{equation*}
we obtain
\begin{equation*}
 \mu_{\phi_k}(\boldsymbol{x})\longrightarrow\mu_\phi(\boldsymbol{x})
 \qquad\text{for a.e. }\boldsymbol{x}\in\Omega.
\end{equation*}

We now show that this convergence carries over to the corresponding wavefields. For a fixed source, we omit the source index for brevity and denote
\begin{equation*}
 u_k := u_{\phi_k},\qquad
 u := u_\phi,\qquad
 w_k := u_k - u.
\end{equation*}
Subtracting the equations satisfied by $u$ and $u_k$, we observe that the error $w_k$ solves
\begin{equation}\label{eq:wavefield-difference}
 \mu_{\phi_k}\partial_{tt}w_k
 +\sigma\partial_t w_k-\Delta w_k
 = h_k,
 \qquad
 h_k := (\mu_\phi - \mu_{\phi_k})\partial_{tt}u,
\end{equation}
equipped with zero initial conditions and the same homogeneous boundary conditions as problem \eqref{eq:forward_problem}.

It remains to analyze the right-hand side $h_k$. For each fixed $t \in [0,T]$, the almost-everywhere convergence of $\mu_{\phi_k}$ ensures that $h_k(\cdot,t) \to 0$ almost everywhere in $\Omega$. Since
\begin{equation*}
 |h_k(\boldsymbol{x},t)| \le |\mu_{\mathrm{in}} - \mu_{\mathrm{out}}|\,|\partial_{tt}u(\boldsymbol{x},t)|
\end{equation*}
with $\partial_{tt}u(\cdot,t) \in L^2(\Omega)$, the Lebesgue dominated convergence theorem implies that
\begin{equation*}
 \|h_k(t)\|_{L^2(\Omega)} \longrightarrow 0
 \qquad\text{for every } t \in [0,T].
\end{equation*}
Furthermore, the bound
\begin{equation*}
 \|h_k(t)\|_{L^2(\Omega)}
 \le |\mu_{\mathrm{in}} - \mu_{\mathrm{out}}|\,\|\partial_{tt}u(t)\|_{L^2(\Omega)}
\end{equation*}
together with the regularity $u \in C^2([0,T]; L^2(\Omega))$ guarantees an integrable majorant in time. Applying the dominated convergence theorem once more, now with respect to $t \in [0,T]$, we obtain
\begin{equation}\label{eq:forcing-difference-convergence}
 \|h_k\|_{L^1(0,T; L^2(\Omega))} \longrightarrow 0 \qquad \text{as } k \to \infty.
\end{equation}

Applying the basic energy estimate \eqref{eq:estimate-grad-u}, which requires only an $L^1(0,T; L^2(\Omega))$ source, combining with the Poincar\'e inequality, we obtain
\begin{equation*}
 \|w_k\|_{C([0,T]; H^1(\Omega))}
 \le \frac{C_\Omega}{\sqrt{\mu_{\min}}} \|h_k\|_{L^1(0,T; L^2(\Omega))}
 \longrightarrow 0 \qquad \text{as } k \to \infty.
\end{equation*}
Meanwhile, the uniform bound in \eqref{eq:uniform-field-bound} yields
$\|w_k\|_{C([0,T]; H^2(\Omega))} \le 2M$.
Fixing $\varepsilon \in (0,1)$ such that $H^{1+\varepsilon}(\Omega) \hookrightarrow C(\overline\Omega)$, the Sobolev interpolation inequality \cite{chandler2015interpolation} gives
\begin{align*}
 \sup_{t\in[0,T]} \|w_k(t)\|_{C(\overline\Omega)}
 &\le C \sup_{t\in[0,T]} \|w_k(t)\|_{H^{1+\varepsilon}(\Omega)} \\
 &\le C \left(\sup_{t\in[0,T]} \|w_k(t)\|_{H^1(\Omega)}\right)^{1-\varepsilon}
        \left(\sup_{t\in[0,T]} \|w_k(t)\|_{H^2(\Omega)}\right)^\varepsilon
 \longrightarrow 0.
\end{align*}
Since this convergence holds for each source, it follows that the forward observation satisfies $G(\phi_k) \to G(\phi)$ in the data space. Consequently, the potential $\Phi(\cdot;d)$ is continuous at every $\phi \in U_0$, and hence $\pi_0$-almost surely.

\medskip
\noindent\textbf{(iii) Local Lipschitz continuity with respect to the data.} For $|d|_\Gamma,|d'|_\Gamma\leq r_d$, the polarization identity yields
\begin{align}\label{eq:potential-lipschitz}
 |\Phi(\phi;d)-\Phi(\phi;d')|
 &=\frac12\left|
   \langle d-d',d+d'-2G(\phi)\rangle_\Gamma\right|\notag\\
 &\leq(K_\Gamma+r_d)|d-d'|_\Gamma
 =:L_r|d-d'|_\Gamma,
\end{align}
where $K_{\Gamma}$ given by in \eqref{eq:guniform}.
This verifies the third condition in Assumption \ref{assumption-3}. 
In particular, $K_{\Gamma}$ and $L_r$ are independent of the particular coefficient field;
their dependence on $\mu_\phi$ is only through $\mu_{\min}$ and
$\mu_{\max}$.

Since all conditions in Assumption \ref{assumption-3} have been verified, the conclusion of the present theorem follows from \cite[Theorem 2.3]{iglesias2016bayesian}.
\end{proof}


\section{Multiscale multimodal Laplace approximation for Bayesian level-set inversion}\label{sec:Inversion_algorithm}

Having established the function-space formulation and the well-posedness of the Bayesian level-set posterior, we now discretize the level-set function and develop a practical multiscale multimodal Laplace approximation on the coefficient space. Unlike standard frequency continuation schemes that track a single minimizer across frequencies, our approach maintains multiple plausible solution branches. At each frequency stage, candidate modes of the scale-dependent objective are identified via multistart optimization. Using local Gauss--Newton approximations, we then merge duplicate modes and generate targeted starting points for the next frequency stage. Finally, the retained modes and their local Gaussian approximations form a Gaussian mixture, with nonnegative weights determined by fitting the unnormalized target posterior in a least-squares sense. This construction captures intra-modal local uncertainty alongside inter-modal structural diversity across distinct shapes.

\subsection{Finite-dimensional representation}
Theorem~\ref{thm: well_posed} establishes the well-posedness of the Bayesian inverse problem in function space. For numerical implementation, we
represent the level-set function in a finite-dimensional space.
Let $\Psi_1,\ldots,\Psi_L\in U\cap H^1(\Omega)$ be linearly
independent basis functions, and let $m_0\in U\cap H^1(\Omega)$
be a prescribed reference function. Define a coefficient-to-fuction map
\begin{equation}\label{eq:finite_levelset_representation}
T_L:\mathbb{R}^L\to U, \qquad
\phi_{\mathbf{q}} = T_L\mathbf{q} = m_0+\sum_{j=1}^{L}q_j\Psi_j.
\end{equation}
A Gaussian coefficient prior
$\nu_{0,L}=\mathcal{N}(\bar {\mathbf{q}}_L,C_{\mathbf{q}})$, with $C_{\mathbf{q}}$ symmetric
positive definite, induces a  Gaussian field prior
\begin{equation*}
\pi_{0,L}=(T_L)_\#\nu_{0,L}.
\end{equation*}
Thus, the finite-dimensional parametrization can be incorporated
into the function-space Bayesian framework through its induced
prior measure.

In this work, we use radial basis functions of the form
\begin{equation}
\Psi_j(\boldsymbol{x})=\psi\bigl(|\boldsymbol{x}-\xi_j|\bigr),
\qquad j=1,\ldots,L,
\label{eq:radial_basis_functions}
\end{equation}
where the centers $\xi_j$ are prescribed, and the radial basis function is fixed to an exponential form, 
$\psi = \exp\left(-\|x\|^2 / (2l^2)\right)$. 

To specify the coefficient prior, we introduce the quadratic
gradient penalty
\begin{equation*}\label{eq:coefficient_prior_energy}
J_{\mathrm{prior},L}(\mathbf{q})=\frac{\mu_{\mathrm{prior}}}{2}
\int_\Omega |\nabla\phi_{\mathbf{q}}(\boldsymbol{x})|^2\,d\boldsymbol{x},\qquad
p_{0,L}(\mathbf{q})\propto \exp\bigl(-J_{\mathrm{prior},L}(\mathbf{q})\bigr),
\end{equation*}
where $\mu_{\mathrm{prior}}>0$ controls the penalty strength. This penalty suppresses
large spatial gradients of the represented level-set function.

Define
\begin{equation*}
(K_L)_{ij}=\int_\Omega\nabla\Psi_i\cdot\nabla\Psi_j\,d\boldsymbol{x},
\qquad (b_L)_i =\int_\Omega\nabla m_0\cdot\nabla\Psi_i\,d\boldsymbol{x}.
\end{equation*}
Then
\begin{equation*}
J_{\mathrm{prior},L}(\mathbf{q})
=\frac{\mu_{\mathrm{prior}}}{2}
\left(\mathbf{q}^\top K_L\mathbf{q}+2b_L^\top \mathbf{q}\right)+c_L,
\end{equation*}
where $c_L$ is independent of $\mathbf{q}$. Assuming that $K_L$ is
positive definite, completing the square yields
\begin{equation*}\label{eq:gradient_prior_covariance}
\nu_{0,L}=\mathcal{N}(\bar{\mathbf{q}}_L,Q_L), \qquad
\bar{\mathbf{q}}_L=-K_L^{-1}b_L, \qquad
C_{\mathbf{q}}=\frac{1}{\mu_{\mathrm{prior}}}K_L^{-1}.
\end{equation*}
In particular, $\bar{\mathbf{q}}_L=0$ when $m_0$ is constant.

We next relate this computational prior to the well-posedness
analysis above. Since the basis functions and $m_0$ belong to
$U$, the induced prior satisfies $\pi_{0,L}(U)=1$.
Writing
\begin{equation*}
\boldsymbol{\Psi}(\boldsymbol{x})=\bigl(\Psi_1(\boldsymbol{x}),\ldots,\Psi_L(\boldsymbol{x})\bigr)^\top,
\end{equation*}
we assume that $\boldsymbol{\Psi}(\boldsymbol{x})\neq0$ for almost every
$\boldsymbol{x}\in\Omega$. Consequently,
\begin{equation*}
\operatorname{Var}_{\pi_{0,L}}[\phi(\boldsymbol{x})]
=\boldsymbol{\Psi}(\boldsymbol{x})^\top C_{\mathbf{q}}\boldsymbol{\Psi}(\boldsymbol{x})>0
\qquad\text{for almost every }\boldsymbol{x}\in\Omega.
\end{equation*}
The same Tonelli argument used above therefore gives
\begin{equation*}
\pi_{0,L}(U_0)=1,\qquad
U_0=\left\{\phi\in U: \text{meas}_2(\{\boldsymbol{x}\in\Omega:\phi(\boldsymbol{x})=0\})=0\right\},
\end{equation*}
where $\text{meas}_2(\cdot)$ denotes Lebesgue measure. Together with the
previously established forward estimates and continuity
properties, this allows the well-posedness argument of
Theorem~\ref{thm: well_posed} to be applied to $\pi_{0,L}$ for each fixed $L$.

Using the same observation model and likelihood as in the
function-space formulation, define
\begin{equation*}
G_L(\mathbf{q})=G(T_L{\mathbf{q}}),\qquad
\Phi_L(\mathbf{q};d)=\Phi(T_L\mathbf{q};d)=\frac12|d-G_L(\mathbf{q})|_\Gamma^2.
\end{equation*}
The coefficient posterior has Lebesgue density
\begin{equation}\label{eq:computational_coefficient_posterior}
p_L(\mathbf{q}\mid d)=\frac{1}{\widetilde Z_L(d)}
\exp\left(-\Phi_L(\mathbf{q};d)-\frac{\mu_{\mathrm{prior}}}{2}
\int_\Omega|\nabla\phi_\mathbf{q}(\boldsymbol{x})|^2\,d\boldsymbol{x}\right),
\end{equation}
where $\widetilde Z_L(d)$ is the normalizing constant.
Denoting this posterior measure by $\nu_L^d$, the corresponding
field posterior is
\begin{equation*}
\pi_L^d=(T_L)_\#\nu_L^d.
\end{equation*}
The numerical inversion is therefore performed in coefficient
space, and posterior information about the level-set function
is obtained through the reconstruction map $T_L$.

\subsection{Target posterior}
\label{BLI}

Let $\mathbf{q} \in \mathbb{R}^L$ denote the coefficient vector defined in \eqref{eq:finite_levelset_representation}, and let $G^\varepsilon_{L,h}: \mathbb{R}^L \to \mathbb{R}^M$ represent the regularized discrete parameter-to-data forward map given by \eqref{eq:app-numerical-forward-map}. This forward operator encapsulates the composition of the regularized level-set parametrization, acoustic wave propagation, and observation sampling at receiver locations; the detailed definitions of the underlying maps are provided in Appendix~\ref{FAP}. Consequently, the discretized posterior density for the coefficients admits the form 
\begin{equation}
    p^\varepsilon_{L,h}(\mathbf{q} \mid d) = \frac{1}{Z} \exp\bigl[-J(\mathbf{q})\bigr],
    \label{eq:posterior_discrete}
\end{equation}
where the total potential $J(\mathbf{q})$ is defined by
\begin{equation*}
    J(\mathbf{q}) := J^\varepsilon_{L,h}(\mathbf{q}; d) = \frac{1}{2\delta^2} \bigl\| G^\varepsilon_{L,h}(\mathbf{q}) - d \bigr\|_2^2 + J_{\mathrm{prior},L,h}(\mathbf{q}).
    \label{eq:cost_functional}
\end{equation*}
Here, the observation noise is modeled as additive, zero-mean Gaussian with covariance $\Gamma = \delta^2 I_M$. Correspondingly, the discrete Gaussian prior potential is given by the quadratic form
\begin{equation}
    J_{\mathrm{prior},L,h}(\mathbf{q}) = \frac{1}{2} (\mathbf{q} - \bar{\mathbf{q}}_{L,h})^\top \Lambda_{\mathrm{prior},h} (\mathbf{q} - \bar{\mathbf{q}}_{L,h}) + \widetilde{c}_h, \quad \Lambda_{\mathrm{prior},h} \succ 0,
    \label{eq:mml-prior-potential}
\end{equation}
where the constant $\widetilde{c}_h$ is independent of $\mathbf{q}$. The prior mean $\bar{\mathbf{q}}_{L,h}$ and precision matrix $\Lambda_{\mathrm{prior},h}$ are derived from the finite-element discretization of the gradient-based penalty; their explicit formulations and the handling of the constant null-space mode are elaborated in Appendix~B.1. The discrete distribution \eqref{eq:posterior_discrete} serves as the target measure for our computational approximation.

\subsection{Branch-preserving frequency continuation for multimodal Laplace approximation}
\label{sec:multiscale-continuation}

We now introduce the main computational construction of this work. The
method couples three ingredients that are usually treated separately:
coarse-to-fine mode discovery through nested frequency filtering, local Gauss--Newton geometry for mode separation and branch propagation, and joint nonnegative least-squares (NNLS) weighting at the final frequency level. Its defining feature is a branch-preserving continuation strategy: instead of following a single minimizer from low to high frequencies, we identify and
propagate several plausible solution branches. Local Laplace models therefore serve two different purposes. At intermediate levels, their covariance matrices determine how each branch is explored and how duplicate solutions are identified; only at the final level do they become components of the posterior approximation.

The full-waveform objective exhibits severe non-convexity and multi-modality. Because simulated waveforms depend nonlinearly on the level-set parameterization, phase mismatch makes the misfit landscape increasingly oscillatory at higher frequencies, leading to severe cycle-skipping issues. While low-frequency data mainly constrain the large-scale location and overall shape of the anomaly, higher frequencies resolve fine boundary details. This complementary behavior naturally suggests a multiscale frequency continuation strategy, incorporating waveform information progressively from low to high frequencies.

Let $0<f_1<\cdots<f_{N_f}$ be increasing frequency levels. For each
source--receiver trace, let $\mathcal F_t$ denote the discrete Fourier
transform in time and let $\mathcal M_k$ be the binary mask that retains the
frequencies up to level $f_k$. The corresponding time-domain filter is
defined by
\begin{equation*}
	P_k=\mathcal F_t^{-1}\mathcal M_k\mathcal F_t.
	\label{eq:alg-frequency-filter}
\end{equation*}
The mask retains conjugate-frequency pairs, so $P_k$ maps real traces to real traces. We choose nested masks and retain all frequencies at the final level;
consequently,
\begin{equation}\label{eq:alg-frequency-projections}
	P_k^\top=P_k,\qquad
	P_k^2=P_k,\qquad
	\operatorname{ran}P_k\subseteq\operatorname{ran}P_{k+1},\qquad
	P_{N_f}=I_M.	
\end{equation}
Thus, applying $P_k$ amounts to transforming each temporal trace to the frequency domain, applying the prescribed mask, and transforming the result back to the time domain.

For independent and identically distributed Gaussian observation errors with covariance $\delta^2I_M$, we define the continuation objective at level $k$ by
\begin{equation}\label{eq:filtered-data-misfit}
	J_k(\mathbf q)
	=\frac{1}{2\delta^2}\left\|
	P_k\bigl(G_{L,h}^{\varepsilon}(\mathbf q)-d\bigr)
	\right\|_2^2+J_{\mathrm{prior},L,h}(\mathbf q).
\end{equation}
Because $P_{N_f}=I_M$, the final objective satisfies $J_{N_f}=J$. The same prior is used at every level; only the frequency content of the data-misfit term changes. The intermediate objectives $J_k$, with $k<N_f$, are auxiliary objectives used to discover and refine solution branches. They are not the posterior ultimately approximated by the method.

At each frequency level, local optimization is performed from multiple
starting points. This produces several candidate modes rather than a single continuation path. After duplicate candidates have been removed, every retained mode is carried to the next level together with perturbations drawn from its local Gaussian model. The resulting continuation therefore preserves alternative geometries that may later develop into distinct modes of the full posterior.

\subsection{Local Laplace models}\label{sec:local-laplace-models}
At frequency level $k$, multistart local minimization of $J_k$ initialized from $\mathcal Q_k^{\mathrm{start}}$ yields a set of candidate minimizers, denoted by $\{\mathbf q_{k,a}^{\mathrm{MAP}}\}_{a=1}^{M_k}$. When $k < N_f$, these points correspond to local modes of the intermediate density proportional to $\exp[-J_k]$. At the final level $k = N_f$, they represent candidate modes of the target posterior. Since $\nabla J_k(\mathbf q_{k,a}^{\mathrm{MAP}}) = \mathbf 0$, expanding $J_k$ in a local Taylor series about each mode yields the quadratic model
\begin{equation*}
    J_k(\mathbf q)
    \approx
    J_k(\mathbf q_{k,a}^{\mathrm{MAP}})
    +\frac{1}{2}
    (\mathbf q-\mathbf q_{k,a}^{\mathrm{MAP}})^\top
    H_{k,a}^{\mathrm{GN}}
    (\mathbf q-\mathbf q_{k,a}^{\mathrm{MAP}}).
    \label{eq:mml-local-quadratic}
\end{equation*}

To avoid evaluating the computationally prohibitive second-order derivatives of the forward wave physics, we adopt the Gauss--Newton approximation to the Hessian. Let the filtered Jacobian evaluated at the mode be
\begin{equation}\label{eq:alg-filtered-jacobian}
    \mathcal J_{k,a}
    :=D\!\left(P_kG_{L,h}^{\varepsilon}\right)
    (\mathbf q_{k,a}^{\mathrm{MAP}})
    = P_kDG_{L,h}^{\varepsilon}(\mathbf q_{k,a}^{\mathrm{MAP}}).
\end{equation}
The corresponding local precision matrix and covariance are given by
\begin{equation}\label{eq:alg-gn-hessian}
    H_{k,a}^{\mathrm{GN}}=
    \delta^{-2}\mathcal J_{k,a}^{\top}\mathcal J_{k,a}
    +\Lambda_{\mathrm{prior},h},\qquad
    C_{k,a}=\left(H_{k,a}^{\mathrm{GN}}\right)^{-1}.
\end{equation}
Consequently, the local posterior surrogate around mode $a$ is represented by the Gaussian distribution
\begin{equation}\label{eq:mml-local-gaussian}
    g_{k,a}(\mathbf q)=\mathcal N\!\left(
    \mathbf q\mid\mathbf q_{k,a}^{\mathrm{MAP}},C_{k,a}\right).
\end{equation}
Here, the mean $\mathbf q_{k,a}^{\mathrm{MAP}}$ represents a candidate anomaly geometry, while the covariance $C_{k,a}$ reflects the local parameter sensitivities in different directions. In intermediate stages ($k < N_f$), $C_{k,a}$ provides a metric to guide initial guesses for the next frequency level; at the final stage ($k = N_f$), it characterizes the local uncertainty around this mode. The coefficient-gradient identities, Gauss--Newton Hessian, and numerical factorization are given in Appendix~\ref{app:mml-derivatives}. 

\subsection{Curvature-aware mode selection and branch propagation}
\label{sec:mode-selection-propagation}
Different starting points may converge to the same local minimum. Such
endpoints correspond to the same mode, so we remove duplicates to ensure that each distinct mode appears only once in the candidate set. To do this, we sort the valid candidates by objective value, starting with the lowest. We retain the first candidate and then check each remaining candidate against the modes already retained. For a candidate $\mathbf q^*$ and a retained mode $r$, define
\begin{equation*}
	d_{k,r}(\mathbf q^*)
	=(\mathbf q^*-\mathbf q_{k,r}^{\mathrm{MAP}})^\top
	H_{k,r}^{\mathrm{GN}}
	(\mathbf q^*-\mathbf q_{k,r}^{\mathrm{MAP}}).
	\label{eq:alg-mode-distance}
\end{equation*}
The candidate is retained if
\begin{equation*}\label{eq:alg-mode-acceptance}
	d_{k,r}(\mathbf q^*)>F_{\chi_L^2}^{-1}(1-\alpha_{\mathrm{mode}})
	\qquad\text{for every previously retained mode }r.
\end{equation*}
Here $F_{\chi_L^2}^{-1}(1-\alpha_{\mathrm{mode}})$ is the
$(1-\alpha_{\mathrm{mode}})$ quantile of the chi-squared distribution with $L$ degrees of freedom. Thus, we accept a candidate only if it lies outside the ellipsoid associated with every retained mode. Because the test uses the local curvature matrix, it accounts for the different sensitivities of the objective to changes in different parameter directions. We stop when all candidates have been checked or when $R_{\max}$ modes have been retained. Denote the number of retained modes at level $k$ by $R_k$.

For $k<N_f$, the next starting set contains every retained centre together
with samples from its local Gaussian model:
\begin{equation}
	\begin{aligned}
		\mathcal Q_{k+1}^{\mathrm{start}}
		&=
		\bigcup_{r=1}^{R_k}
		\left(
		\{\mathbf q_{k,r}^{\mathrm{MAP}}\}
		\cup
		\{\mathbf q_{k,r}^{(\ell)}\}_{\ell=1}^{n_{k,r}}
		\right),
		\\
		\mathbf q_{k,r}^{(\ell)}
		&\sim
		\mathcal N\!\left(
		\mathbf q_{k,r}^{\mathrm{MAP}},C_{k,r}
		\right).
	\end{aligned}
	\label{eq:alg-branch-propagation}
\end{equation}
The sampling budget is allocated as evenly as possible among the $R_k$
retained branches. Keeping the centres preserves the solutions already found, whereas the covariance-guided perturbations explore directions in which each local model remains uncertain. Local optimization at level $k+1$ then refines these starting points after finer-frequency information has been introduced. No posterior weight is assigned at an intermediate level: equal allocation is used for exploration, and the mixture weights are estimated
only from the full target at the final level.

\subsection{Final multimodal posterior and joint weight estimation}
\label{sec:final-multimodal-posterior}

At the final frequency level, let $R=R_{N_f}$ and write
$\mathbf q_r^{\mathrm{MAP}}=\mathbf q_{N_f,r}^{\mathrm{MAP}}$ and
$C_r=C_{N_f,r}$. These centres and covariance matrices are fixed and define
the Gaussian components
\begin{equation}
	g_r(\mathbf q)
	=
	\mathcal N\!\left(
	\mathbf q\mid\mathbf q_r^{\mathrm{MAP}},C_r
	\right),
	\qquad r=1,\ldots,R.
	\label{eq:alg-final-components}
\end{equation}
The posterior approximation is sought in the form
\begin{equation}
	\widehat p_{L,h}^{\varepsilon}(\mathbf q\mid d)
	=
	\sum_{r=1}^{R}w_r g_r(\mathbf q),
	\qquad
	w_r\geq0,
	\qquad
	\sum_{r=1}^{R}w_r=1.
	\label{eq:alg-final-mixture}
\end{equation}

We determine the mixture weights $\{w_r\}_{r=1}^{R}$ jointly by approximating the unnormalized target density with a weighted combination of the retained Gaussian components
\begin{equation*}
	\varphi(\mathbf q)=\exp[-J(\mathbf q)].
	\label{eq:alg-unnormalized-target}
\end{equation*}
We draw the same number of samples from each $g_r$ and collect them as
$\{\mathbf q^{(i)}\}_{i=1}^{N_{\mathrm{LS}}}$. This procedure corresponds to stratified sampling from the balanced proposal:
\begin{equation*}
	\rho(\mathbf q)=\frac{1}{R}\sum_{r=1}^{R}g_r(\mathbf q).
	\label{eq:alg-balanced-proposal}
\end{equation*}
The fitting matrix and target vector are defined as follows:
\begin{equation}
	A_{ir}=g_r(\mathbf q^{(i)}),
	\qquad
	b_i=\varphi(\mathbf q^{(i)}),
	\qquad
	i=1,\ldots,N_{\mathrm{LS}},
	\quad
	r=1,\ldots,R.
	\label{eq:alg-ls-system}
\end{equation}
The nonnegative coefficients $\widetilde{\boldsymbol{w}}$ are obtained by solving a nonnegative least squares problem, and the normalized mixture weights $\{w_r\}_{r=1}^R$ are computed as:
\begin{equation}
	\begin{aligned}
		\widetilde{\boldsymbol w}
		&\in
		\operatorname*{arg\,min}_{\boldsymbol z\geq0}
		\|A\boldsymbol z-\boldsymbol b\|_2^2,
		\\
		w_r&=\widetilde w_r
		/ \displaystyle\sum_{j=1}^{R}\widetilde w_j,
		\qquad r=1,\ldots,R.
	\end{aligned}
	\label{eq:alg-nnls-weights}
\end{equation}

Since each $g_r$ is normalized, normalizing the fitted coefficients yields the probability density in \eqref{eq:alg-final-mixture}. The posterior normalizing constant is not needed: scaling $\varphi$ by a positive constant scales $\widetilde{\boldsymbol{w}}$ equally, leaving the normalized weights unchanged. Here, every target evaluation uses the full objective $J=J_{N_f}$. 

The preceding subsections introduced the main ingredients of the proposed multiscale multimodal Laplace approximation. We now assemble them into a complete computational procedure that combines multistart optimization over nested frequency levels, curvature-aware mode selection, and covariance-guided propagation of the retained solution branches. The intermediate frequency-dependent objectives are used only to discover and refine candidate modes, whereas, at the final frequency level, the retained centres and local covariances define the Gaussian components whose nonnegative weights are determined jointly by least-squares fitting to the full unnormalized posterior. The complete procedure is summarized in Algorithm~\ref{alg:multi-scale-Bayesian-inversion} and illustrated in Figure~\ref{fig:multiscale-multimodal-flowchart}.

\par\addvspace{0.5\baselineskip}
\begin{breakablealgorithm}
{Multiscale multimodal Laplace approximation for Bayesian FWI}
\label{alg:multi-scale-Bayesian-inversion}

\textbf{Input:}
Observations $d$ and noise covariance $\Gamma$;
frequency projections $P_1,\ldots,P_{N_f}$;
initial candidates $\mathcal Q_1^{\mathrm{start}}$;
$B,K,\mathbf m_{0,h},\mu_{\mathrm{prior}}$;
$\mu_{\mathrm{in}},\mu_{\mathrm{out}},H_\varepsilon$;
$\alpha_{\mathrm{mode}},R_{\max},n_{\mathrm{LS}},N_{\mathrm{sample}}$.

\smallskip
\textbf{Output:}
Retained centres and covariances, mixture weights for the
fixed components, the posterior mixture
$\widehat p_{L,h}^{\varepsilon}(\mathbf q\mid y)$,
and posterior field summaries.

\medskip
\noindent
\textbf{1. Initialization.}
Construct $K_{L,h}$, $b_{L,h}$, and
$\Lambda_{\mathrm{prior},h}$ using
\eqref{eq_appendix:KB_lambda}, and let
$\Gamma_k=P_k\Gamma P_k^\top$.

\medskip
\noindent
\textbf{2. For each frequency level
$k=1,\ldots,N_f$:}

\smallskip
\hspace*{1em}%
Use the objective $J_k$ defined in
Appendix~\ref{app:mml-derivatives}, keep the prior
unchanged, and partition
$\mathcal Q_k^{\mathrm{start}}$ into candidate blocks.
Perform the following steps.

\smallskip
\begin{list}{}{
  \setlength{\leftmargin}{2.5em}
  \setlength{\labelwidth}{1.8em}
  \setlength{\labelsep}{0.5em}
  \setlength{\itemsep}{0.4em}
  \setlength{\parsep}{0pt}
  \setlength{\topsep}{0pt}
}

\item[(a)]
\textbf{Local MAP estimation.}
Optimize $J_k$ from each candidate in
$\mathcal Q_k^{\mathrm{start}}$
to obtain local MAP estimates.

\item[(b)]
\textbf{Local Gaussian approximation.}
Compute the Jacobians \eqref{eq:total-gradient}
and Gauss--Newton Hessians
\eqref{eq:alg-local-covariance}, and construct
local Gaussian components
\eqref{eq:mml-local-gaussian}.

\item[(c)]
\textbf{Mode selection.}
Remove duplicate modes using the Mahalanobis
criterion and retain at most $R_{\max}$ components.

\item[(d)]
\textbf{Multiscale propagation.}
If $k<N_f$, generate
$\mathcal Q_{k+1}^{\mathrm{start}}$
(see \eqref{eq:alg-branch-propagation})
from the retained modes and Gaussian perturbations.

\end{list}

\smallskip
\hspace*{1em}\textbf{End for.}

\medskip
\noindent
\textbf{3. Mixture weight estimation.}
Draw $n_{\mathrm{LS}}$ fitting samples per component
and determine the mixture weights
\eqref{eq:alg-nnls-weights}.

\medskip
\noindent
\textbf{4. Posterior approximation.}
Construct
$\widehat p_{L,h}^{\varepsilon}
(\mathbf q\mid\mathbf y)$
and compute posterior field summaries using
$N_{\mathrm{sample}}$ samples and reconstruct the corresponding
fields using \eqref{eq:alg-field-samples}.
\end{breakablealgorithm}
\par\addvspace{0.5\baselineskip}

\begin{figure}[H]
\begin{center}
  \includegraphics[width=0.55\textwidth]{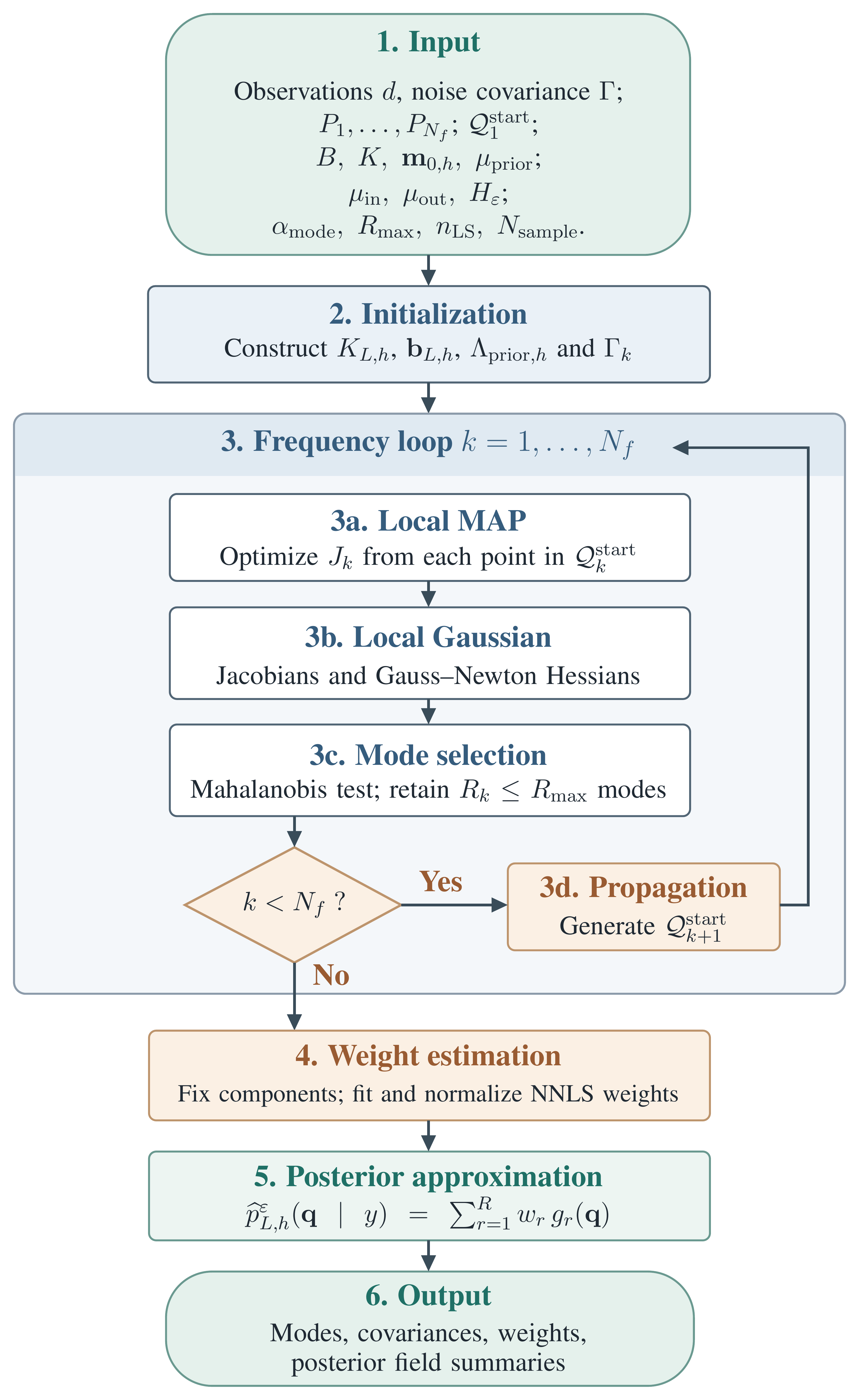}
  \caption{Flowchart of Multiscale multimodal Laplace approximation for Bayesian FWI.}
  \label{fig:multiscale-multimodal-flowchart}
\end{center}
\end{figure}

\section{Numerical experiments}\label{sec:Numerical experiments}
We evaluate the proposed multiscale multimodal Laplace method for Bayesian level-set full-waveform inversion using four two-dimensional examples. The first is a two-parameter test illustrating the need to retain multiple posterior modes. The remaining examples consider a single nonconvex inclusion, three separated inclusions of different shapes, and a geophysically motivated salt-body geometry derived from the BP2004 model. Comparisons are made with the standard single-component Laplace approximation and gPCN-based Markov chain Monte Carlo (MCMC). Reconstruction accuracy and computational cost are assessed using posterior mean errors, Hausdorff boundary distance, boundary concentration, node-based boundary coverage, and elapsed time. Pointwise posterior standard deviations are also reported to characterize reconstruction uncertainty.

\subsection{Experimental settings} \label{subsec:experimental-settings}

The computational domain is $\Omega=[0,2200]\times[0,1100]\,\mathrm{m}$, where $x\in[0,2200]~\mathrm{m}$ and $y\in[0,1100]~\mathrm{m}$ denote the
horizontal and vertical spatial coordinates, respectively. The region of interest is
$\Omega_{\mathrm{ROI}}=[100,2100]\times[100,1100]\,\mathrm{m}$.
A $100\,\mathrm{m}$-thick absorbing boundary layer surrounds the left, right, and lower sides of the region of interest, while its upper side remains unpadded. The domain is discretized on a uniform grid with $\Delta x=\Delta y=10\,\mathrm{m}$.

The forward wave problem is solved by an explicit time-domain finite-difference method with fourth-order spatial differences. In time, we use a centred second difference for $u_{tt}$ and a backward first difference for the damping term. Data-misfit gradients are computed by a discrete adjoint-state method based on the same forward discretization, with reverse-time propagation driven by the frequency-filtered, noise-weighted observation residuals. Forward and adjoint solves are batched across candidate models and source experiments. For related time-domain level-set FWI formulations and numerical treatments, see~\cite{de2024single}. The quadratic level-set prior relies on a finite-element stiffness matrix $K$ assembled from triangular elements, which is separate from the finite-difference wave discretization. Synthetic pressure data are generated by sampling the wavefield of the prescribed true medium at the receiver locations.

The velocity models feature a $1000\,\mathrm{m/s}$ background and $1500\,\mathrm{m/s}$ inclusions. 

Reconstruction accuracy is evaluated using four metrics: relative errors of the posterior mean fields, the Hausdorff distance \cite{Burago2001} between the reconstructed and exact interfaces, the concentration of posterior variance near the exact interface, and the coverage of the exact interface by the posterior uncertainty band. Their definitions and discrete implementations are given in Appendix~\ref{app:evaluation-metrics}.

All experiments were conducted on a workstation running Ubuntu 24.04.3 LTS
via Windows Subsystem for Linux 2 (WSL2) on Windows 11 Pro (64-bit).
The workstation was equipped with an Intel Core i9-14900K CPU and an
NVIDIA GeForce RTX 5090 GPU with 32 GB of memory. The implementations
used Python 3.12.3 and PyTorch 2.8.0 with CUDA 12.8.

\begin{example}[Two-mode test]\label{example0}\rm{

We first consider a two-parameter test to examine whether separated
posterior modes can be retained. The two modes correspond to a single
inclusion located at either of two horizontally separated positions. Nine Gaussian sources and forty receivers are used. The sources have a Ricker waveform with dominant frequency $4\,\mathrm{Hz}$. Data are recorded over $[0,3]\,\mathrm{s}$ with $\Delta t=0.002\,\mathrm{s}$, giving 1501 time levels.

The level set function is parameterized by
\begin{equation*}
\phi(\boldsymbol{x};\mathbf q)
=q_1\psi_1(\boldsymbol{x})+q_2\psi_2(\boldsymbol{x})
-\tau,\qquad\psi_j(\boldsymbol{x})
=\exp\!\left(
-\frac{\|\boldsymbol{x}-\mathbf c_j\|_2^2}{2\ell^2}\right),
\end{equation*}
where $\mathbf q=(q_1,q_2)^{\top},~\mathbf c_1=(700,500)\,\mathrm{m},~
\mathbf c_2=(1500,500)\,\mathrm{m},$
and $\ell=180\,\mathrm{m}$ and $\tau=0.13$. The data-generating coefficient vector is $\mathbf q_{\mathrm{true}}=(0.2,0)^{\top}$,
while its reflected counterpart is
$\mathbf q_{\mathrm{mirror}}=(0,0.2)^{\top}$.

Synthetic observations are generated from $\mathbf q_{\mathrm{true}}$ and
perturbed by independent Gaussian noise with standard deviation equal to $2\%$. We use the prior
$\mathcal N(0,25I_2)$. To produce a symmetric two-mode target, the original
data-misfit potential $\Phi$ is replaced by
\begin{equation}
\Phi_{\mathrm{sym}}(\mathbf q)
=
\min\bigl\{
\Phi(q_1,q_2),\,
\Phi(q_2,q_1)
\bigr\},
\label{eq:symmetric-two-mode-potential}
\end{equation}
with the prior term unchanged. The resulting posterior has two
symmetry-related modes associated with the two candidate inclusion
locations.

For the standard Laplace approximation, the local optimization starts
from $(0.02,0)^{\top}$ and uses 100 Adam iterations with learning rate
$0.02$. The multiscale multimodal calculation proceeds over the frequency
bands $[0,2]\,\mathrm{Hz}$ and $[0,6]\,\mathrm{Hz}$, starting from
\begin{equation*}
\mathcal Q_1^{\mathrm{start}}
=
\bigl\{
(0.020,0.002)^{\top},\
(0.002,0.020)^{\top},\
(0.006,0.040)^{\top},\
(0.040,0.006)^{\top}
\bigr\}.
\end{equation*}
Each start is updated by 100 Adam iterations per scale with learning rate $0.01$. The coefficient search is restricted to $[-0.1,0.3]^2$, the mode-selection threshold is $0.05$, and at most four modes are retained.

\begin{table}[htbp]
\centering
\caption{Computed Gaussian centres and weights for the two-mode test.
Values are rounded.}
\label{tab:two-mode-centres}
\begin{tabular}{@{}llcc@{}}
\toprule
Method
& Frequency band
& Centre $\mathbf q=(q_{1},q_{2})^{\top}$
& Weight $w_r$ \\
\midrule
Standard Laplace
& --
& $(0.200297,-0.000139)^{\top}$
& $1.0000$ \\
\midrule
\multirow{4}{*}{\shortstack{Multiscale multimodal\\Laplace}}
& \multirow{2}{*}{$0$--$2\,\mathrm{Hz}$}
& $(0.203144,0.006229)^{\top}$
& $0.5000$ \\
&
& $(0.006229,0.203144)^{\top}$
& $0.5000$ \\
\cmidrule(lr){2-4}
& \multirow{2}{*}{$0$--$6\,\mathrm{Hz}$}
& $(0.200997,-0.000002)^{\top}$
& $0.5000$ \\
&
& $(-0.000002,0.200997)^{\top}$
& $0.5000$ \\
\bottomrule
\end{tabular}
\end{table}

\begin{figure}[H]
\centering
\makebox[\textwidth][c]{%
\begin{minipage}[t]{0.45\textwidth}
\centering
\includegraphics[width=\linewidth]
{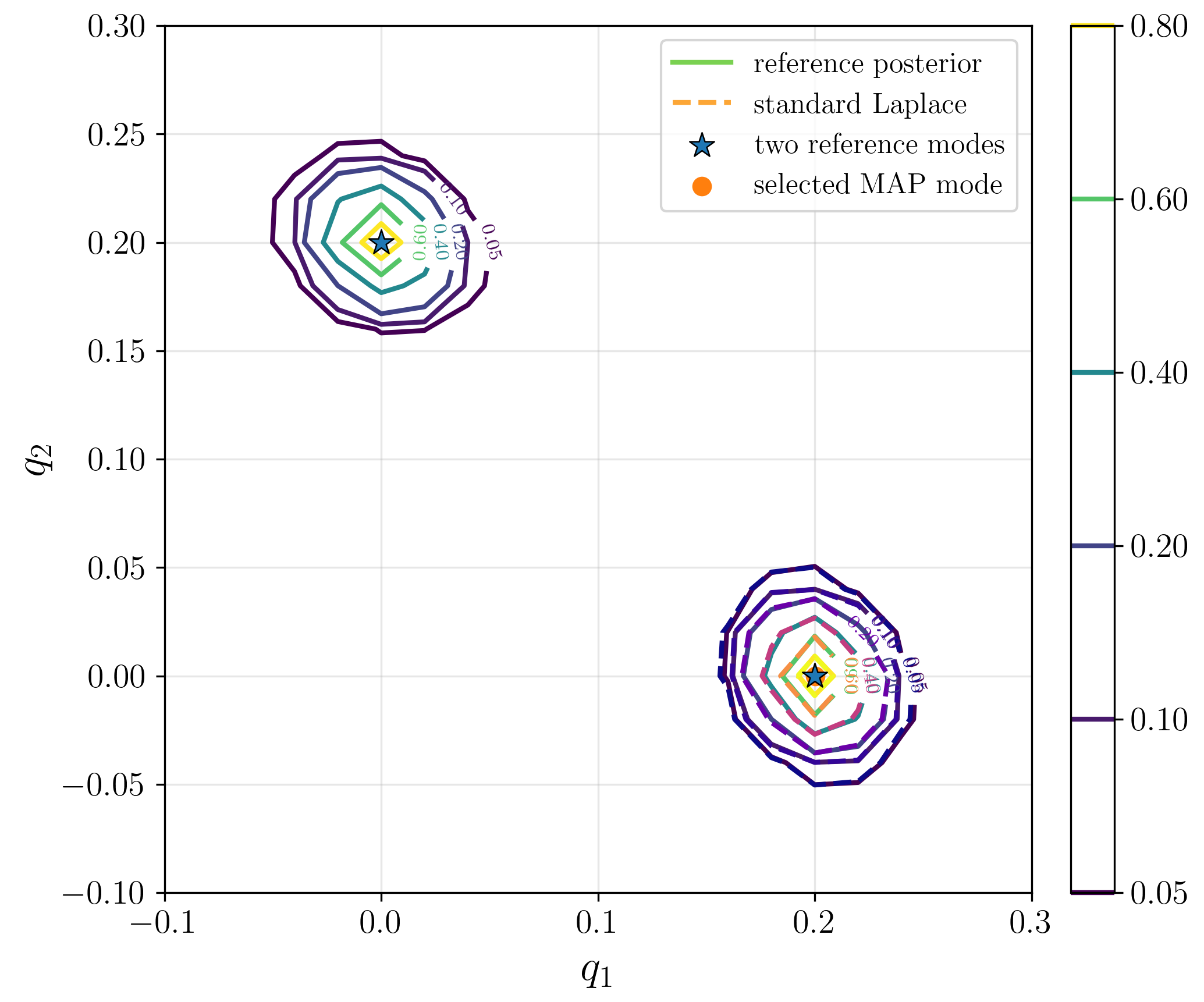}
\par\smallskip
{\small (a) Standard Laplace.\par}
\end{minipage}%
\hspace{0.01\textwidth}%
\begin{minipage}[t]{0.45\textwidth}
\centering
\includegraphics[width=\linewidth]
{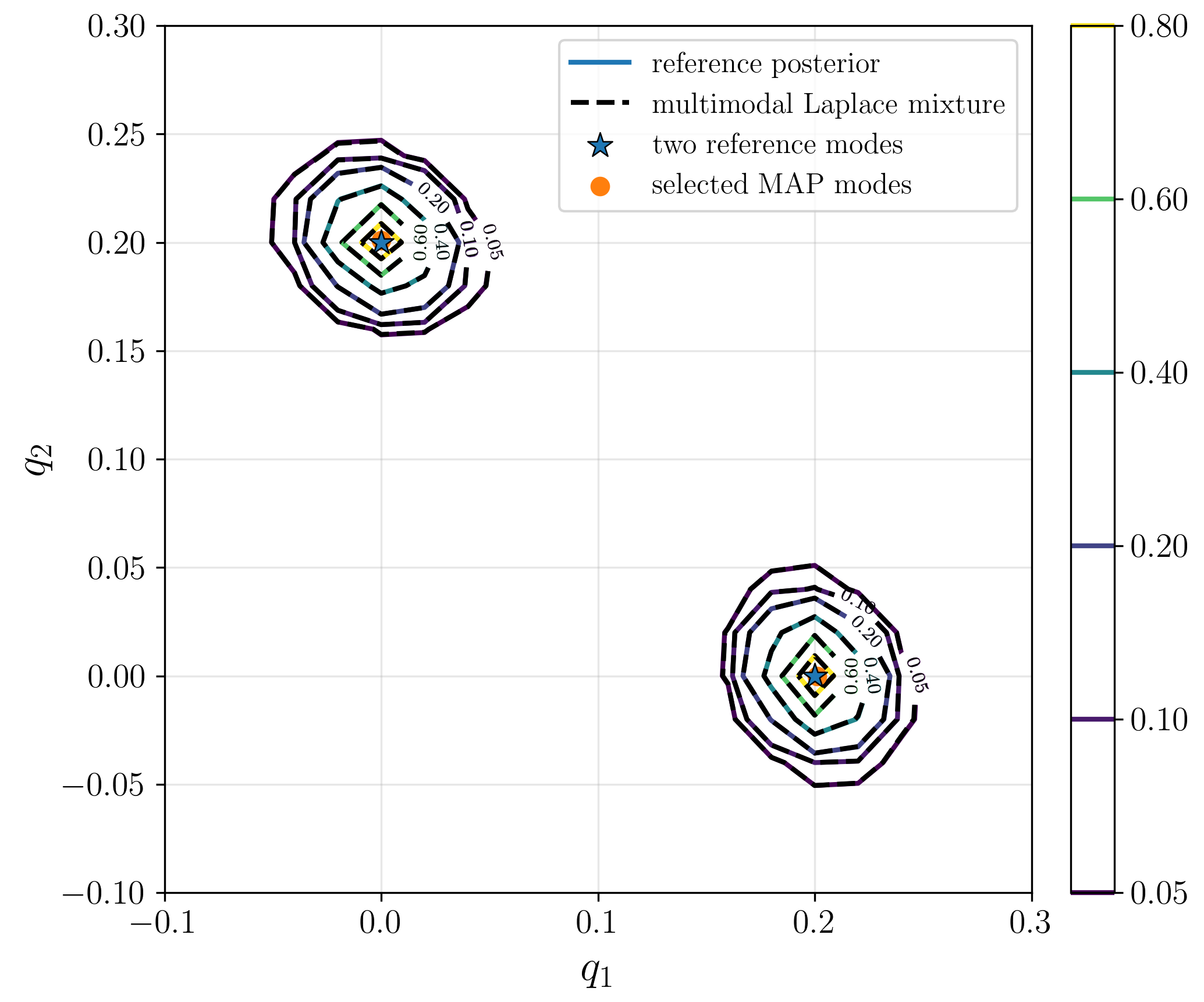}
\par\smallskip
{\small (b) Multiscale multimodal Laplace.\par}
\end{minipage}%
}
\caption{Peak-scaled contours of the reference posterior and the Laplace approximations for the two-mode test: (a) standard single-mode Laplace approximation; (b) multiscale multimodal Laplace approximation.}
\label{fig:two-mode-posteriors}
\end{figure}

Table~\ref{tab:two-mode-centres} and Figure~\ref{fig:two-mode-posteriors} jointly illustrate the limitation of the standard Laplace approximation. The standard method identifies a
centre close to $(0.2,0)^{\top}$, and its Gaussian approximation accurately follows the local posterior contours around this mode. 
However, it completely misses the second mode near 
$(0,0.2)^{\top}$, which is also present in the reference posterior. As a result, even an accurate local MAP approximation cannot fully capture the bimodal posterior.

The multiscale multimodal method retains the symmetric pair at both
frequency scales. At the final scale, the two centres are close to
$(0.2,0)^{\top}$ and $(0,0.2)^{\top}$, each with weight $0.5$, and the resulting mixture follows the reference contours around both peaks. These results show that a single Gaussian Laplace approximation fails to capture well-separated posterior modes, but the multimodal construction keeps both possible locations.}
\end{example}

\begin{example}\label{example1}\rm{

In this example, we reconstruct a single nonconvex
inclusion from time-dependent pressure measurements. 
Let $D_{\mathrm{true}} \subset \Omega$ denote the region enclosed by the smooth parameterised curve $\partial D_{\mathrm{true}} = \{(x(\theta), y(\theta)) : \theta \in [0, 2\pi)\}$, where
\begin{equation*}
\begin{aligned}
    x(\theta) &= 400 + 220 \left( \frac{3}{4}\cos\theta + \frac{1}{2}\cos(2\theta) - \frac{1}{480} \right), \\
    y(\theta) &= 500 - 220 \sin\theta.
\end{aligned}
\end{equation*}
The exact velocity and squared slowness are
\begin{equation}\label{eq:ex1-exact-medium}
v_{\mathrm{exact}}(\boldsymbol{x})=
\begin{cases}1500\,\mathrm{m/s},
& \boldsymbol{x}\in D_{\mathrm{true}},\\
1000\,\mathrm{m/s},
& \boldsymbol{x}\notin D_{\mathrm{true}},
\end{cases}
\qquad
\mu_{\mathrm{exact}}(\boldsymbol{x})=
v_{\mathrm{exact}}(\boldsymbol{x})^{-2}.
\end{equation}
The two phase velocities remain fixed during inversion; only the geometry of the inclusion is inferred.

The acquisition system consists of nine sources and forty receivers,
located at
\begin{equation*}
\begin{aligned}
\boldsymbol{x}_s
&=
(250+200(s-1),\,0)\,\mathrm{m},
&& s=1,\ldots,9,\\
\boldsymbol{x}_r
&=
(250+40(r-1),\,200)\,\mathrm{m},
&& r=1,\ldots,40.
\end{aligned}
\end{equation*}
The acquisition geometry and damping profile are shown in
Figure~\ref{fig:ex1-acquisition}. Each source has a Gaussian spatial
profile with standard deviation $20\,\mathrm{m}$ and is normalized to have
total strength $A_s=1$. Its temporal waveform is a Ricker wavelet with
dominant frequency $4\,\mathrm{Hz}$.

The pressure data are recorded over $[0,3]\,\mathrm{s}$ with
$\Delta t=0.002\,\mathrm{s}$, giving 1501 time levels. Independent
Gaussian noise with standard deviation
$\delta_{\mathrm{obs}}=10^{-4}$ is added to each measurement. The resulting observation tensor has dimensions $9\times40\times1501$, corresponding to source, receiver, and time indices, respectively.

\begin{figure}[H]
\centering

\begin{minipage}[t]{0.435\textwidth}
\centering
\includegraphics[width=\linewidth]{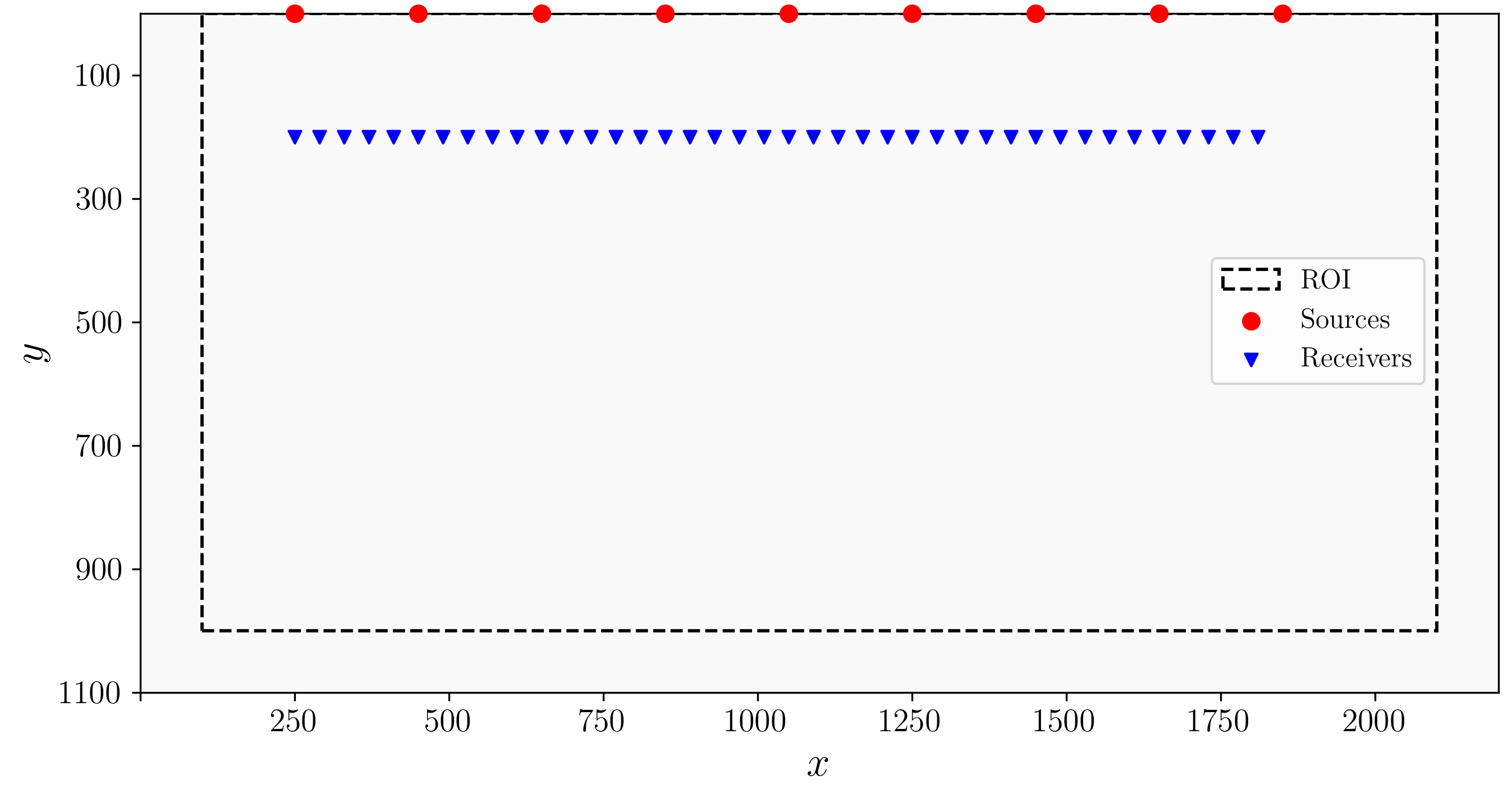}
\par\smallskip{\small (a) Source--receiver geometry and ROI.\par}
\end{minipage}
\hspace{0.01\textwidth}%
\begin{minipage}[t]{0.49\textwidth}
\centering
\includegraphics[width=\linewidth]{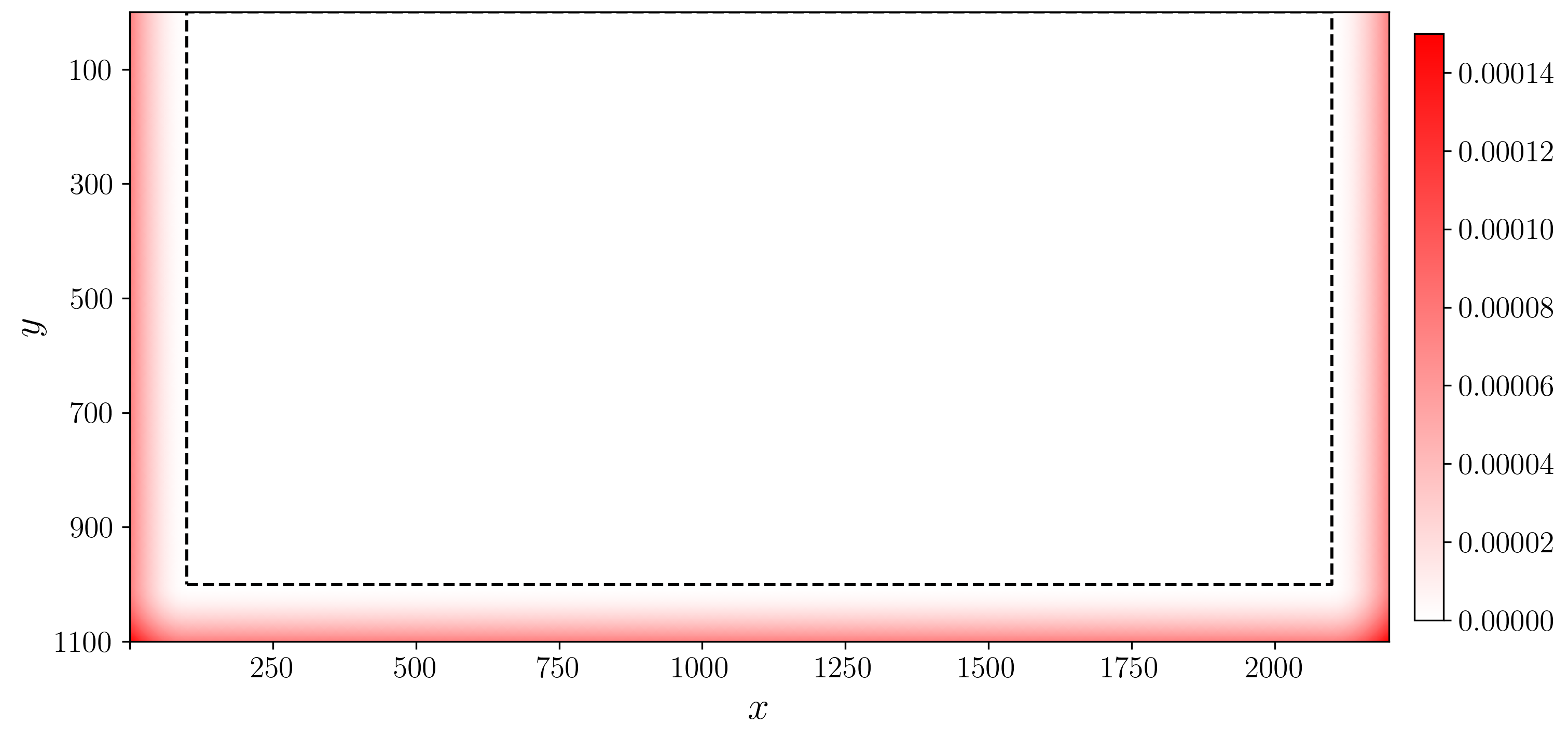}
\par\smallskip{\small (b) Absorbing-layer damping profile.\par}
\end{minipage}

\caption{Numerical setup for the nonconvex-inclusion experiment (Example~\ref{example1}): (a) source and receiver locations, with the dashed line indicating the ROI; (b) spatial distribution of the absorbing-layer damping coefficient $\sigma(\boldsymbol{x})$.}
\label{fig:ex1-acquisition}
\end{figure}

The level-set function is represented by $400$ Gaussian RBFs whose centres are arranged on a $20\times20$ grid over the ROI, together with one ROI-supported constant basis function. Thus, the coefficient dimension is $m=401$. All basis functions are set to zero outside the ROI. We choose
$\ell=70$ in \eqref{eq:radial_basis_functions}, $\mathrm{m}=-40$ in \eqref{eq:finite_levelset_representation}, $\varepsilon=20$ in \eqref{eq:alg-regularized-heaviside}, $\mu_{\mathrm{prior}}=0.08$ in \eqref{eq:computational_coefficient_posterior}.
Three initial coefficient vectors are obtained by ridge least-squares fits to circular signed-distance functions. Their induced initial geometries, shown in Figure~\ref{fig:ex1-initial}, are used at the first frequency level.

The inversion proceeds through three cumulative frequency bands,
$[0,2]$, $[0,4]$, and $[0,6]\,\mathrm{Hz}$. Adam is run for $100$, $100$, and $200$ epochs at the respective frequency levels, with learning rate $7.0$. At most four local Gaussian components are retained at each level using a retention threshold of $\alpha_{\rm{mode}}=0.05$.

\begin{figure}[H]
\centering
\begin{minipage}[t]{0.33\textwidth}
\centering
\includegraphics[width=\linewidth]{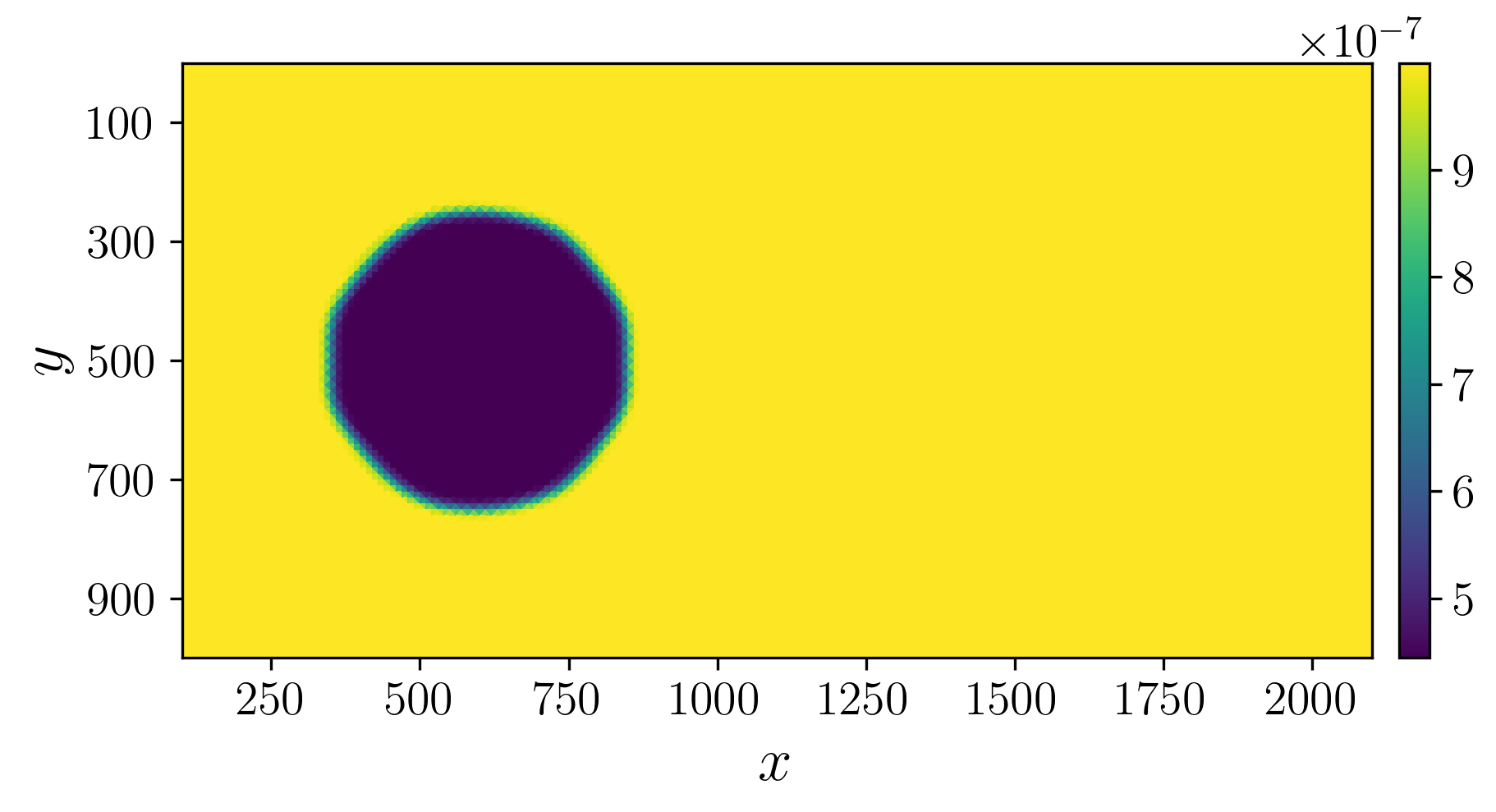}
\par\smallskip{\small (a) First initial model.\par}
\end{minipage}\hfill
\begin{minipage}[t]{0.33\textwidth}
\centering
\includegraphics[width=\linewidth]{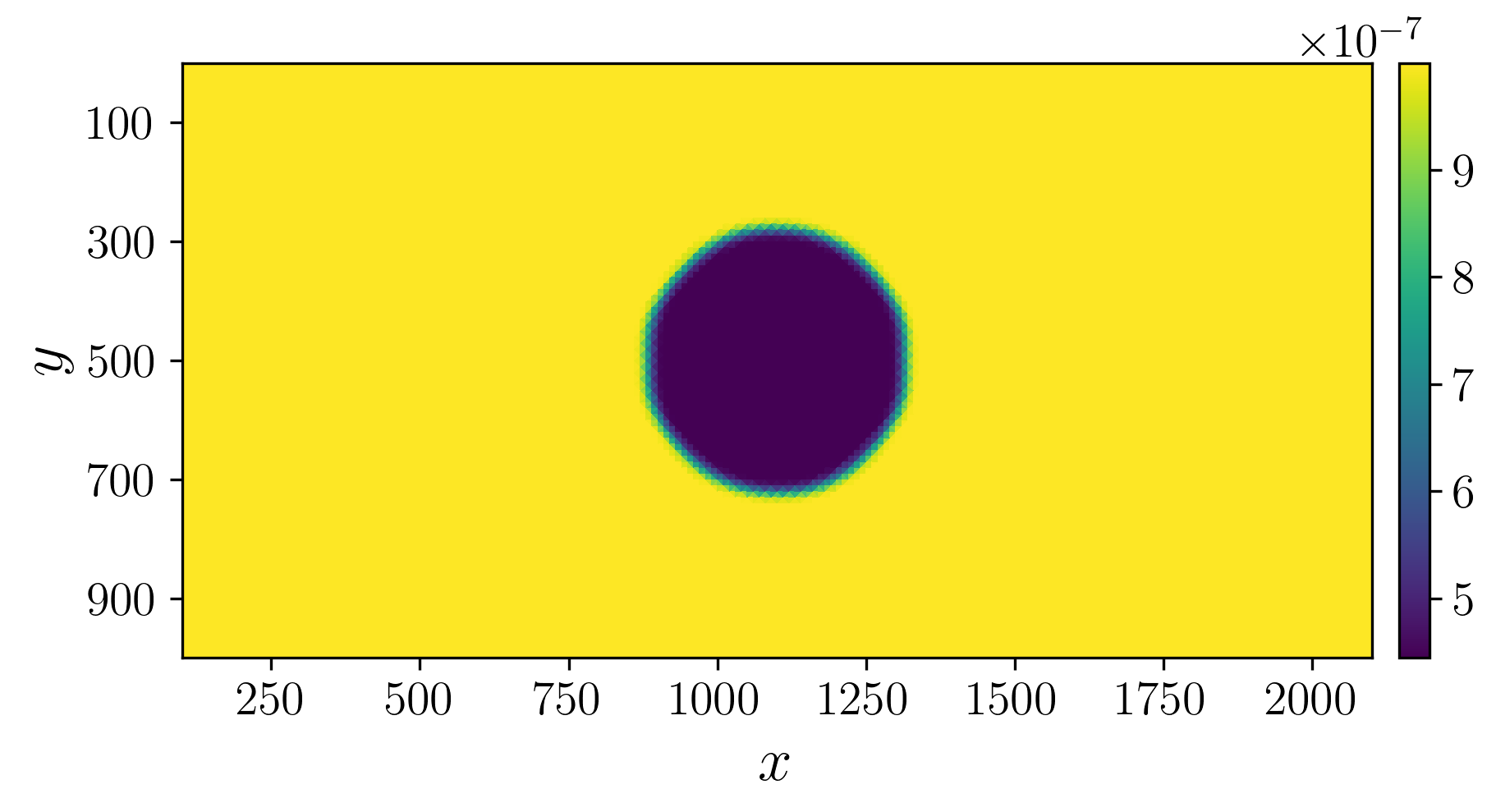}
\par\smallskip{\small (b) Second initial model.\par}
\end{minipage}\hfill
\begin{minipage}[t]{0.33\textwidth}
\centering
\includegraphics[width=\linewidth]{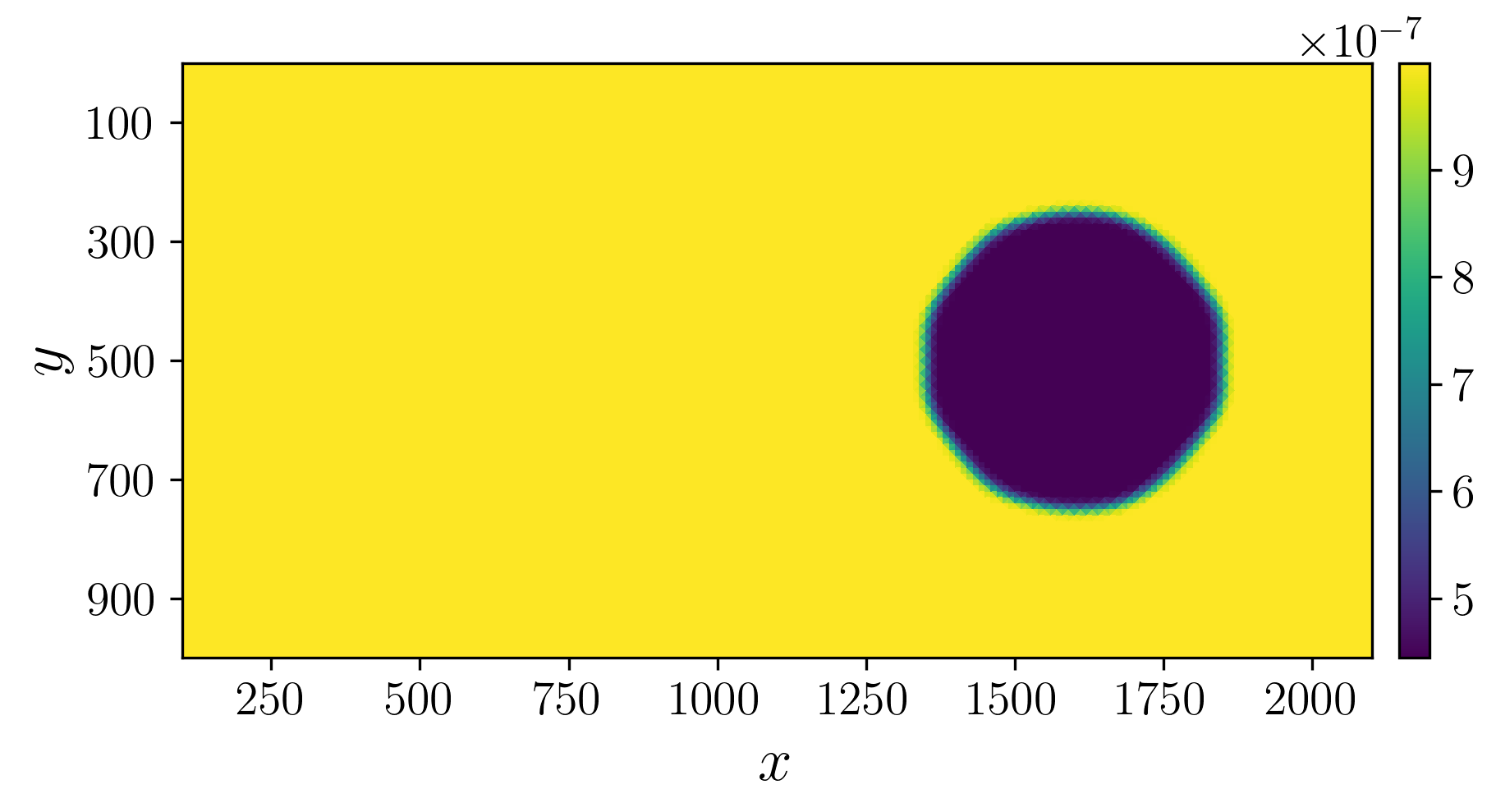}
\par\smallskip{\small (c) Third initial model.\par}
\end{minipage}
\caption{Three initial squared-slowness fields fitted for Example~\ref{example1}, shown over the region of interest (ROI). The target circles have centres $(600,500)$, $(1100,500)$, and $(1600,500)$, with radii $250$, $220$, and $250$, respectively. All coordinates and radii are given in metres.}
\label{fig:ex1-initial}
\end{figure}

At the highest frequency level, four Gaussian components are retained. The weights are
$(\pi_1,\pi_2,\pi_3,\pi_4)\approx
(0.758956,\ 0.126133,\ 0.076229,\ 0.038681).$
Figure~\ref{fig:ex1-reconstruction} compares the exact fields with the best retained MAP reconstruction and the posterior mixture means. Here, the best MAP is the retained centre with the smallest full-source objective value at the final frequency level.

Despite the circular initializations, the best retained MAP recovers the location and principal nonconvex features of the inclusion, including the pronounced indentation on the left and the rounded convex boundary on the right. The mixture mean fields preserve the same geometry, which is consistent with the dominant weight $\pi_1\approx0.759$. Furthermore, the reconstructed interfaces (both MAP and mixture mean) are slightly smoother than the exact boundary due to the smoothed level-set representation.

\begin{figure}[H]
\centering

\begin{minipage}[t]{0.33\textwidth}
\centering
\includegraphics[width=\linewidth]{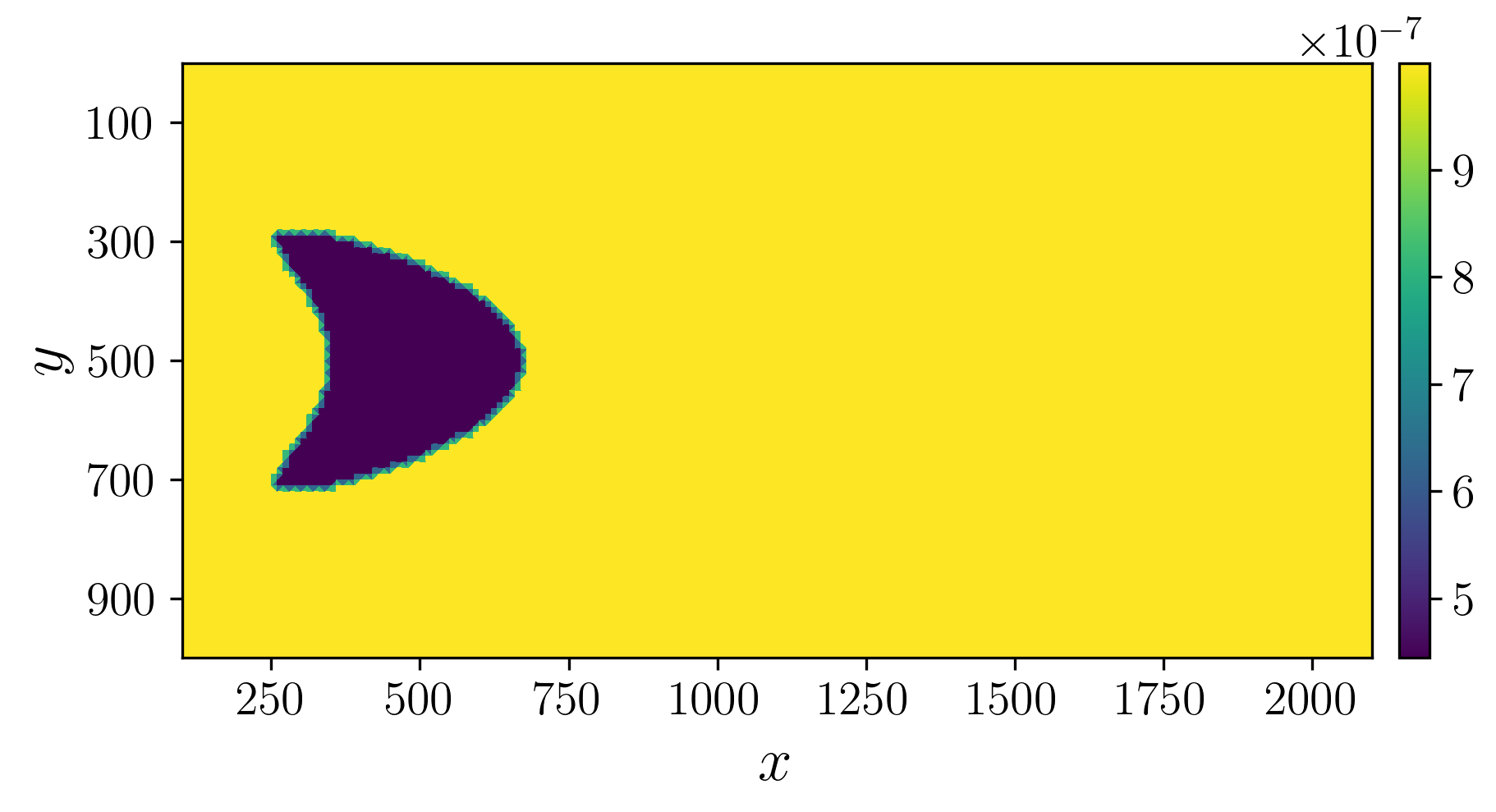}
\par\smallskip{\small (a) Exact squared slowness.\par}
\end{minipage}\hfill
\begin{minipage}[t]{0.33\textwidth}
\centering
\includegraphics[width=\linewidth]{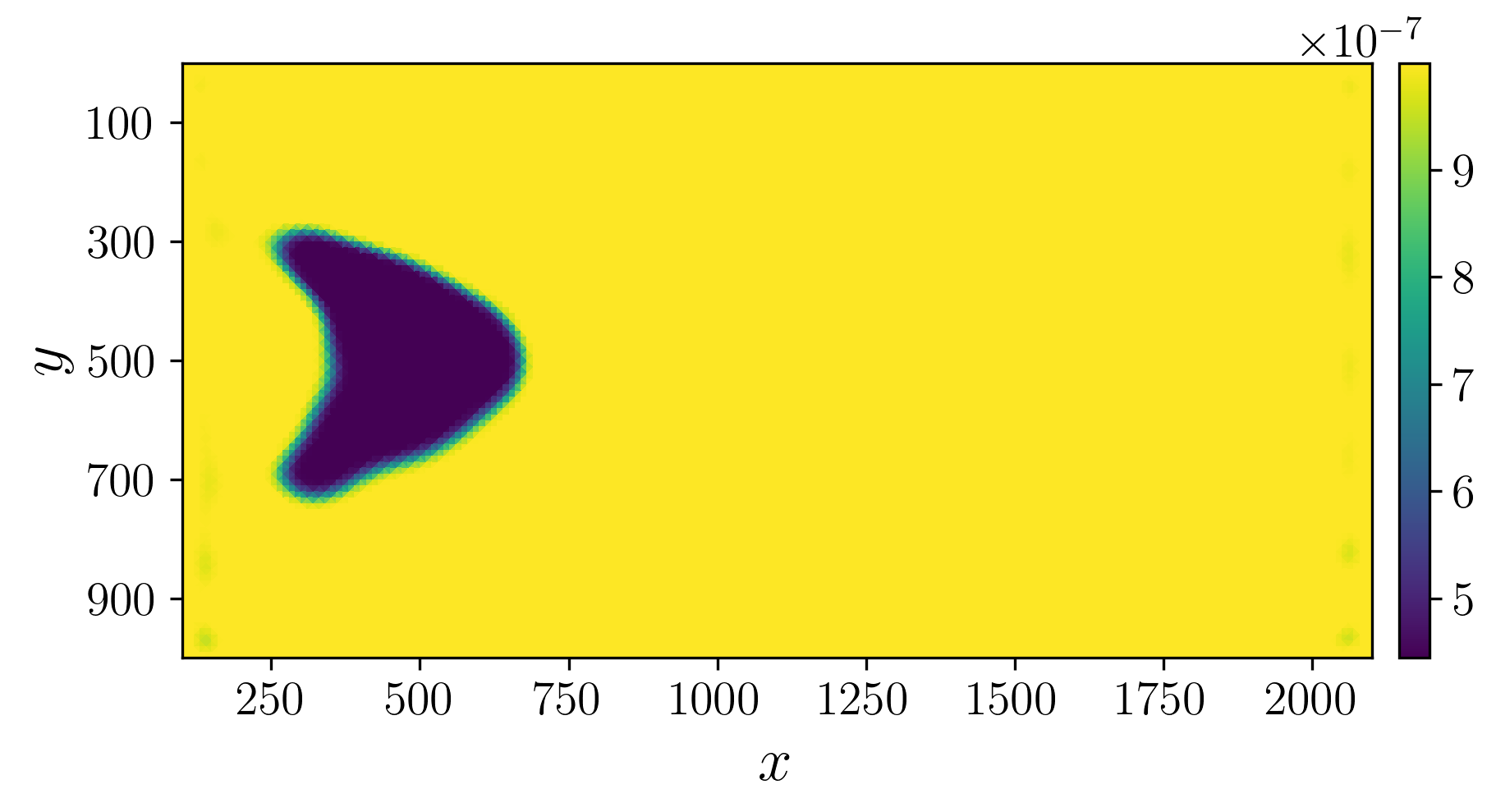}
\par\smallskip{\small (b) Best retained MAP squared slowness.\par}
\end{minipage}\hfill
\begin{minipage}[t]{0.33\textwidth}
\centering
\includegraphics[width=\linewidth]{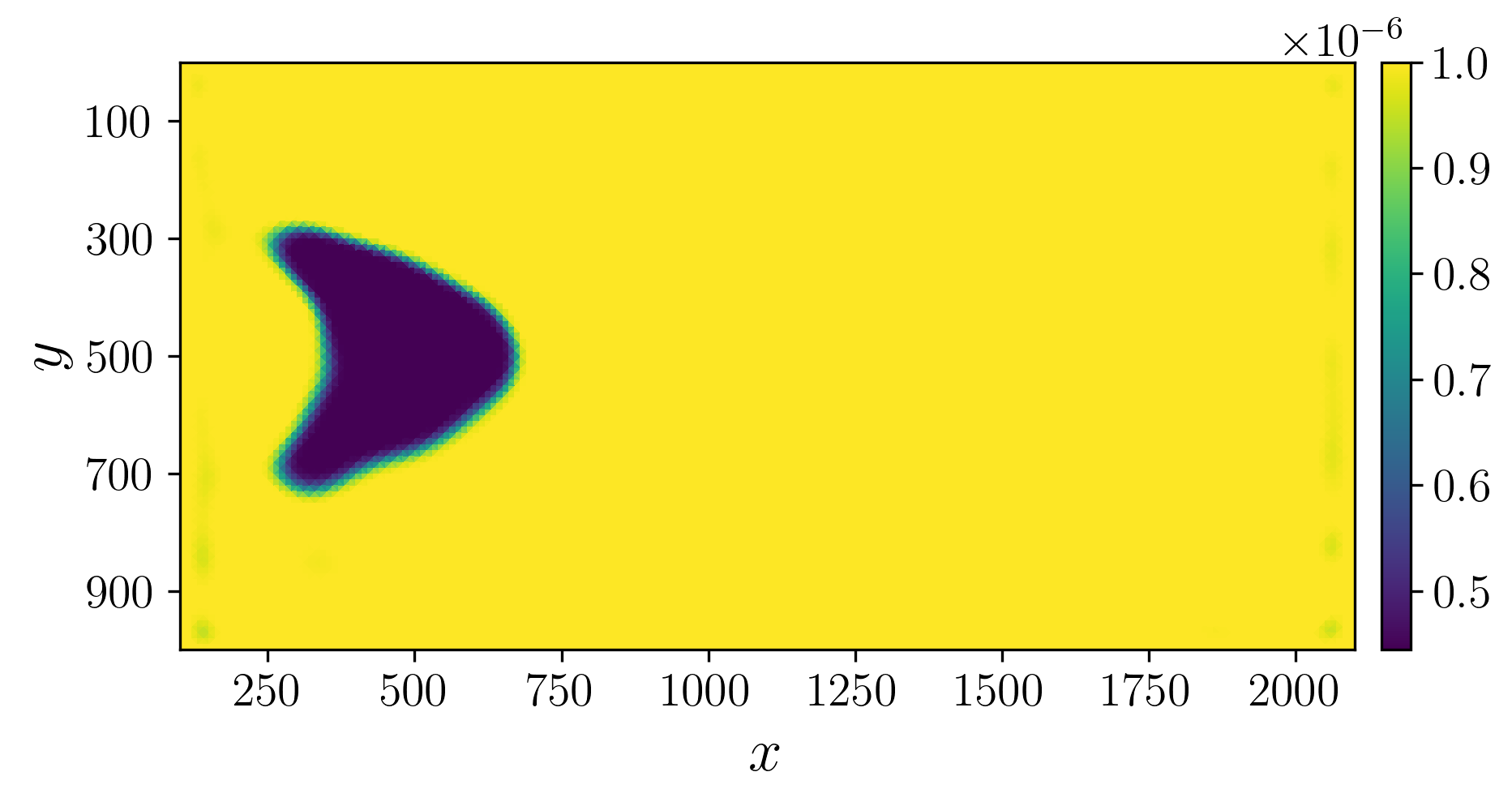}
\par\smallskip{\small (c) Mixture mean squared slowness.\par}
\end{minipage}

\par\medskip

\begin{minipage}[t]{0.33\textwidth}
\centering
\includegraphics[width=\linewidth]{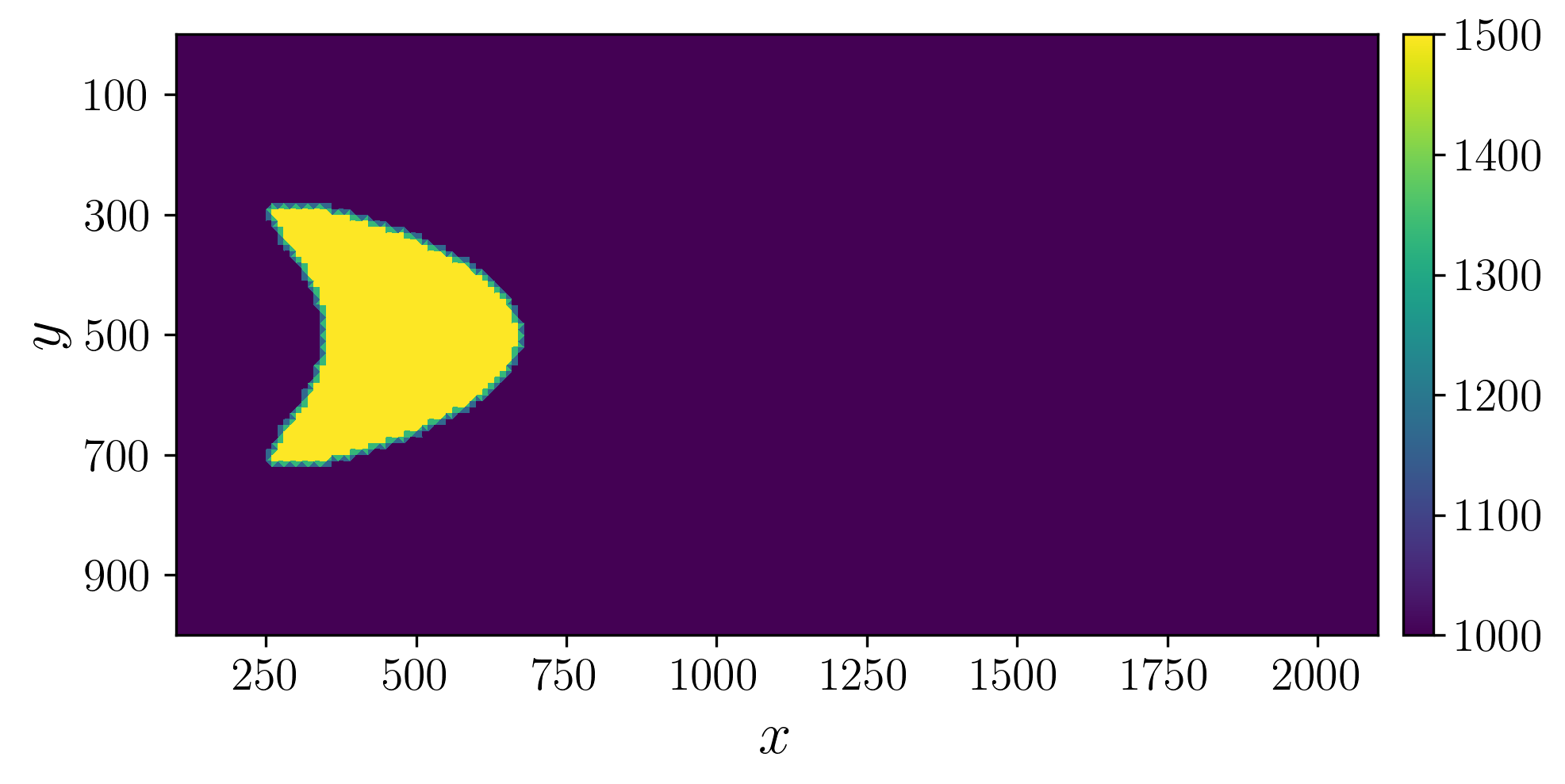}
\par\smallskip{\small (d) Exact velocity.\par}
\end{minipage}\hfill
\begin{minipage}[t]{0.33\textwidth}
\centering
\includegraphics[width=\linewidth]{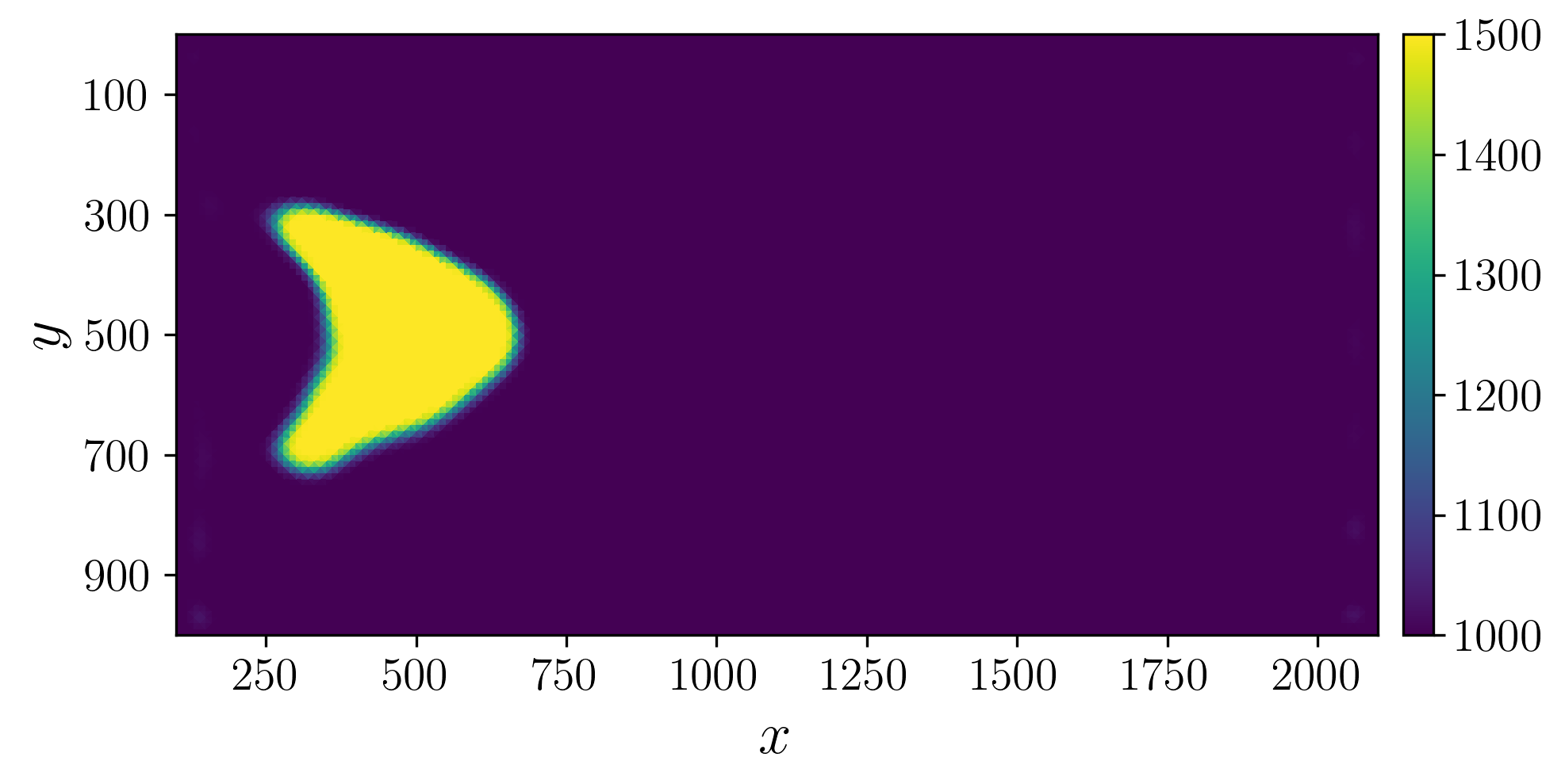}
\par\smallskip{\small (e) Best retained MAP velocity.\par}
\end{minipage}\hfill
\begin{minipage}[t]{0.33\textwidth}
\centering
\includegraphics[width=\linewidth]{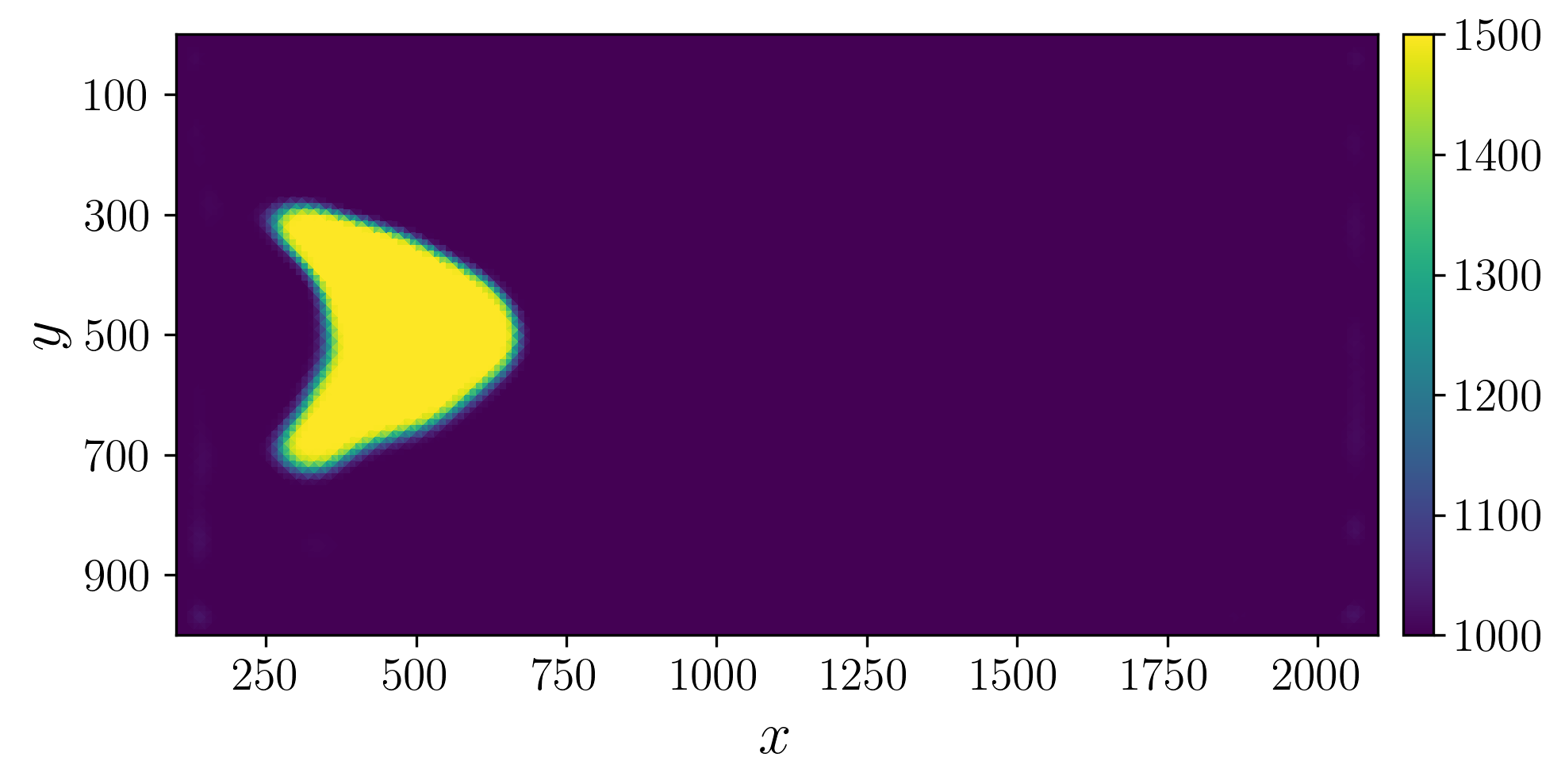}
\par\smallskip{\small (f) Mixture mean velocity.\par}
\end{minipage}

\caption{Exact, best retained MAP, and mixture mean fields for squared slowness (top row) and velocity (bottom row) in Example~\ref{example1}.}
\label{fig:ex1-reconstruction}
\end{figure}

Figure~\ref{fig:ex1-uncertainty} shows the pointwise standard deviations of the squared slowness and the velocity over the computational domain. The main uncertainty is concentrated near the boundary of the inclusion, especially around its lower-left part. In contrast, the interior of the inclusion and most of the background region exhibit much smaller uncertainty.
\begin{figure}[H]
\centering
\begin{minipage}[t]{0.33\textwidth}
\centering
\includegraphics[width=\linewidth]{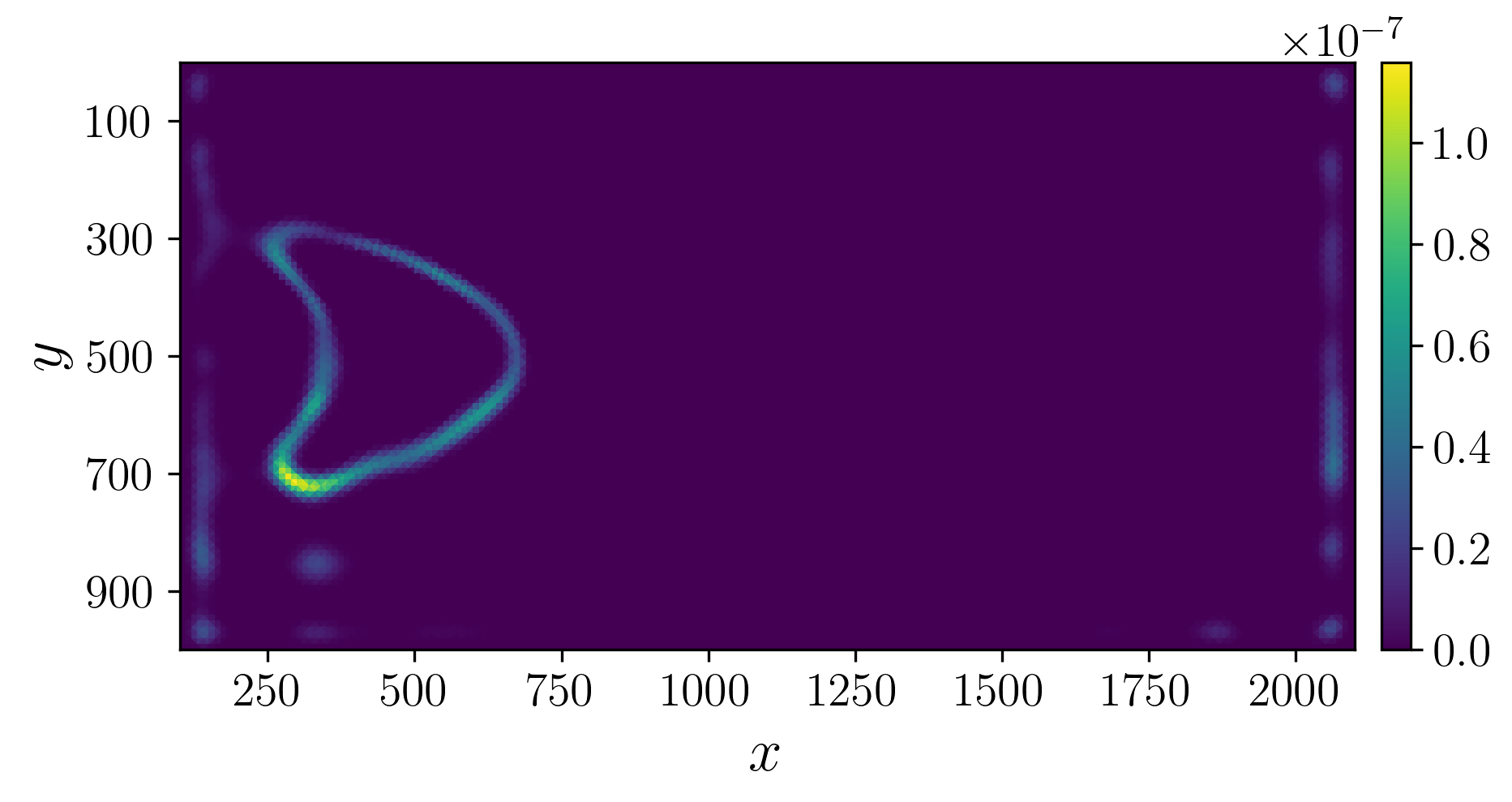}
\par\smallskip{\small (a) Squared slowness standard deviation.\par}
\end{minipage}
\begin{minipage}[t]{0.33\textwidth}
\centering
\includegraphics[width=\linewidth]{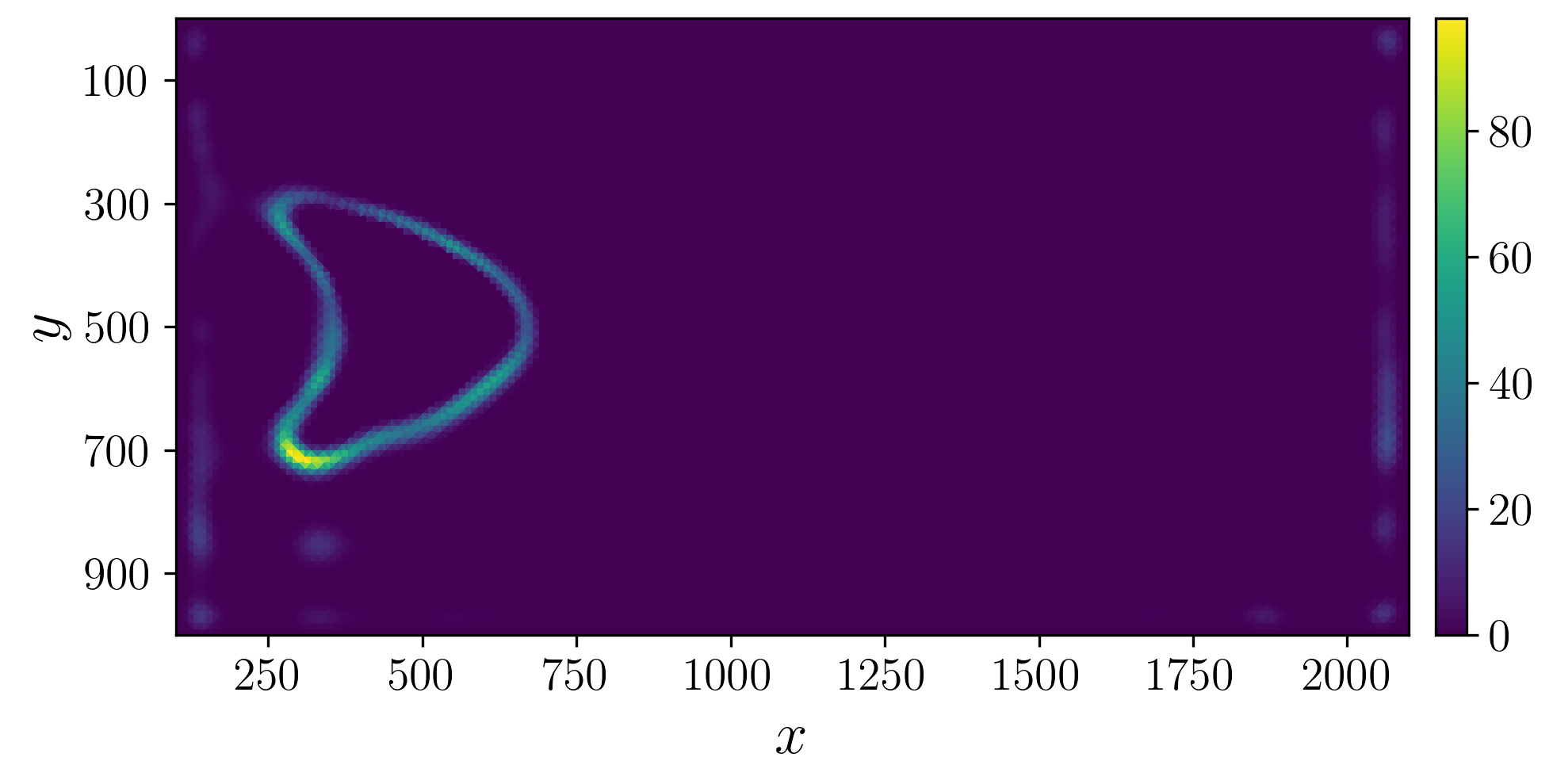}
\par\smallskip{\small (b) Velocity standard deviation.\par}
\end{minipage}
\caption{Pointwise standard deviations of squared slowness (a) and velocity (b) under the constructed Gaussian mixture.}
\label{fig:ex1-uncertainty}
\end{figure}

We compare the proposed multiscale multimodal Laplace method with a
standard single-component Laplace approximation and gPCN-MCMC. The former provides a direct baseline for assessing the effect of retaining multiple local components, while the latter serves as a sampling-based reference. All three methods use the same forward model, observations, discretization, acquisition geometry, noise level, prior, and $401$-coefficient RBF representation.

For the standard Laplace approximation, the MAP estimate is computed using $500$ Adam epochs with learning rate $0.8$. For gPCN, the chain is initialized from a fitted circular level-set profile of radius $250\,\mathrm{m}$ centred at $(1100,500)\,\mathrm{m}$. We run $40000$ samples, discard the first $10000$, and retain every tenth subsequent state, yielding $3000$ samples. The proposal parameter is $0.08$.

Figures~\ref{fig:ex2-single-laplace} and \ref{fig:ex2-gpcn-comparison} show that both methods recover the location
and overall nonconvex shape of the inclusion. For the standard Laplace approximation, the posterior means are nearly indistinguishable from the MAP-candidate fields. The lower part of the reconstructed interface is more rounded than the exact boundary, and faint artifacts appear in the background and near the ROI margins. The corresponding standard deviations are relatively small and are concentrated mainly along parts of the interface, although several isolated patches also occur away from the inclusion.

The gPCN best state and posterior means recover the same principal
geometry, but the interface transition is broader, particularly near the lower-left boundary. Noticeable velocity variation remains inside the inclusion. The gPCN standard deviations are larger and extend around most of the reconstructed interface, with the strongest variability near its lower-left part. Compared with the proposed multimodal reconstruction, the gPCN posterior means therefore appear more diffuse.

These results distinguish the two reference methods clearly. The standard Laplace approximation describes uncertainty only near one MAP candidate and produces relatively small local variability, whereas gPCN explores a broader range of interface configurations. The proposed method retains several local Gaussian components and is therefore not restricted to a single MAP neighbourhood.

\begin{figure}[H]
\centering

\begin{minipage}[t]{0.33\textwidth}
\centering
\includegraphics[width=\linewidth]{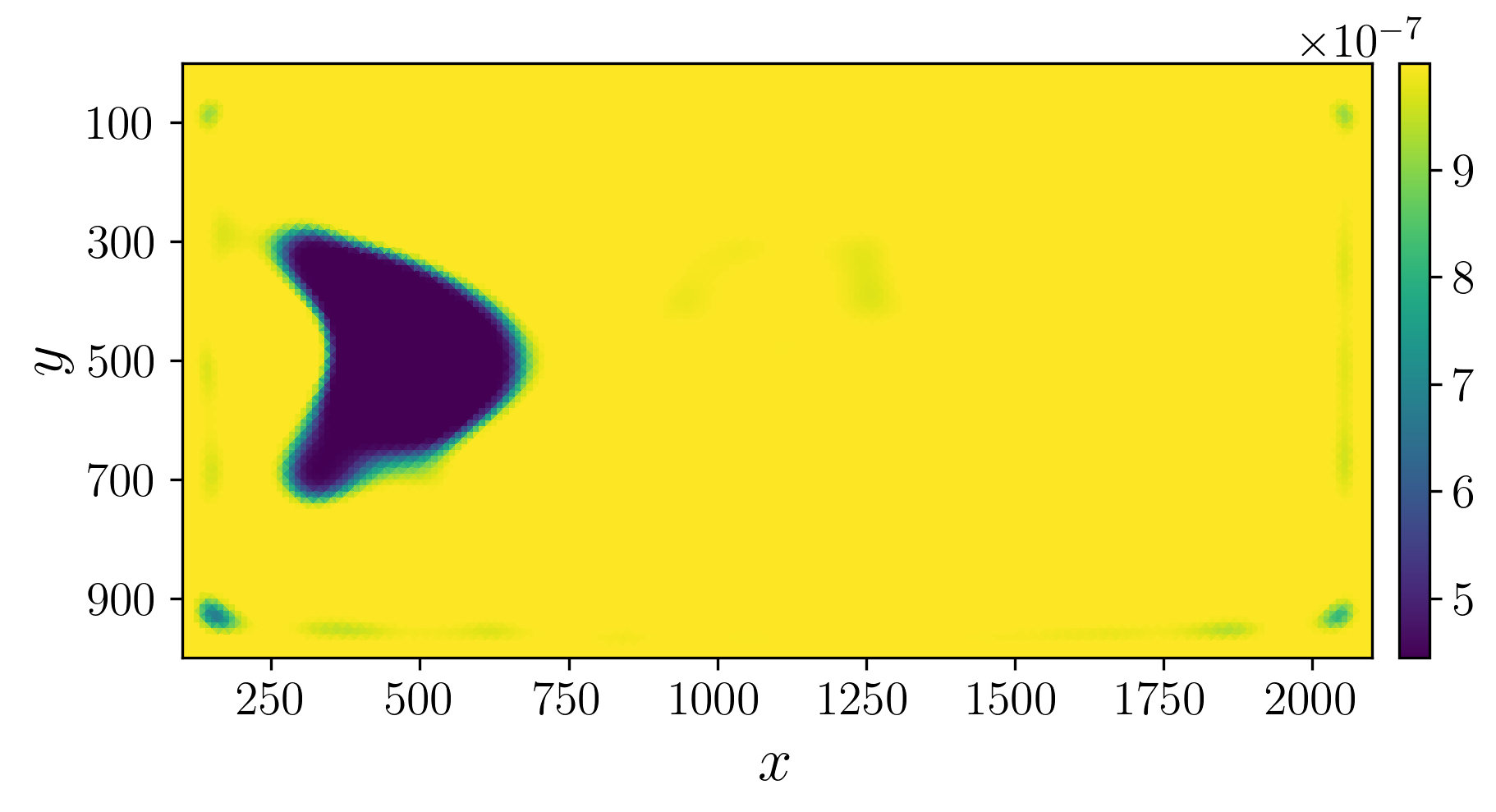}
\par\smallskip{\small (a) MAP-candidate squared slowness.\par}
\end{minipage}\hfill
\begin{minipage}[t]{0.33\textwidth}
\centering
\includegraphics[width=\linewidth]{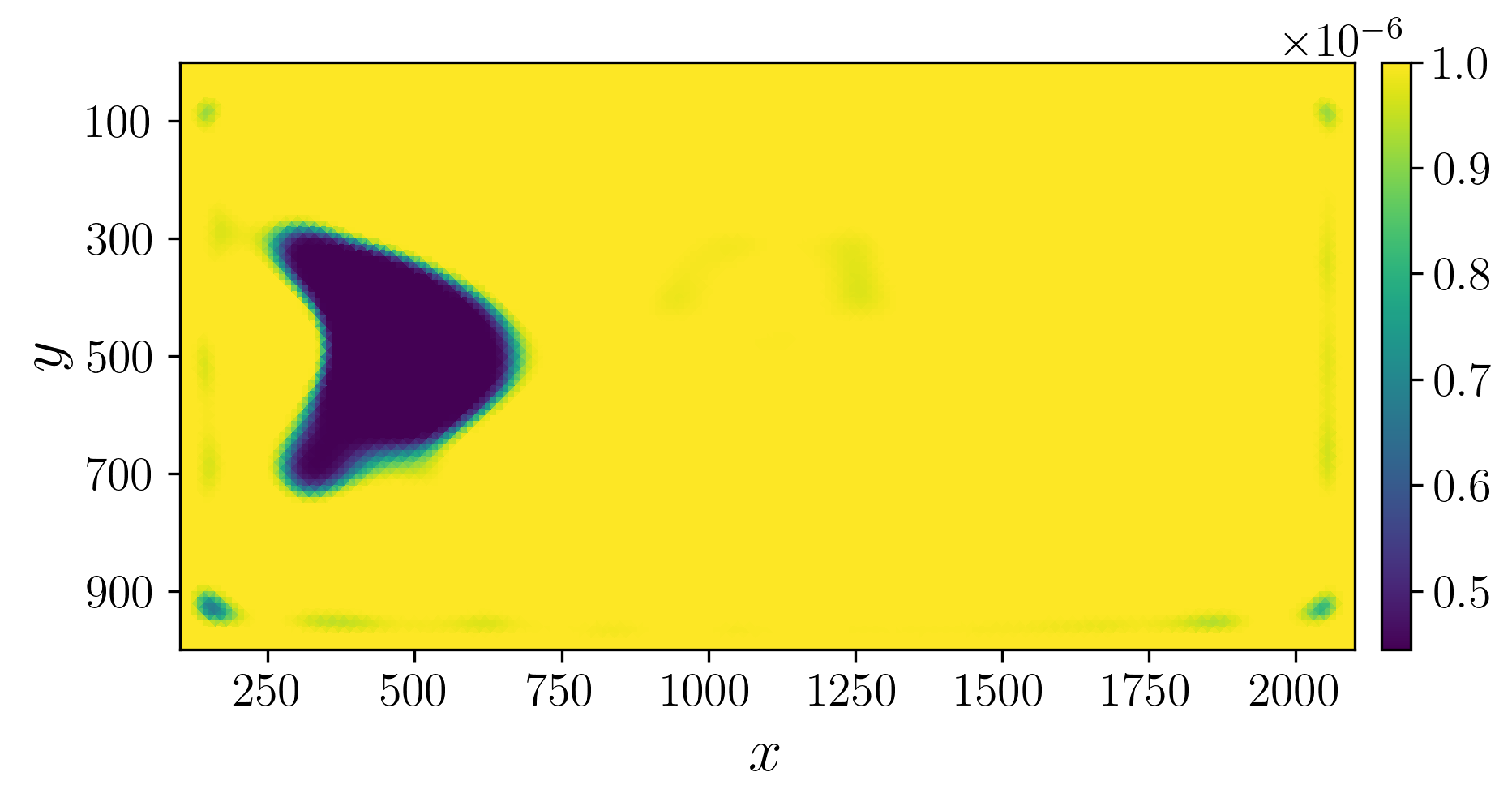}
\par\smallskip{\small (b) Posterior mean of squared slowness.\par}
\end{minipage}\hfill
\begin{minipage}[t]{0.33\textwidth}
\centering
\includegraphics[width=\linewidth]{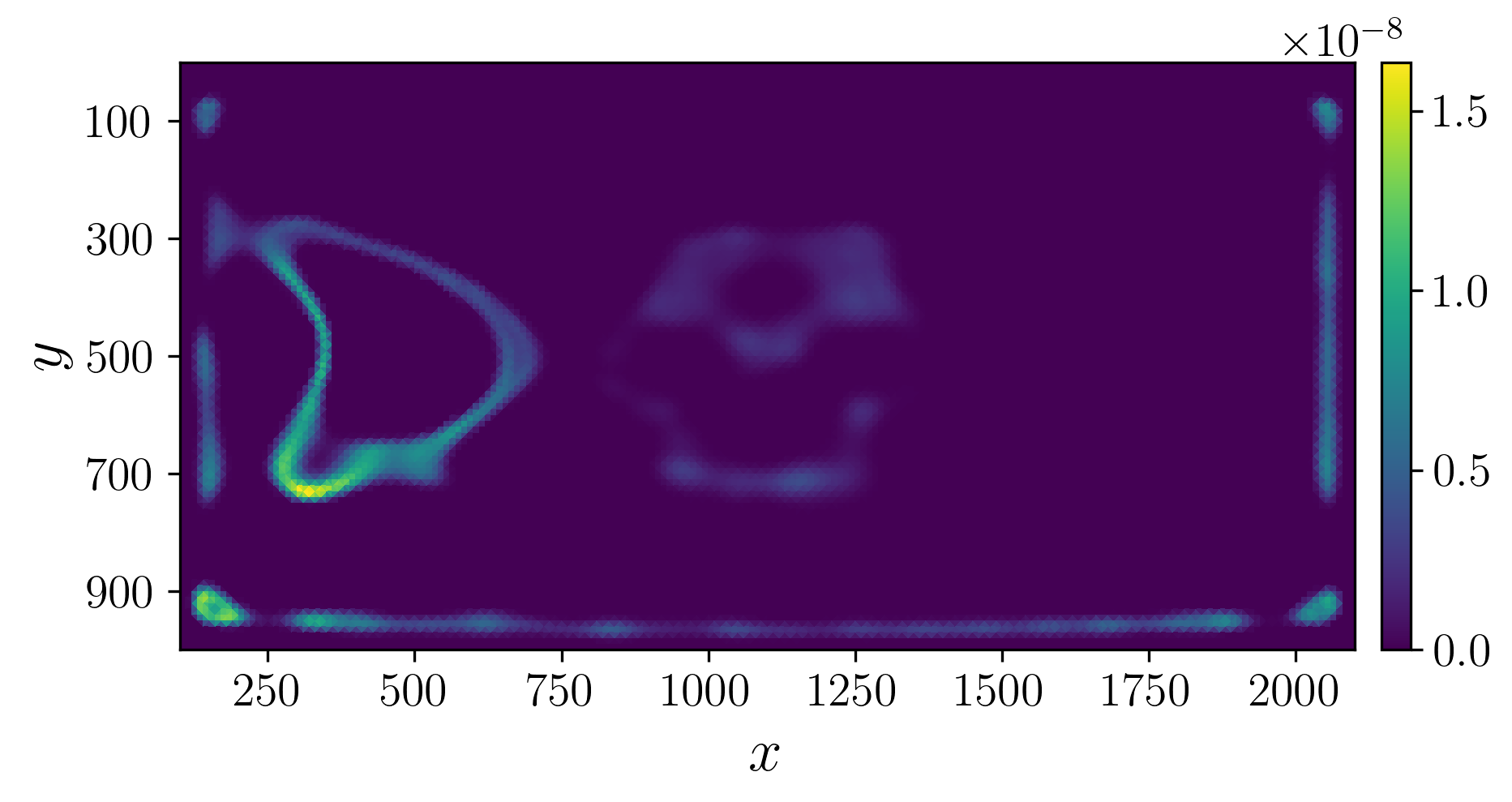}
\par\smallskip{\small (c) Posterior standard deviation of squared slowness.\par}
\end{minipage}

\par\smallskip

\begin{minipage}[t]{0.33\textwidth}
\centering
\includegraphics[width=\linewidth]{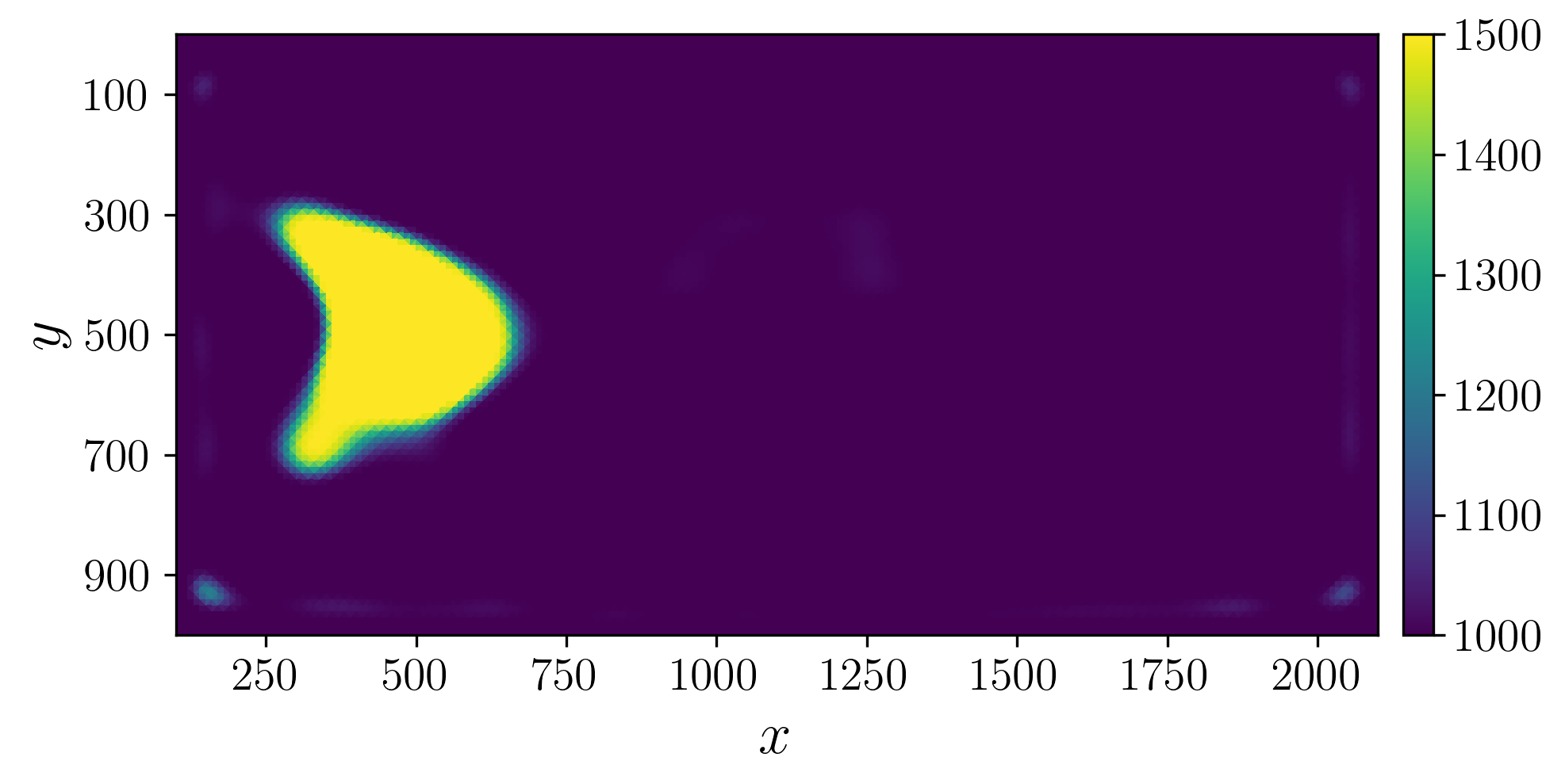}
\par\smallskip{\small (d) MAP-candidate velocity.\par}
\end{minipage}\hfill
\begin{minipage}[t]{0.33\textwidth}
\centering
\includegraphics[width=\linewidth]{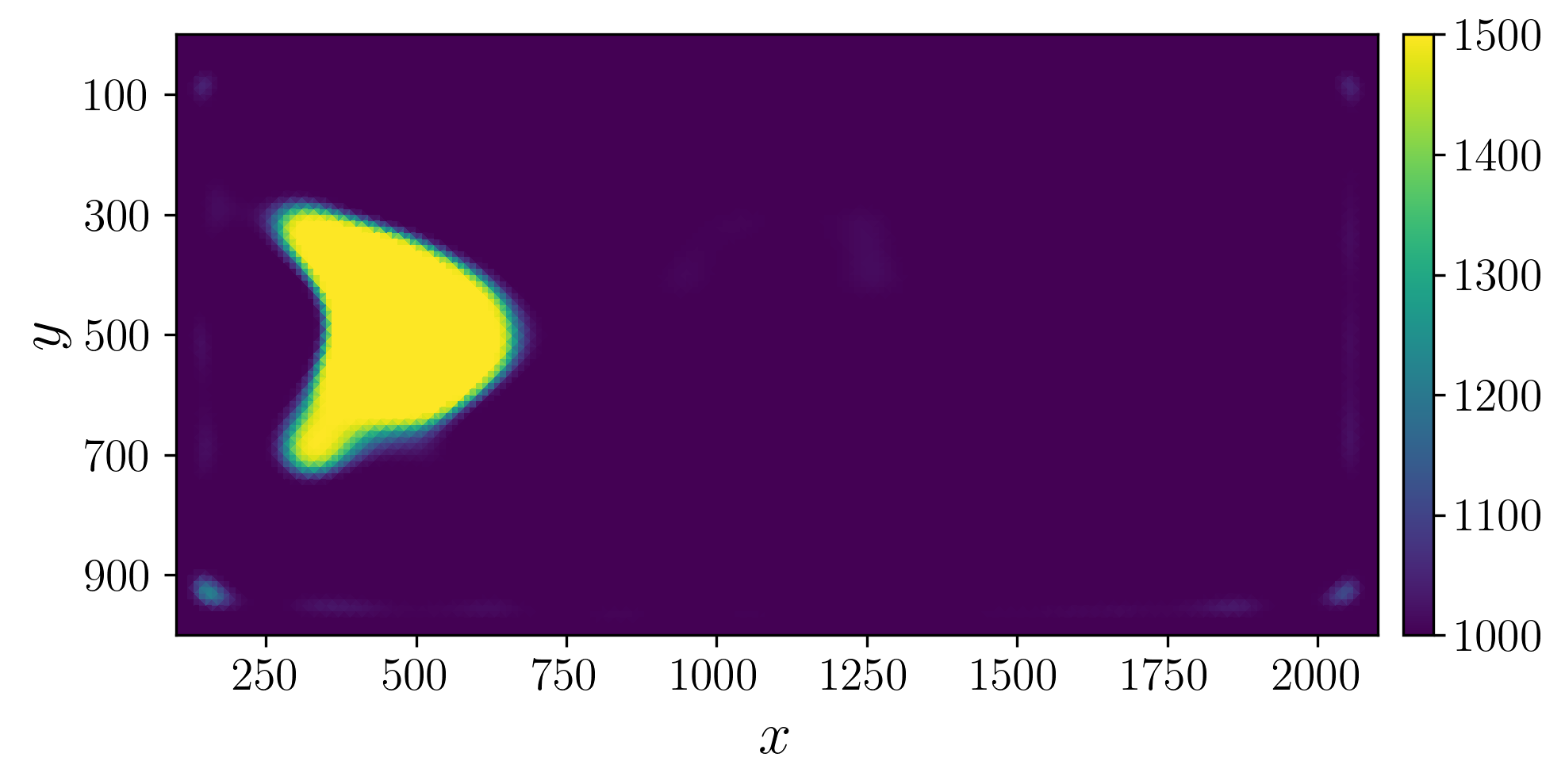}
\par\smallskip{\small (e) Posterior mean of velocity.\par}
\end{minipage}\hfill
\begin{minipage}[t]{0.33\textwidth}
\centering
\includegraphics[width=\linewidth]{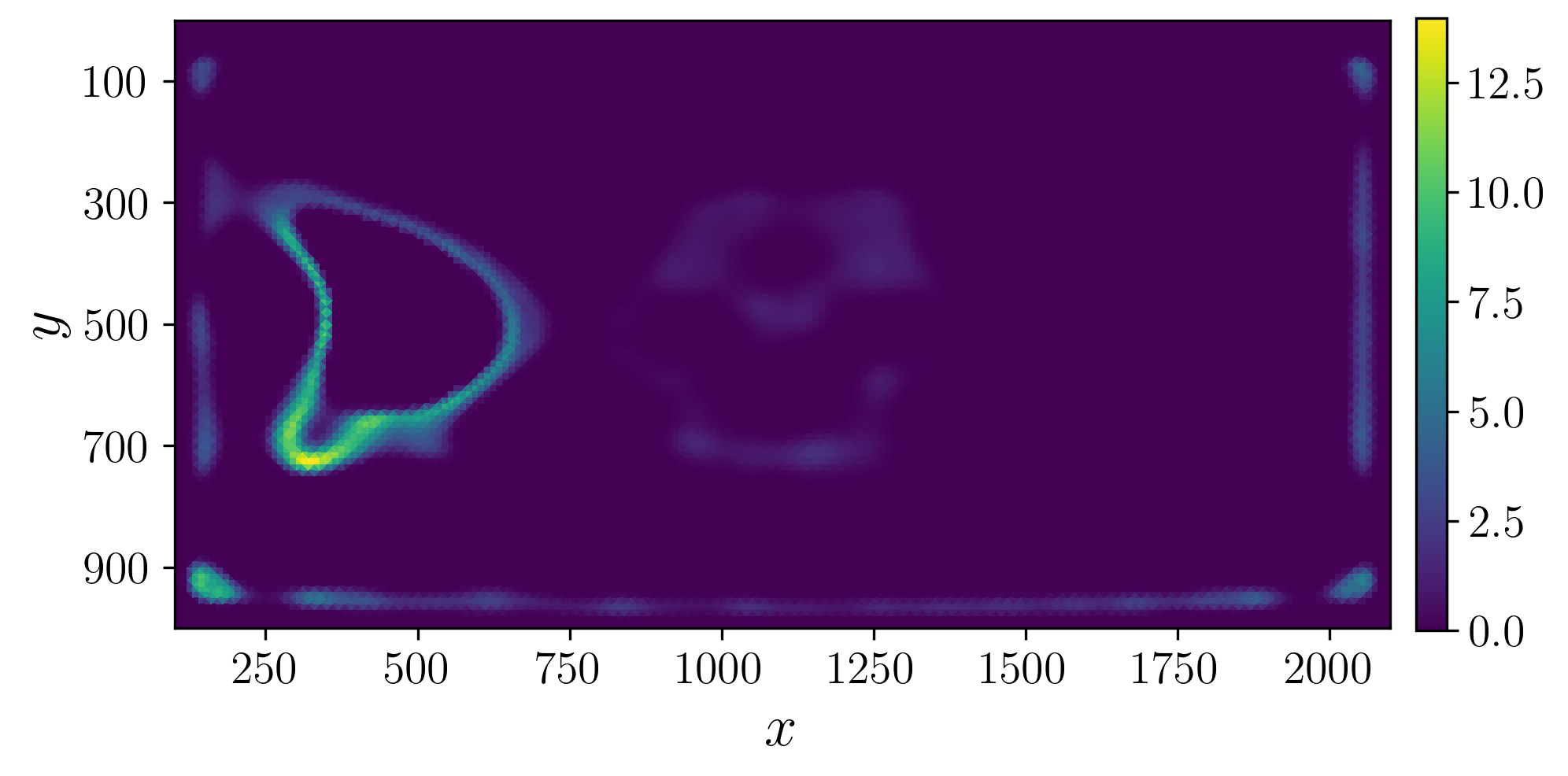}
\par\smallskip{\small (f) Posterior standard deviation of velocity.\par}
\end{minipage}

\caption{Standard single-component Laplace approximation for the
nonconvex-inclusion experiment (Example~\ref{example1}). The top and bottom rows show squared slowness and velocity, respectively. From left to right: the selected MAP candidate, posterior mean, and pointwise posterior standard deviation.}
\label{fig:ex2-single-laplace}
\end{figure}

\begin{figure}[H]
\centering

\begin{minipage}[t]{0.33\textwidth}
\centering
\includegraphics[width=\linewidth]{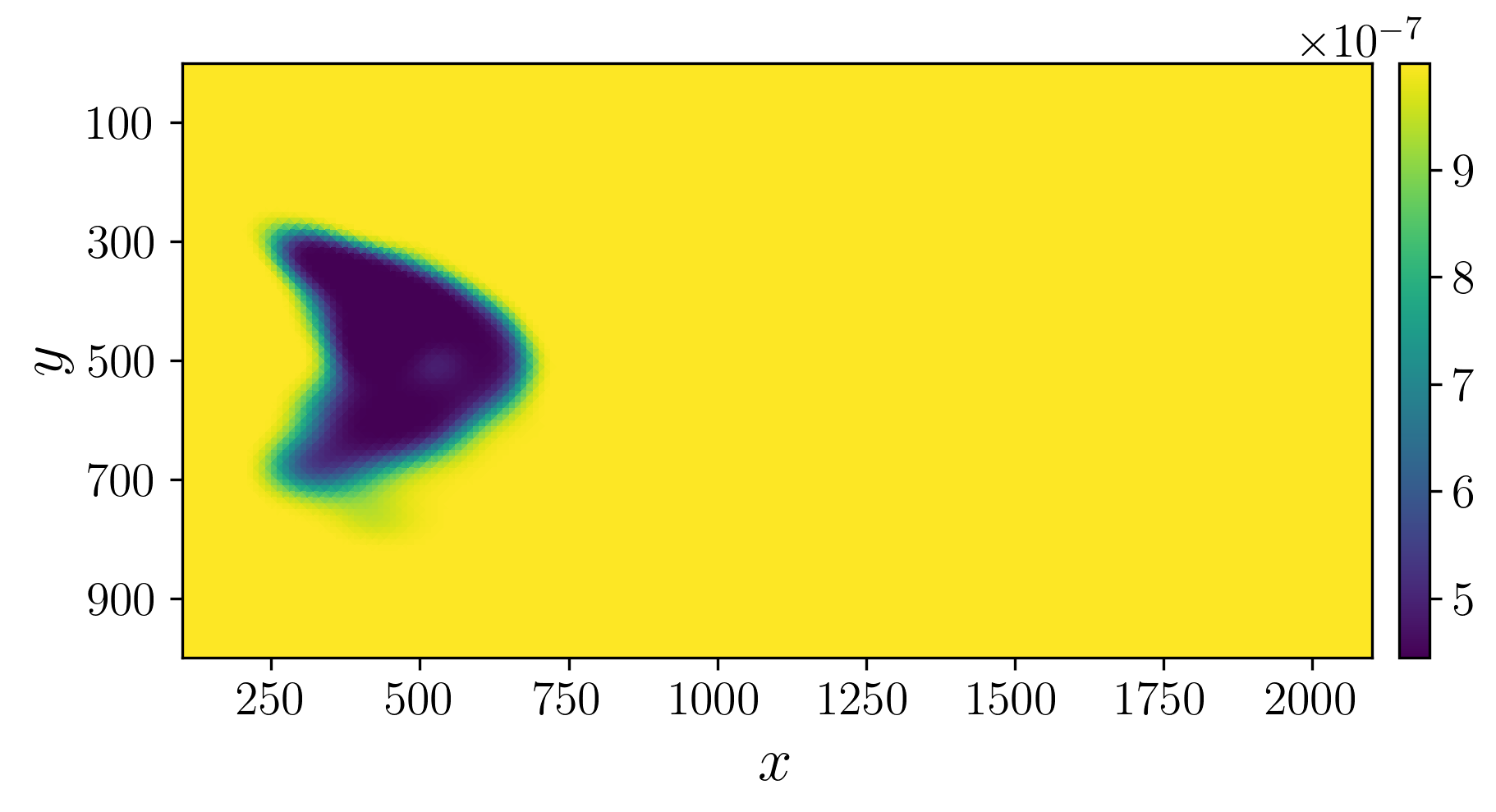}
\par\smallskip{\small (a) Lowest-objective squared slowness.\par}
\end{minipage}\hfill
\begin{minipage}[t]{0.33\textwidth}
\centering
\includegraphics[width=\linewidth]{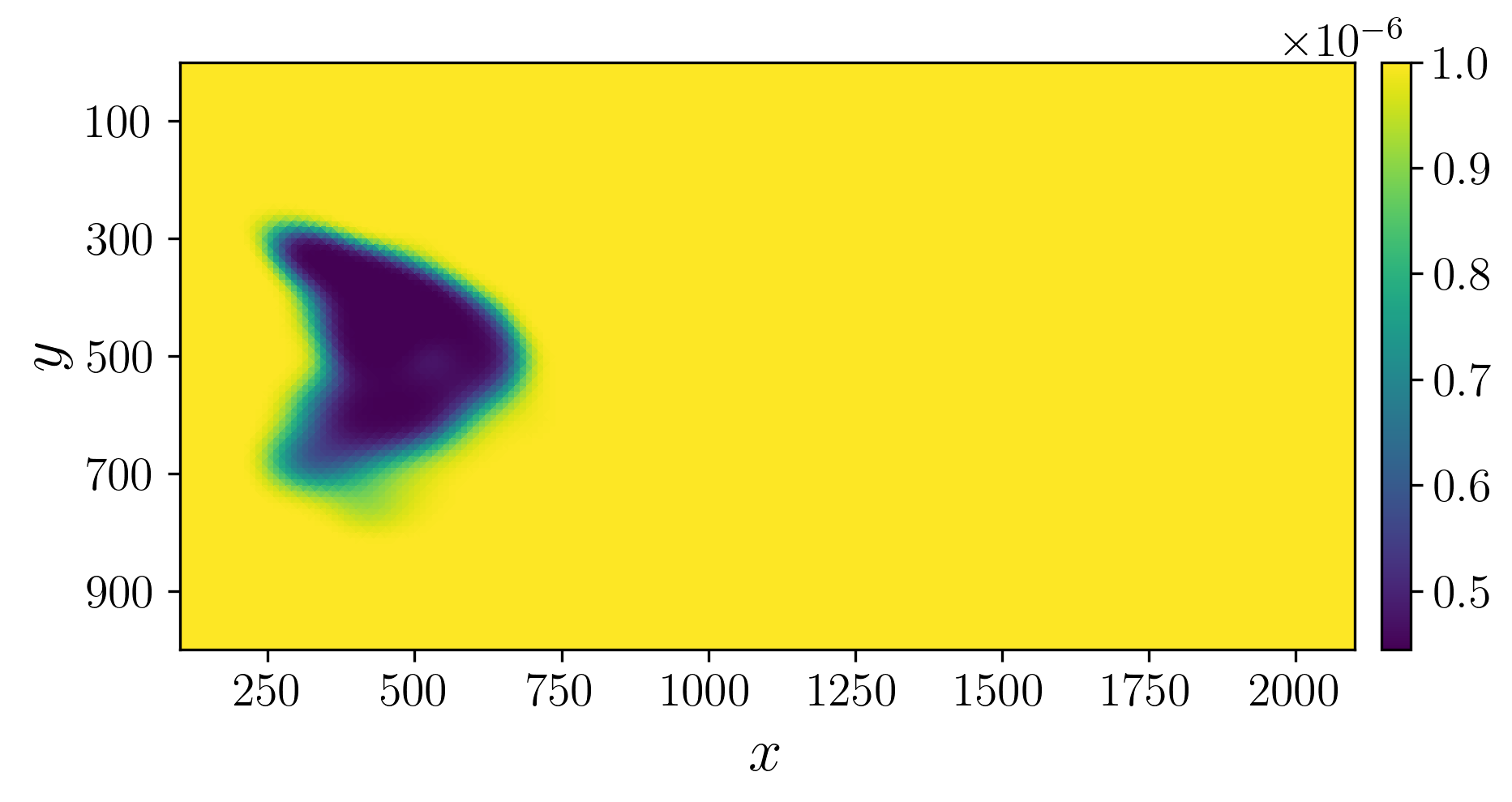}
\par\smallskip{\small (b) Posterior mean of squared slowness.\par}
\end{minipage}\hfill
\begin{minipage}[t]{0.33\textwidth}
\centering
\includegraphics[width=\linewidth]{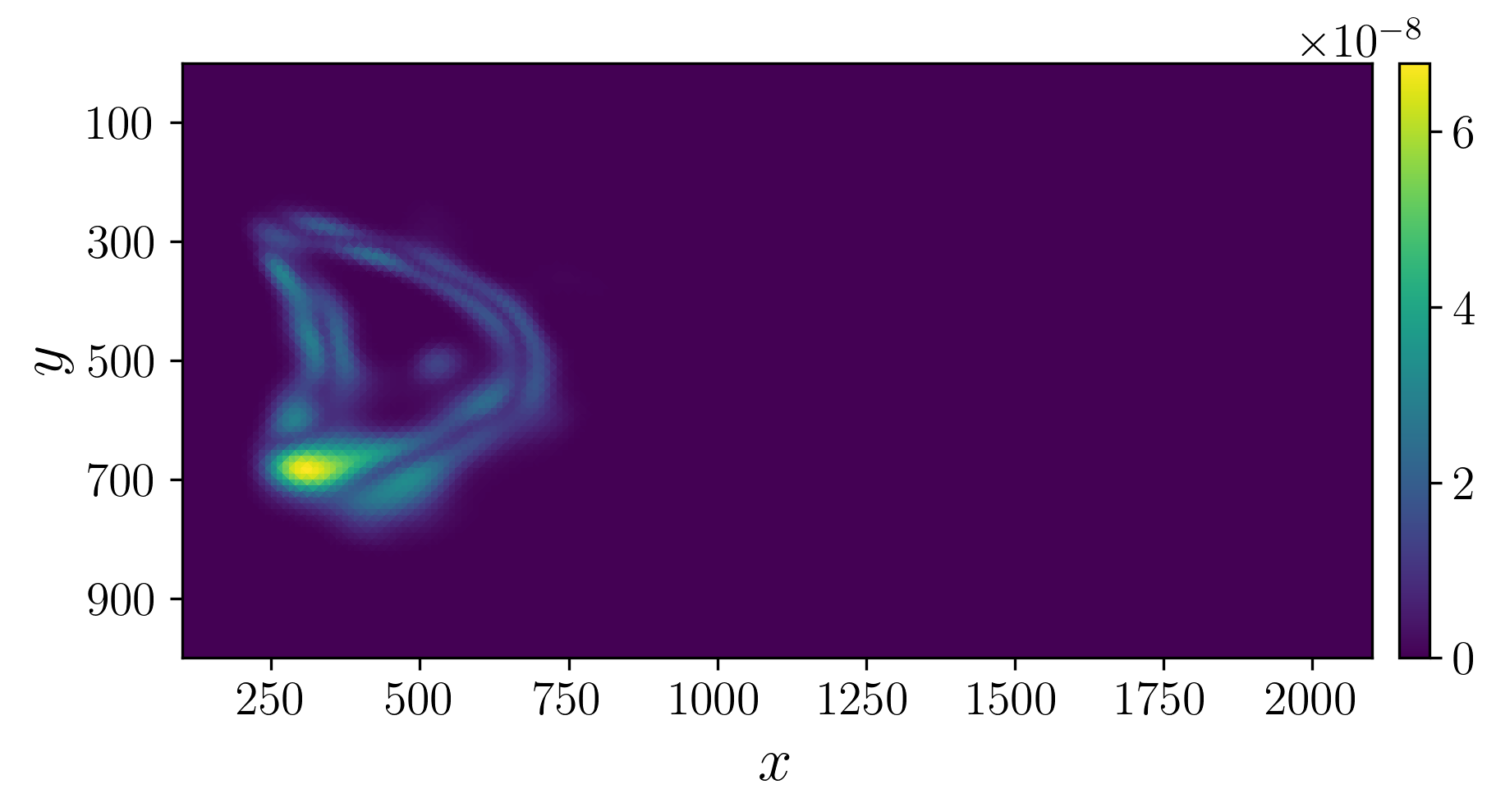}
\par\smallskip{\small (c) Posterior standard deviation of squared slowness.\par}
\end{minipage}

\par\medskip

\begin{minipage}[t]{0.33\textwidth}
\centering
\includegraphics[width=\linewidth]{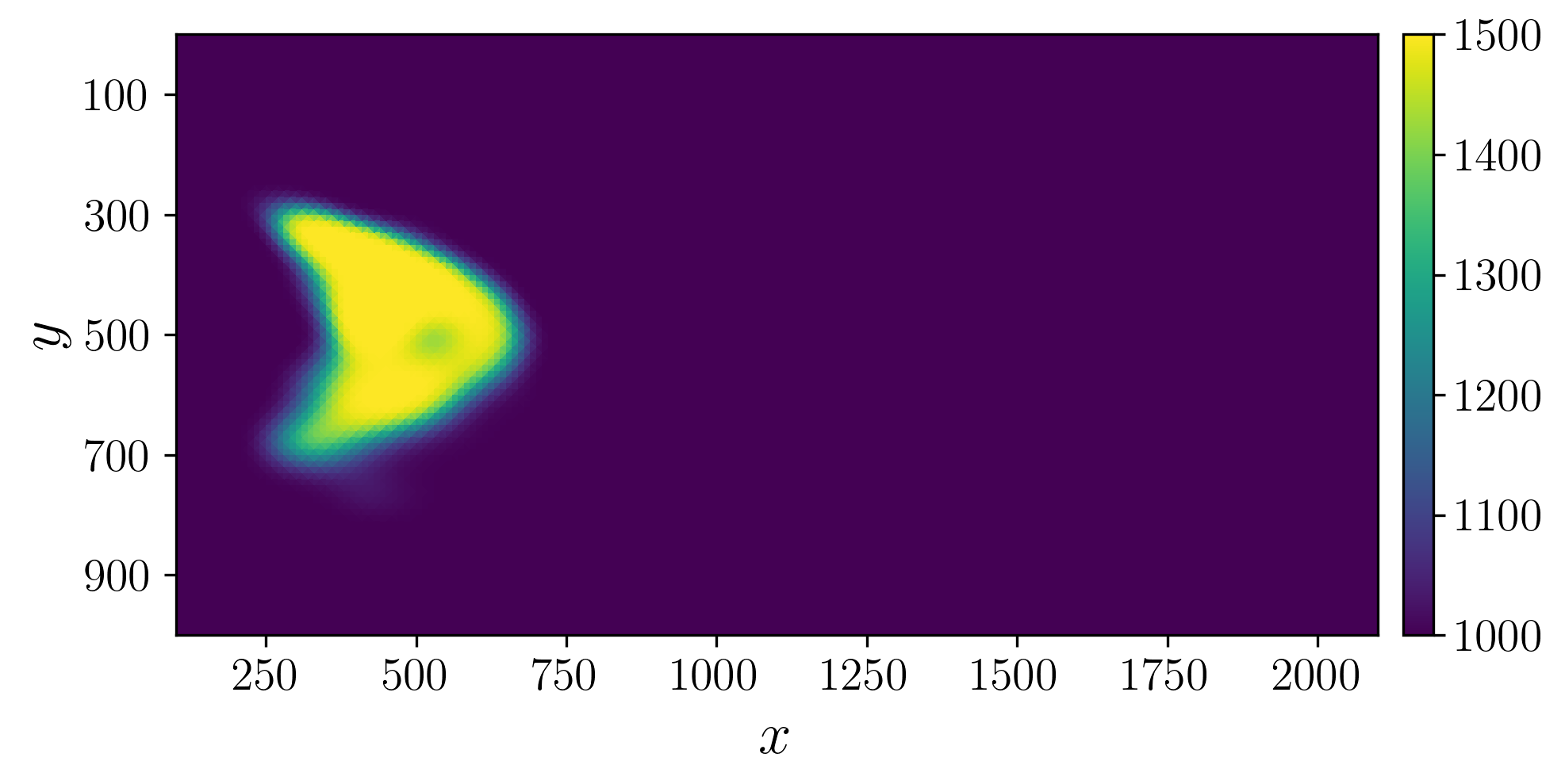}
\par\smallskip{\small (d) Lowest-objective velocity.\par}
\end{minipage}\hfill
\begin{minipage}[t]{0.33\textwidth}
\centering
\includegraphics[width=\linewidth]{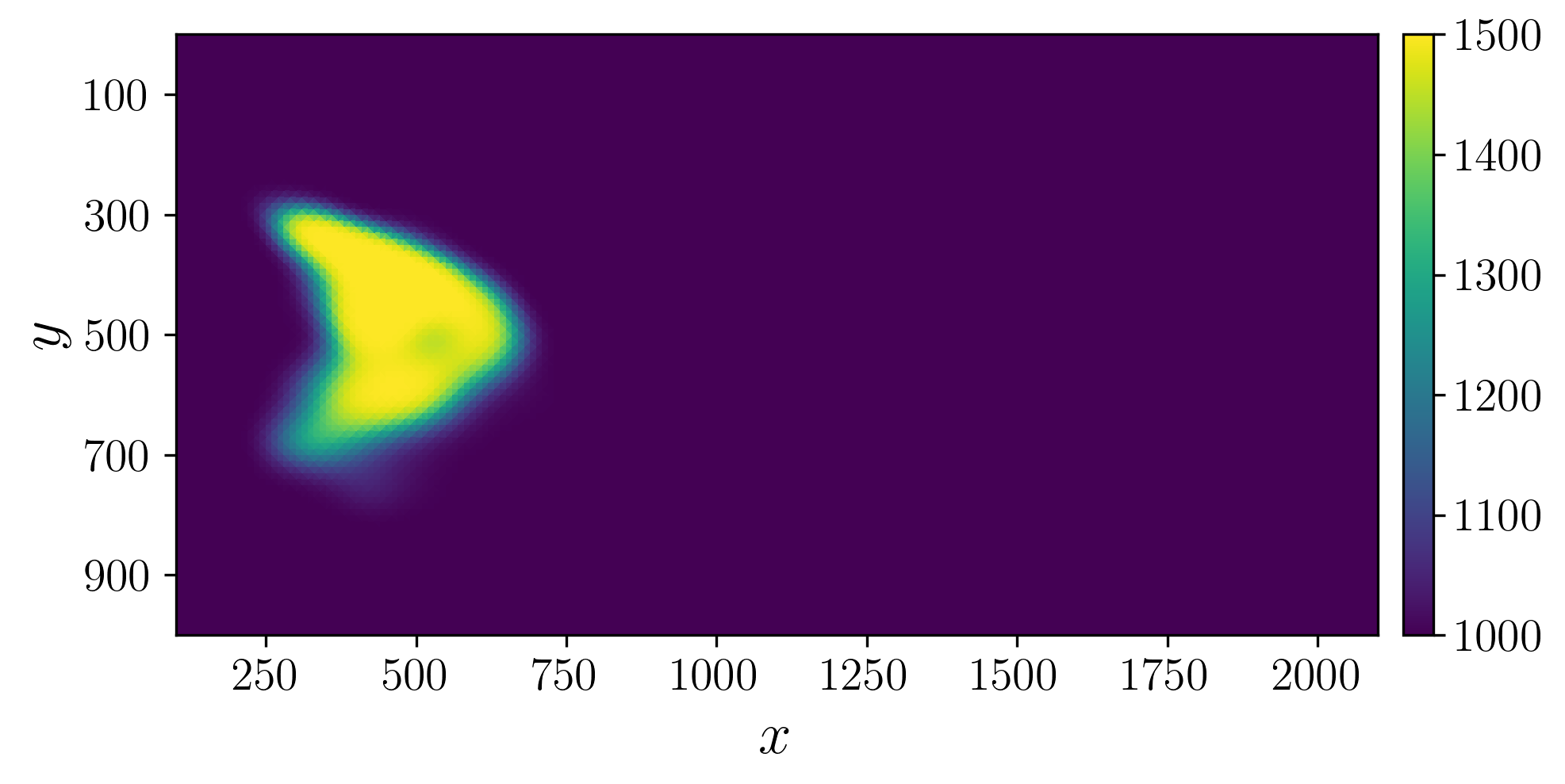}
\par\smallskip{\small (e) Posterior mean of velocity.\par}
\end{minipage}\hfill
\begin{minipage}[t]{0.33\textwidth}
\centering
\includegraphics[width=\linewidth]{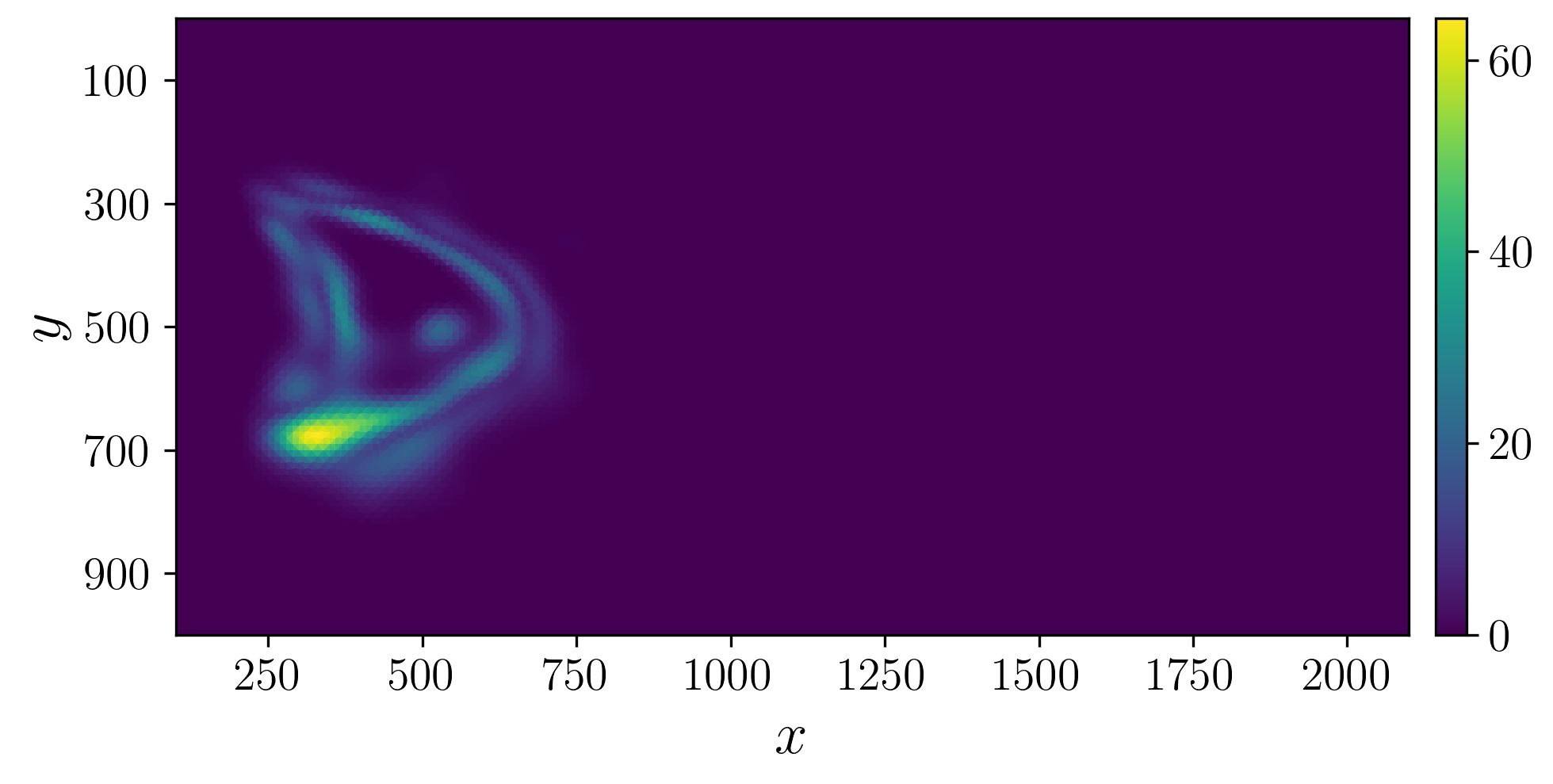}
\par\smallskip{\small (f) Posterior standard deviation of velocity.\par}
\end{minipage}

\caption{gPCN-MCMC results for the nonconvex-inclusion experiment (Example~\ref{example1}). The top and bottom rows show squared slowness and velocity, respectively. From left to right: the retained state with the smallest posterior objective, posterior
mean, and pointwise posterior standard deviation.}
\label{fig:ex2-gpcn-comparison}
\end{figure}


\begin{table}[H]
\centering
\caption{Quantitative comparison of gPCN-MCMC, the standard
single-component Laplace approximation, and the proposed multiscale
multimodal Laplace method for the nonconvex-inclusion experiment
(Example~\ref{example1}). The best value in each row is shown in bold.}
\label{tab:ex2-three-method-comparison}
\begingroup
\small
\renewcommand{\arraystretch}{1.25}
\setlength{\tabcolsep}{5pt}
\begin{tabularx}{\linewidth}{
@{}>{\raggedright\arraybackslash}p{0.43\linewidth}
*{3}{>{\centering\arraybackslash}X}@{}
}
\toprule
Evaluation metric
& \shortstack{gPCN\\(MCMC)}
& \shortstack{Single\\Laplace}
& \shortstack{Multiscale\\multimodal\\Laplace} \\
\midrule
Relative error of mean squared slowness $E_{\mathrm{rel}}^{\mu}$
& 3.3530\%
& 2.6522\%
& \textbf{2.2939\%} \\

Relative error of mean velocity $E_{\mathrm{rel}}^{v}$
& 3.2749\%
& 2.4359\%
& \textbf{2.2479\%} \\

Hausdorff distance
& 28.28~m
& 250.80~m
& \textbf{20.00~m} \\

Boundary concentration $R_{\mathrm{bdry}}(20\,\mathrm m)$
& 61.5752\%
& 52.7846\%
& \textbf{87.2991\%} \\

Node-based coverage
$\mathrm{BCov}^{\mathrm{node}}_{0.05,\,20\,\mathrm m}$
& 86.8098\%
& 31.5951\%
& \textbf{92.6380\%} \\

Elapsed time (s)
& 15219.58
& \textbf{365.59}
& 999.32 \\
\bottomrule
\end{tabularx}
\endgroup
\end{table}

Table~\ref{tab:ex2-three-method-comparison} summarizes the reconstruction errors, boundary uncertainty measures, and computational times. Relative to the single-component Laplace approximation, the proposed method reduces the squared-slowness and velocity errors from $2.65\%$ and $2.44\%$ to $2.29\%$ and $2.25\%$, respectively. The Hausdorff distance decreases from $250.80\,\mathrm{m}$ to $20.00\,\mathrm{m}$, while the boundary
concentration and coverage increase from $52.78\%$ and $31.60\%$ to
$87.30\%$ and $92.64\%$. The proposed method also gives the best values among the three methods for all five accuracy and uncertainty metrics.

The recorded runtimes are $365.59\,\mathrm{s}$ for the single-component Laplace approximation, $999.32\,\mathrm{s}$ for the proposed multiscale MAP/Laplace procedure, and $15219.58\,\mathrm{s}$ for gPCN. Thus, the multiscale method costs about $2.7$ times as much as the single-component approximation but uses less than one fifteenth of the gPCN runtime in this experiment.

\textbf{Robustness to observation noise.}
We test four absolute noise levels,$\sigma_{\mathrm{noise}}
\in \{0,\ 10^{-4},\ 5\times10^{-4},\ 10^{-3}\}.$
For noisy observations, the likelihood scale is set to
$\delta_{\mathrm{obs}}=\sigma_{\mathrm{noise}}$. In the noiseless case, no noise is added, while $\delta_{\mathrm{obs}}=10^{-4}$ is used to keep the data-misfit term well defined. The prior parameters and all other experimental settings remain unchanged.


Table~\ref{tab:ex1-noise-robustness} and
Figures~\ref{fig:ex2-noise-uncertainty-mean}--\ref{fig:ex2-noise-uncertainty}
summarize the effect of observation noise. The two lowest noise levels
produce similar mean-field errors and nearly identical reconstructions.
At $5\times10^{-4}$, the errors increase and the interface becomes more
diffuse, with visible variation inside the inclusion. At $10^{-3}$, the
inclusion remains identifiable, but the Hausdorff distance increases to
$85.44\,\mathrm m$, consistent with the increased smoothing and somewhat less precise boundary localization in the posterior mean.

The standard deviation is concentrated near the interface at the two
lowest noise levels. At $5\times10^{-4}$, it spreads into the inclusion
and background, while at $10^{-3}$ a broader band forms around the
interface, with the largest values along its lower-left portion. The
high boundary coverage at the largest noise level therefore reflects a
wider uncertainty band rather than more accurate localization.

\begin{table}[H]
\centering
\caption{Reconstruction accuracy and uncertainty diagnostics at different
observation-noise levels for Example~\ref{example1}.}
\label{tab:ex1-noise-robustness}

\begingroup
\small
\renewcommand{\arraystretch}{1.25}
\setlength{\tabcolsep}{5pt}

\begin{tabularx}{\linewidth}{
@{}>{\raggedright\arraybackslash}p{0.43\linewidth}
*{4}{>{\centering\arraybackslash}X}@{}
}
\toprule

\multirow{2}{*}{Evaluation metric}
& \multicolumn{4}{c}{Gaussian noise standard deviation
$\sigma_{\mathrm{noise}}$} \\

\cmidrule(lr){2-5}

& $0$ & $10^{-4}$ & $5\times10^{-4}$ & $10^{-3}$ \\

\midrule

Squared-slowness mean error $E_{\mathrm{rel}}^{\mu}$ (\%)
& 2.2654 & 2.2939 & 3.4057 & 3.7893 \\

Velocity mean error $E_{\mathrm{rel}}^{v}$ (\%)
& 2.1472 & 2.2479 & 3.1232 & 3.5948 \\

Hausdorff distance $d_H$ (\si{\metre})
& 28.28 & 20.00 & 28.28 & 85.44 \\

Boundary concentration $R_{\mathrm{bdry}}(\SI{20}{\metre})$ (\%)
& 81.4079 & 87.2991 & 59.3177 & 75.9247 \\

Node-based coverage
$\mathrm{BCov}^{\mathrm{node}}_{0.05,\,\SI{20}{\metre}}$ (\%)
& 82.5153 & 92.6380 & 89.2638 & 97.5460 \\

\bottomrule
\end{tabularx}

\endgroup
\end{table}

\begin{figure}[H]
\centering
\begin{minipage}[t]{0.33\linewidth}
\centering
\includegraphics[width=\linewidth]{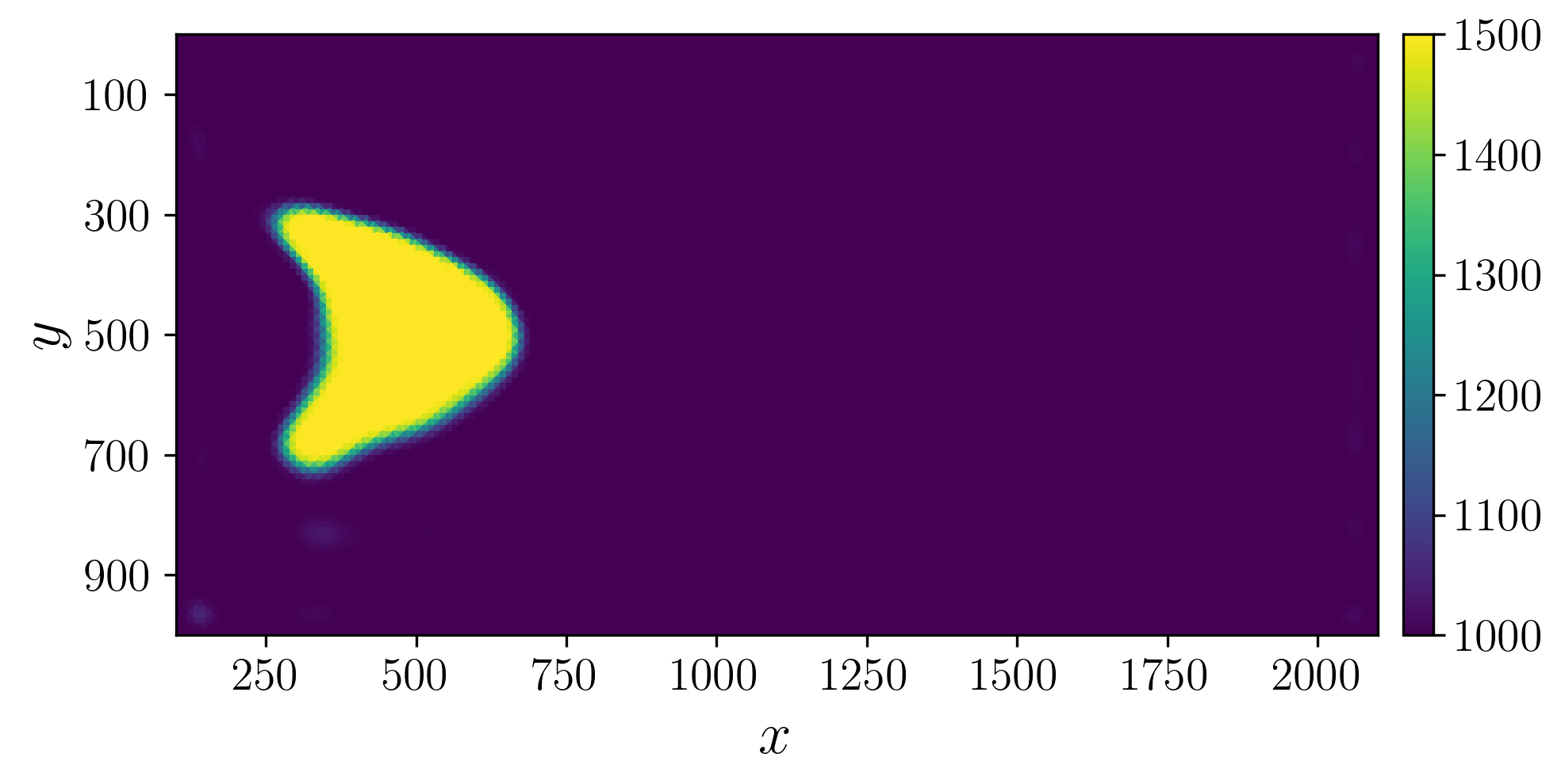}
\par\smallskip{\small (a) $\sigma_{\mathrm{noise}}=0$.\par}
\end{minipage}
\begin{minipage}[t]{0.33\linewidth}
\centering
\includegraphics[width=\linewidth]{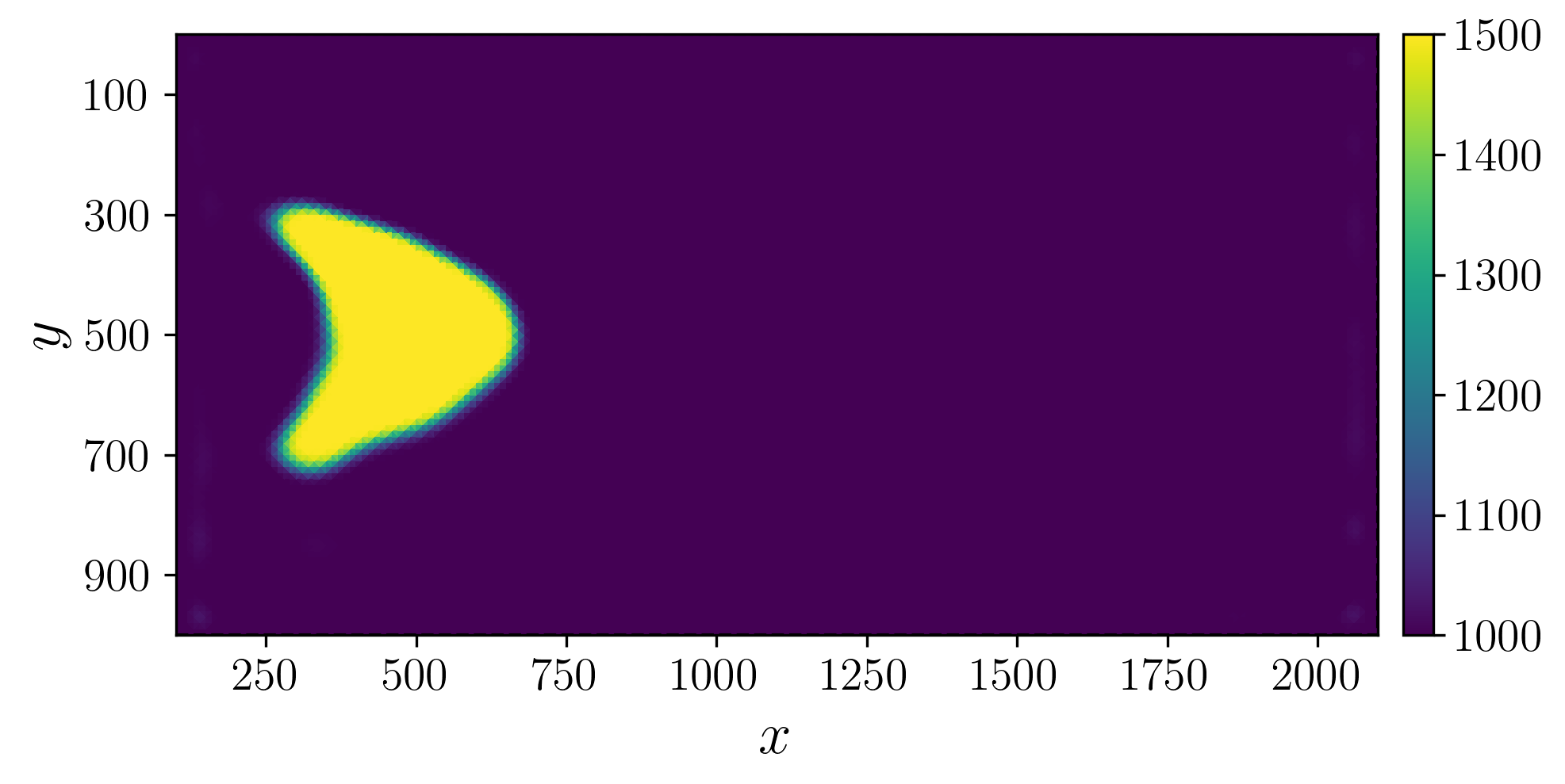}
\par\smallskip{\small (b) $\sigma_{\mathrm{noise}}=1\times10^{-4}$.\par}
\end{minipage}
\par\medskip
\begin{minipage}[t]{0.33\linewidth}
\centering
\includegraphics[width=\linewidth]{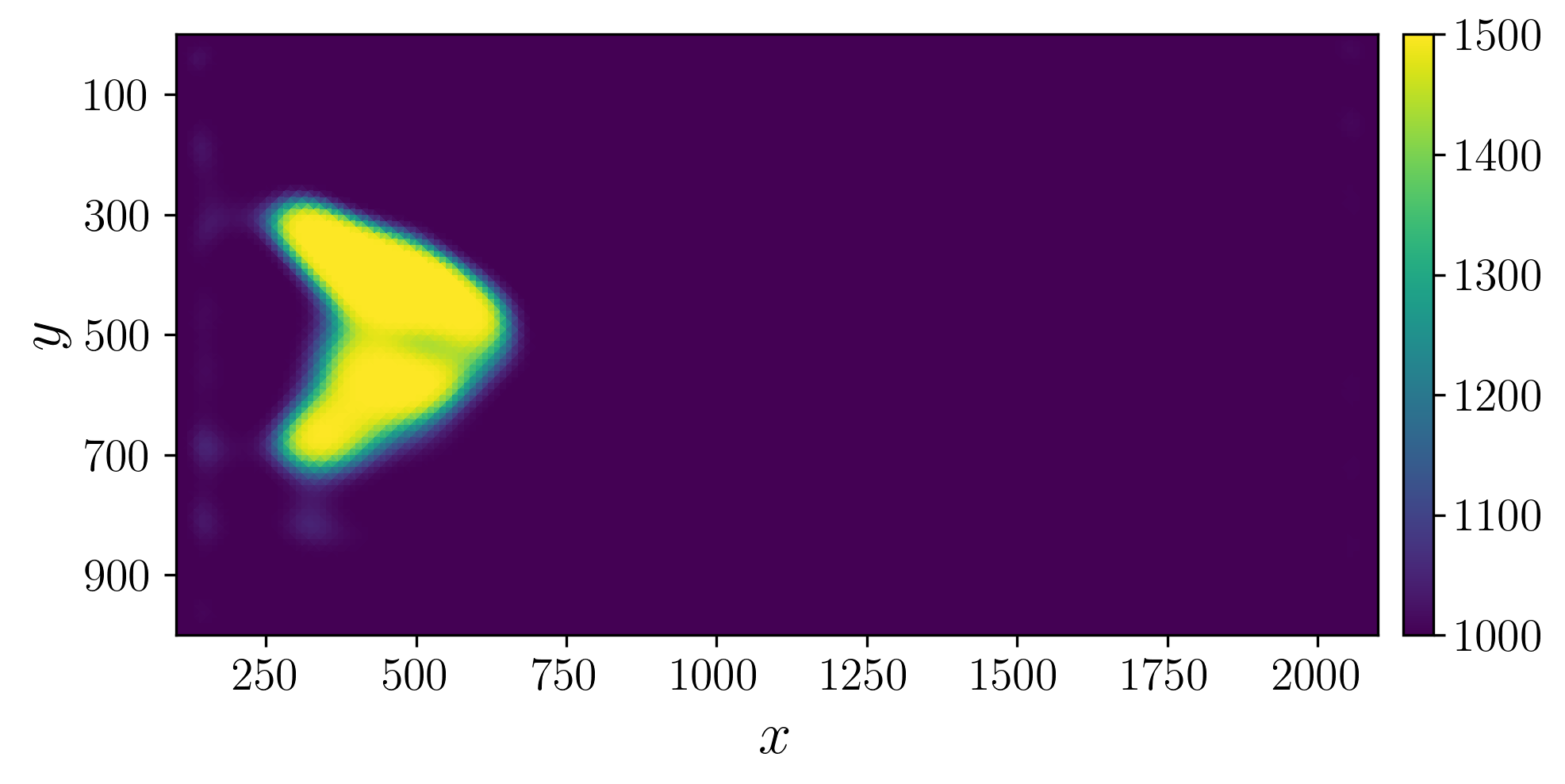}
\par\smallskip{\small (c) $\sigma_{\mathrm{noise}}=5\times10^{-4}$.\par}
\end{minipage}
\begin{minipage}[t]{0.33\linewidth}
\centering
\includegraphics[width=\linewidth]{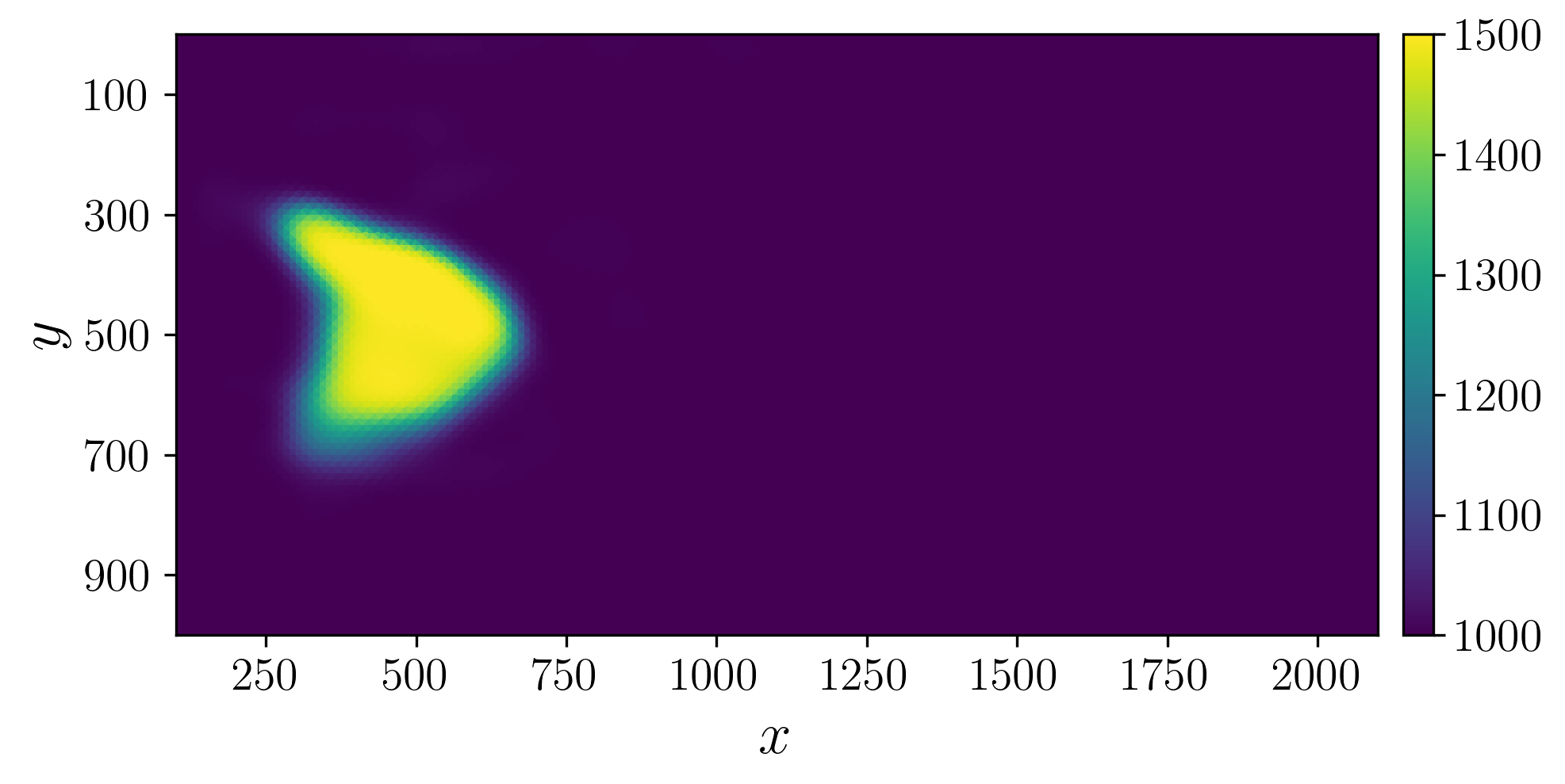}
\par\smallskip{\small (d) $\sigma_{\mathrm{noise}}=1\times10^{-3}$.\par}
\end{minipage}
\caption{Posterior mean velocity fields at different observation-noise levels for Example~\ref{example1}. These fields are computed using our proposed multi-scale multimodal Laplace method and estimated from 100 samples of the final Gaussian-mixture approximation in each run.}
\label{fig:ex2-noise-uncertainty-mean}
\end{figure}

\begin{figure}[H]
\centering
\begin{minipage}[t]{0.33\linewidth}
\centering
\includegraphics[width=\linewidth]{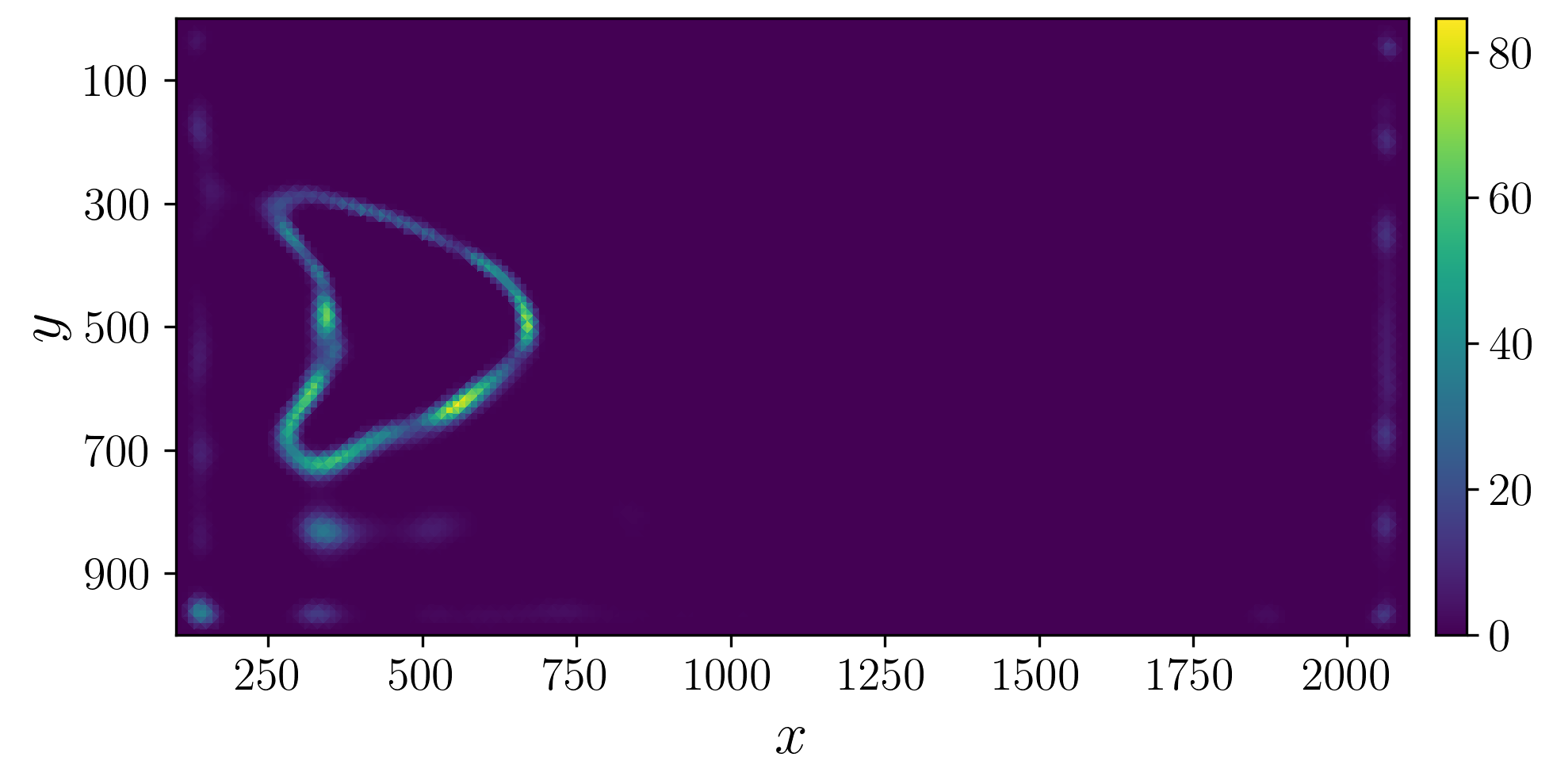}
\par\smallskip{\small (a) $\sigma_{\mathrm{noise}}=0$.\par}
\end{minipage}
\begin{minipage}[t]{0.33\linewidth}
\centering
\includegraphics[width=\linewidth]{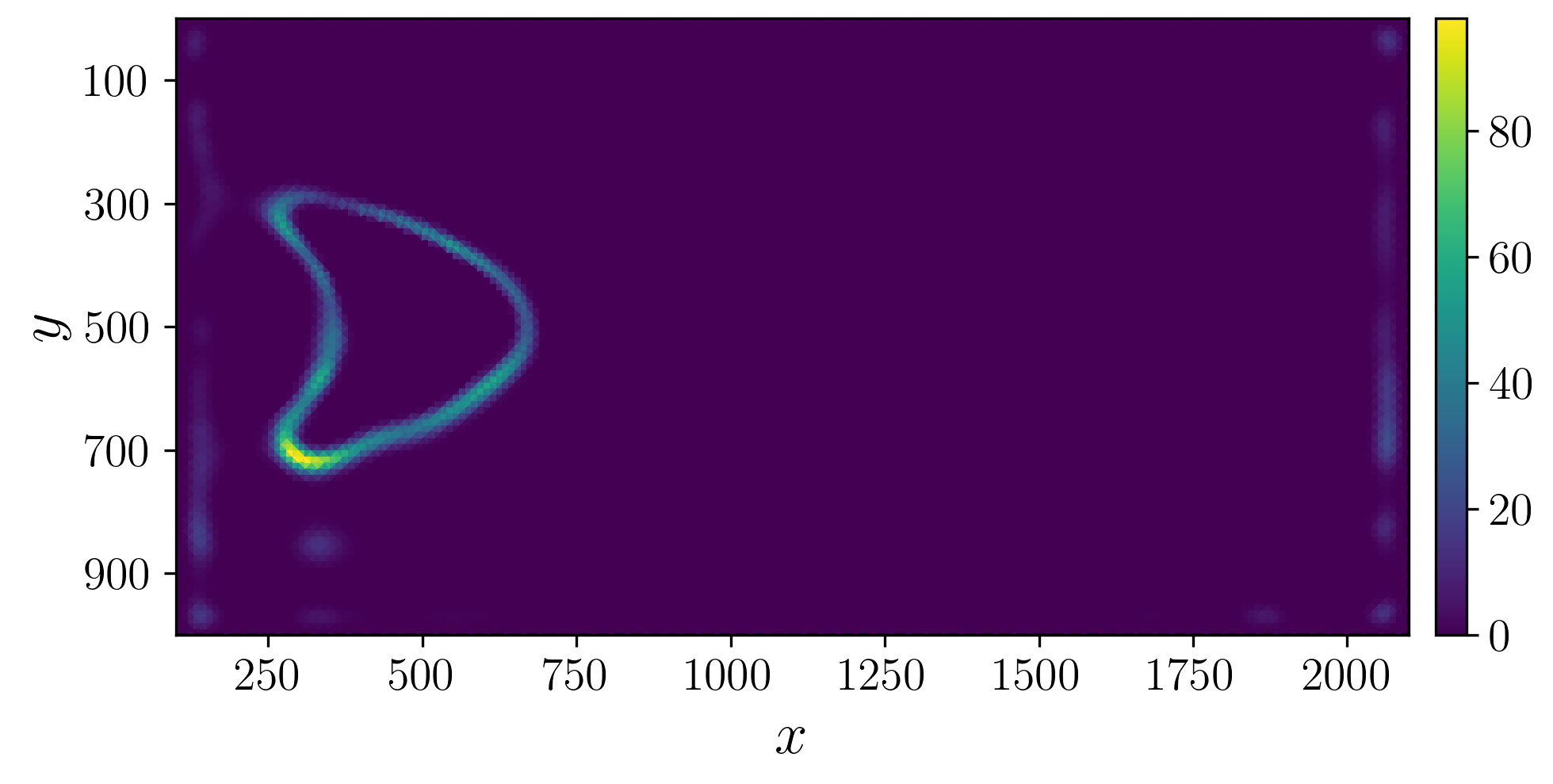}
\par\smallskip{\small (b) $\sigma_{\mathrm{noise}}=1\times10^{-4}$.\par}
\end{minipage}
\par\medskip
\begin{minipage}[t]{0.33\linewidth}
\centering
\includegraphics[width=\linewidth]{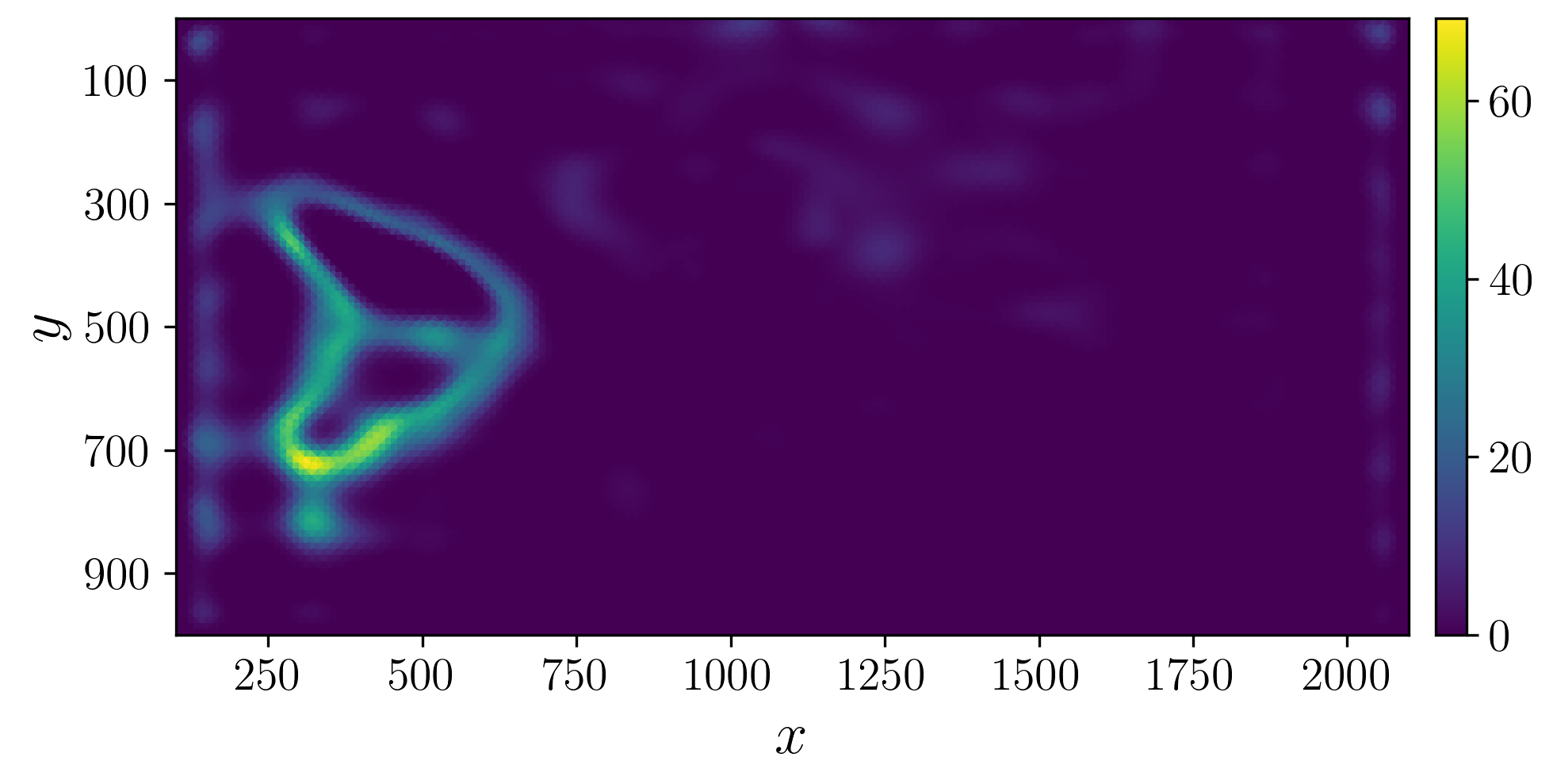}
\par\smallskip{\small (c) $\sigma_{\mathrm{noise}}=5\times10^{-4}$.\par}
\end{minipage}
\begin{minipage}[t]{0.33\linewidth}
\centering
\includegraphics[width=\linewidth]{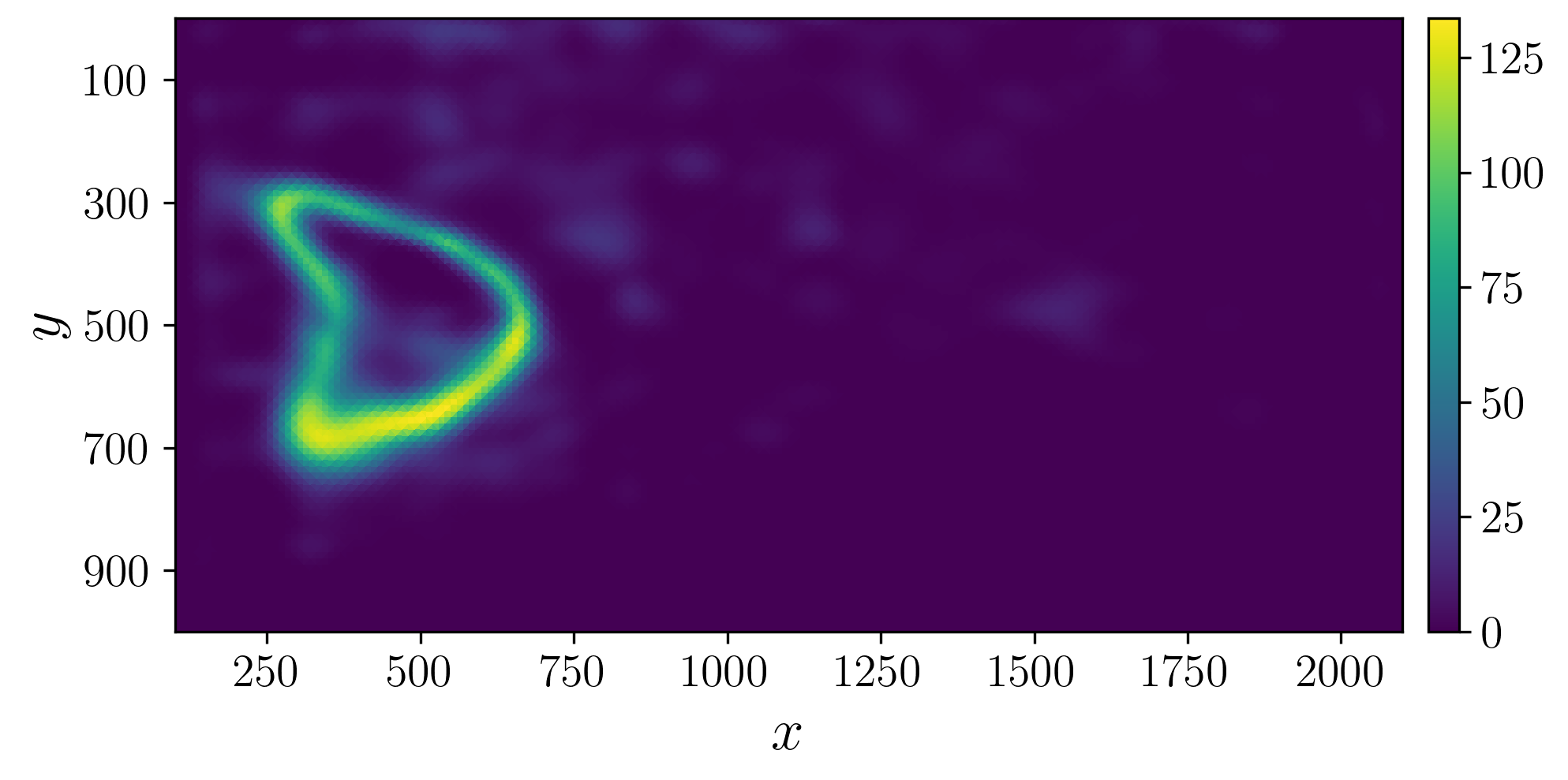}
\par\smallskip{\small (d) $\sigma_{\mathrm{noise}}=1\times10^{-3}$.\par}
\end{minipage}
\caption{Pointwise posterior standard deviations of velocity at four levels of observation noise, estimated from 100 samples of the final
Gaussian-mixture approximation in each run.}
\label{fig:ex2-noise-uncertainty}
\end{figure}

}
\end{example}

\begin{example}\label{example2}\rm{
This experiment tests the recovery of disconnected inclusions from
single-inclusion initializations. The true anomaly consists of a
nonconvex kite, a heart, and a rectangle. All computational, acquisition,
noise, RBF, and prior settings are unchanged from
Example~\ref{example1}.

Let $D_{\mathrm{kite}}$ be the kite-shaped region defined in
Example~\ref{example1}. The other two components are
\begin{equation*}
\begin{aligned}
D_{\mathrm{heart}}
&=
\left\{(x,y)\in\Omega:
\left(X_h^2+Y_h^2-1\right)^3-X_h^2Y_h^3\leq0
\right\},
\quad X_h=\frac{x-1140}{240},
\quad
Y_h=\frac{420-y}{160},\\
D_{\mathrm{rect}}
&=[1540,1900]\times[300,580].
\end{aligned}
\end{equation*}
The exact anomaly region is
\begin{equation*}
D_{\mathrm{true}}
=
D_{\mathrm{kite}}
\cup D_{\mathrm{heart}}
\cup D_{\mathrm{rect}}.
\end{equation*}
As in Example~\ref{example1}, the velocity is
$1500\,\mathrm{m/s}$ inside $D_{\mathrm{true}}$ and
$1000\,\mathrm{m/s}$ elsewhere, with
$\mu_{\mathrm{exact}}=v_{\mathrm{exact}}^{-2}$.

The same three fitted circular models are used as initializations, each containing only one connected component. The target therefore requires the reconstructed level set to split into three separated components. Frequency continuation uses cut-offs of $2$, $4$, and $6\,\mathrm{Hz}$; all other optimization and mode-propagation settings remain unchanged.

Figure \ref{fig:example2-frequency-stages} illustrates how the reconstructed geometry evolves as the frequency cut-off increases. At 0–2 Hz, the data primarily recover the inclusions’ approximate locations and overall sizes, while their boundaries remain poorly resolved and spurious patches are visible. Increasing the cut-off to 4 Hz separates the three inclusions more clearly and reduces these artifacts. At 0–6 Hz, the recovered boundaries more closely match the true shapes, with finer geometric details resolved. This progression shows that low-frequency data capture the large-scale structure, while adding higher-frequency data improves the reconstruction of boundary details.

Four Gaussian components are retained at the final frequency level, with normalized weights
\begin{equation*}
(\pi_1,\pi_2,\pi_3,\pi_4)
\approx(0.595851,\ 0.397877,\ 0.003138,\ 0.003133).
\end{equation*}
The first two components together account for $99.37\%$ of the total weight. The posterior means and pointwise standard deviations reported below are computed from $100$ samples drawn from this mixture.

\begin{figure}[H]
\centering

\begin{minipage}[t]{0.33\textwidth}
\centering
\includegraphics[width=\linewidth]{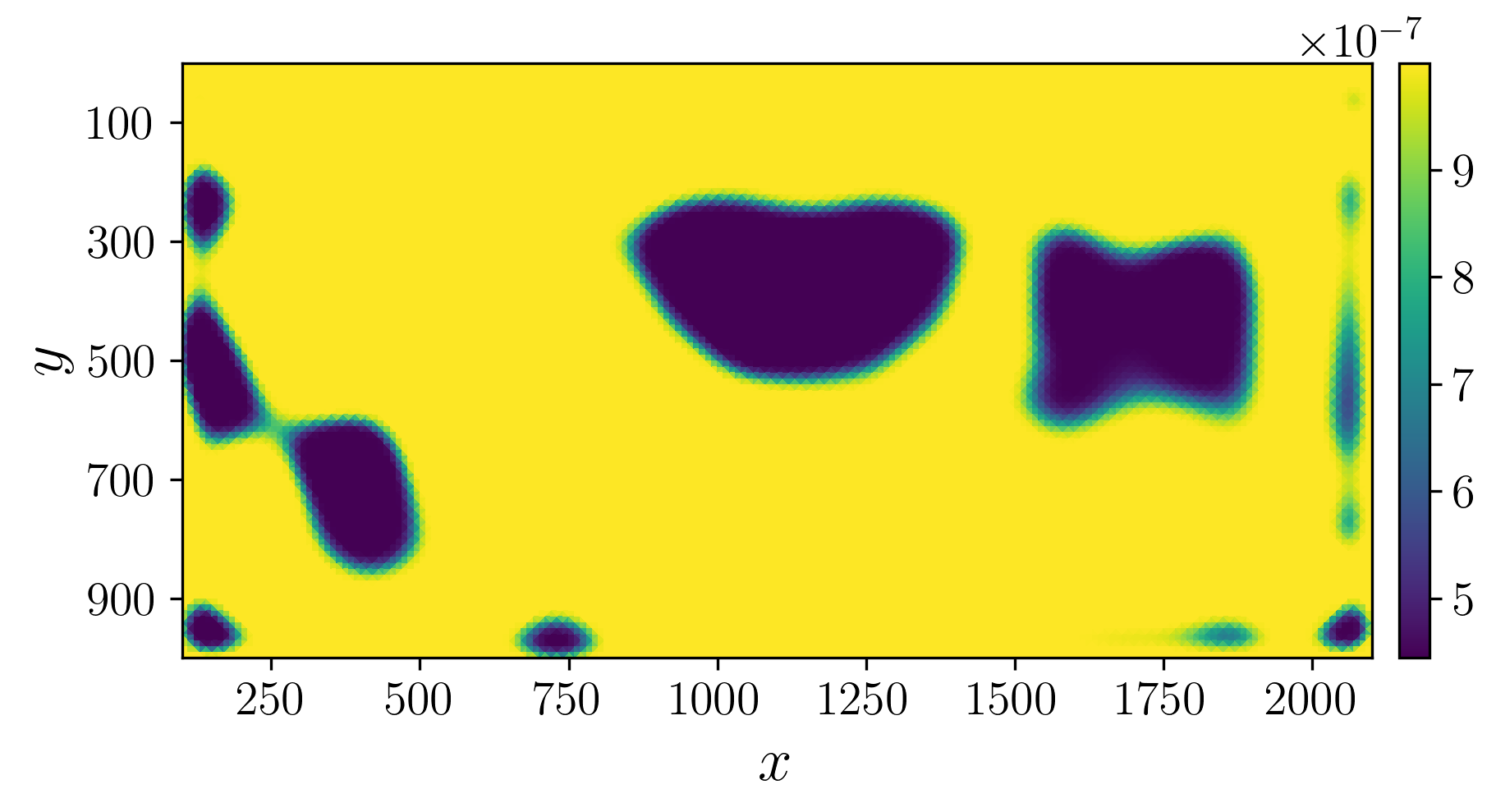}
\par\smallskip{\small (a) Squared slowness at $0$--$2\,\mathrm{Hz}$.\par}
\end{minipage}\hfill
\begin{minipage}[t]{0.33\textwidth}
\centering
\includegraphics[width=\linewidth]{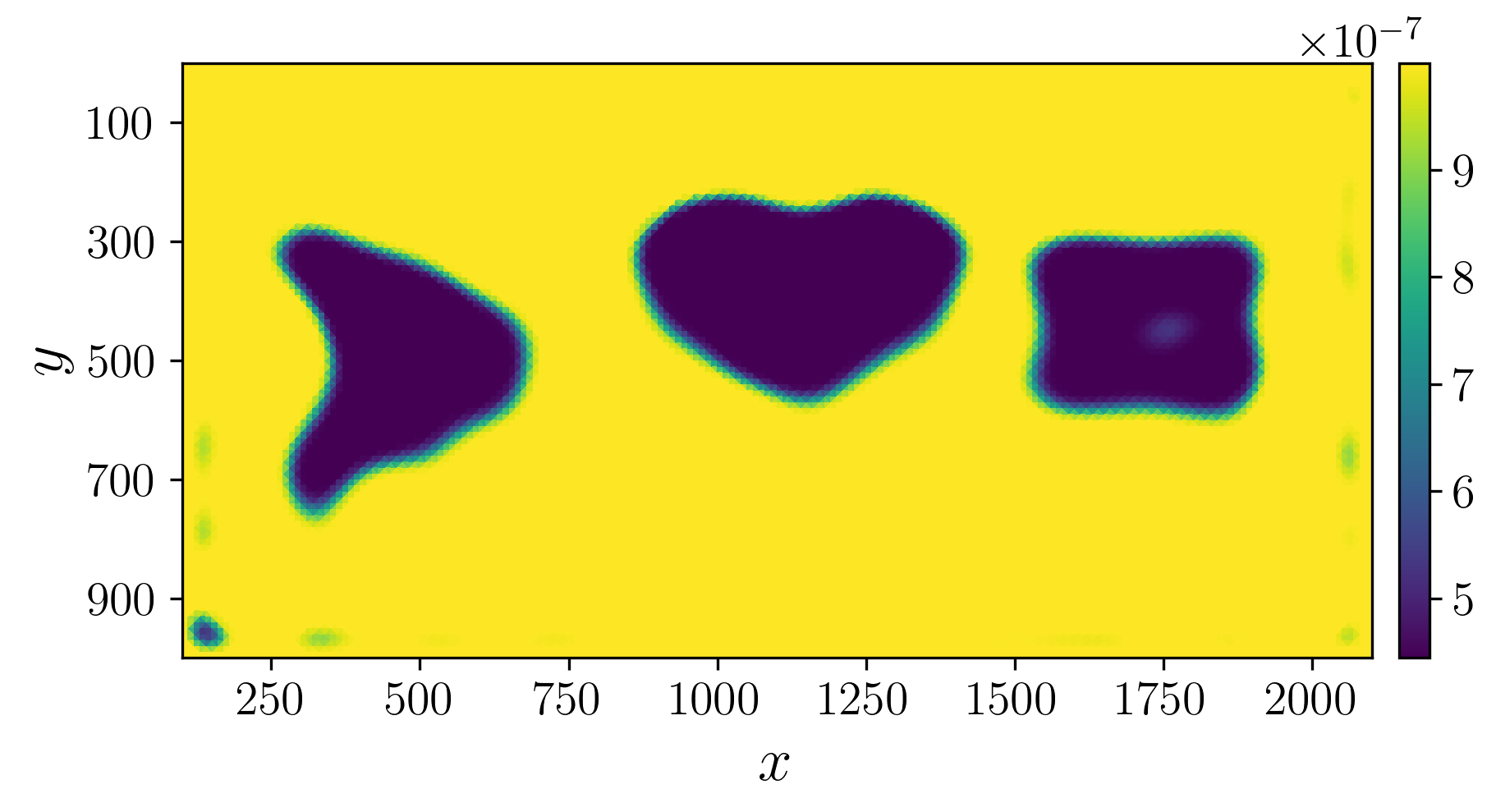}
\par\smallskip{\small (b) Squared slowness at $0$--$4\,\mathrm{Hz}$.\par}
\end{minipage}\hfill
\begin{minipage}[t]{0.33\textwidth}
\centering
\includegraphics[width=\linewidth]{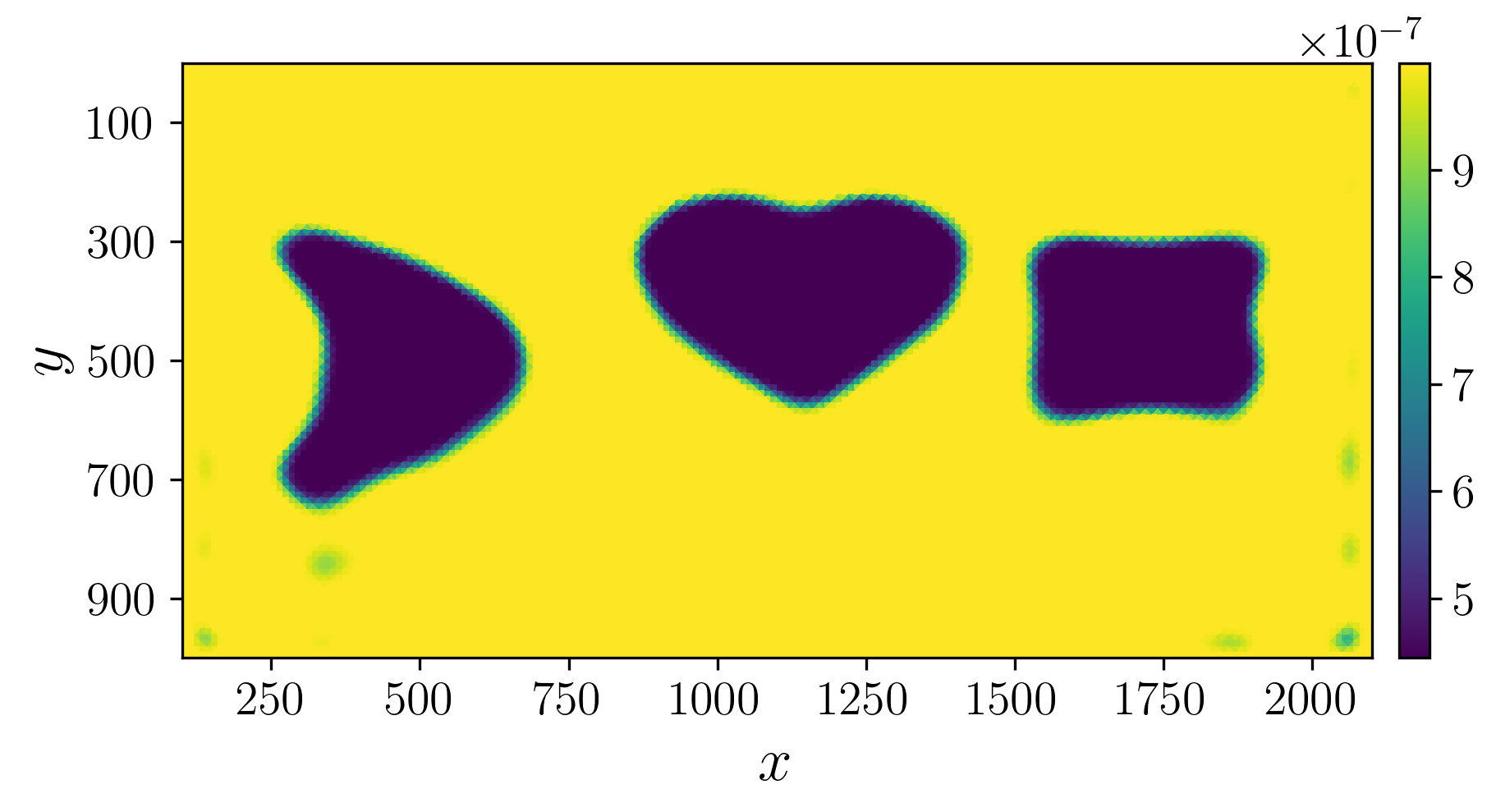}
\par\smallskip{\small (c) Squared slowness at $0$--$6\,\mathrm{Hz}$.\par}
\end{minipage}

\par\medskip

\begin{minipage}[t]{0.33\textwidth}
\centering
\includegraphics[width=\linewidth]{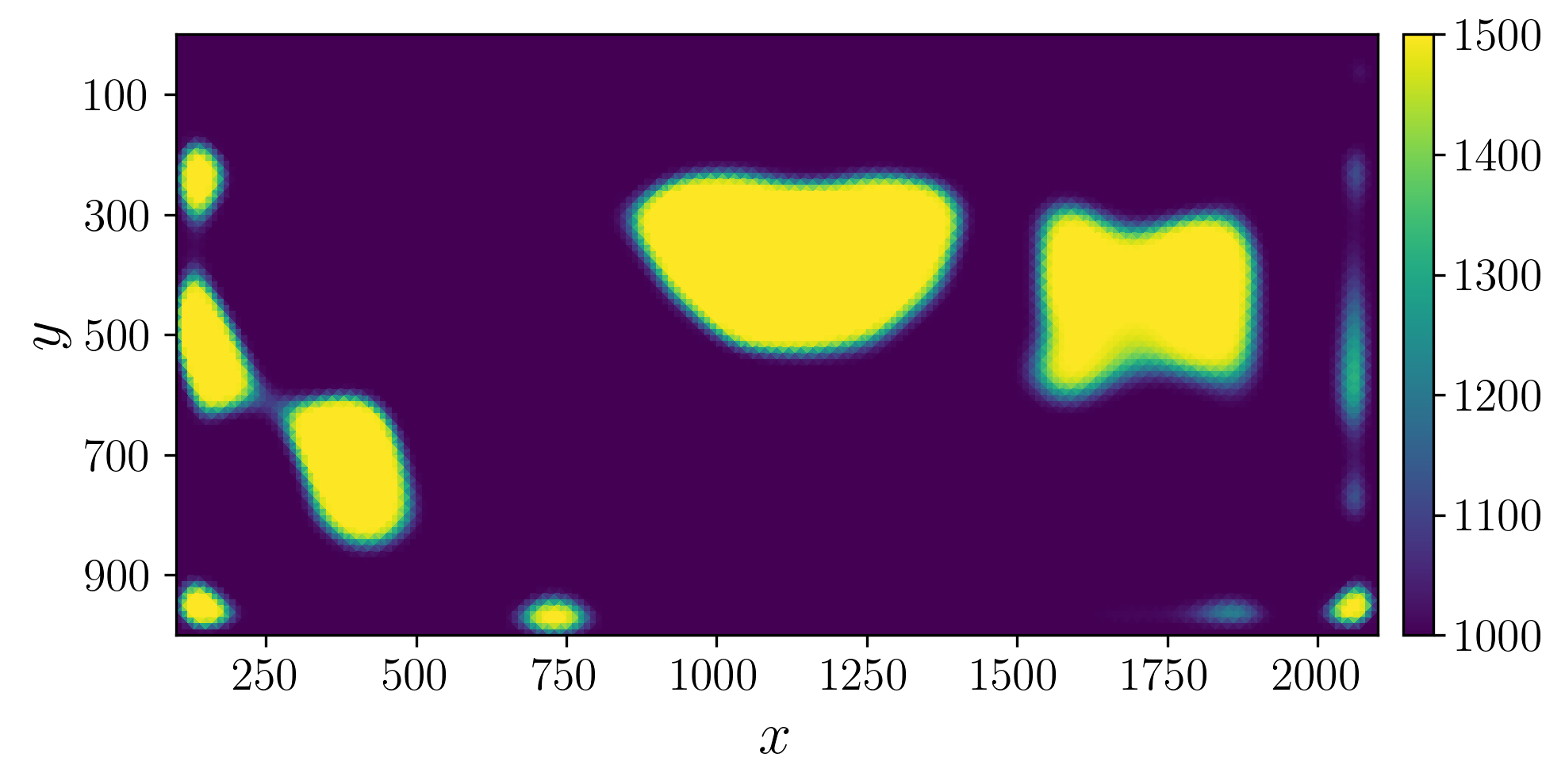}
\par\smallskip{\small (d) Velocity at $0$--$2\,\mathrm{Hz}$.\par}
\end{minipage}\hfill
\begin{minipage}[t]{0.33\textwidth}
\centering
\includegraphics[width=\linewidth]{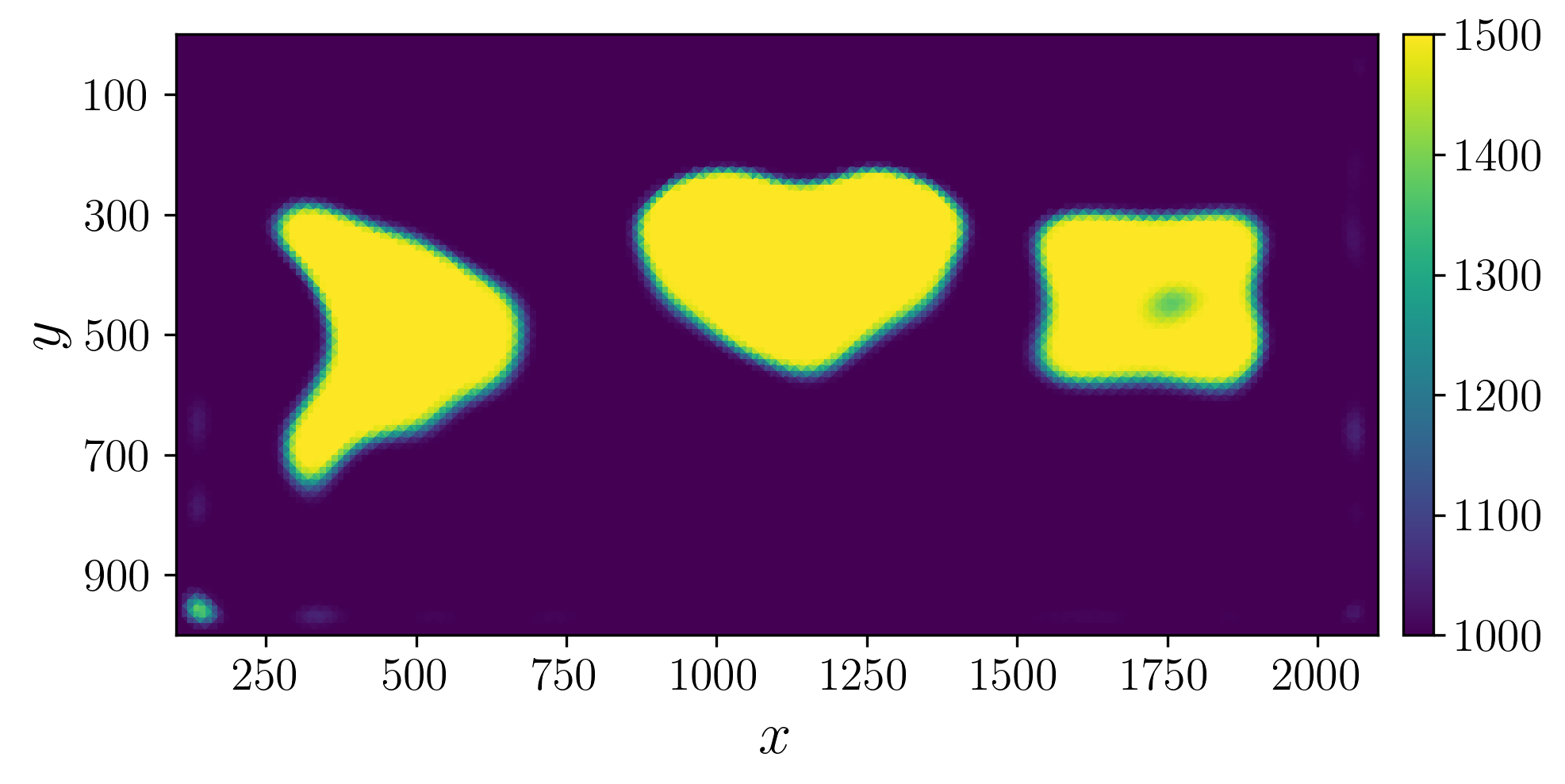}
\par\smallskip{\small (e) Velocity at $0$--$4\,\mathrm{Hz}$.\par}
\end{minipage}\hfill
\begin{minipage}[t]{0.33\textwidth}
\centering
\includegraphics[width=\linewidth]{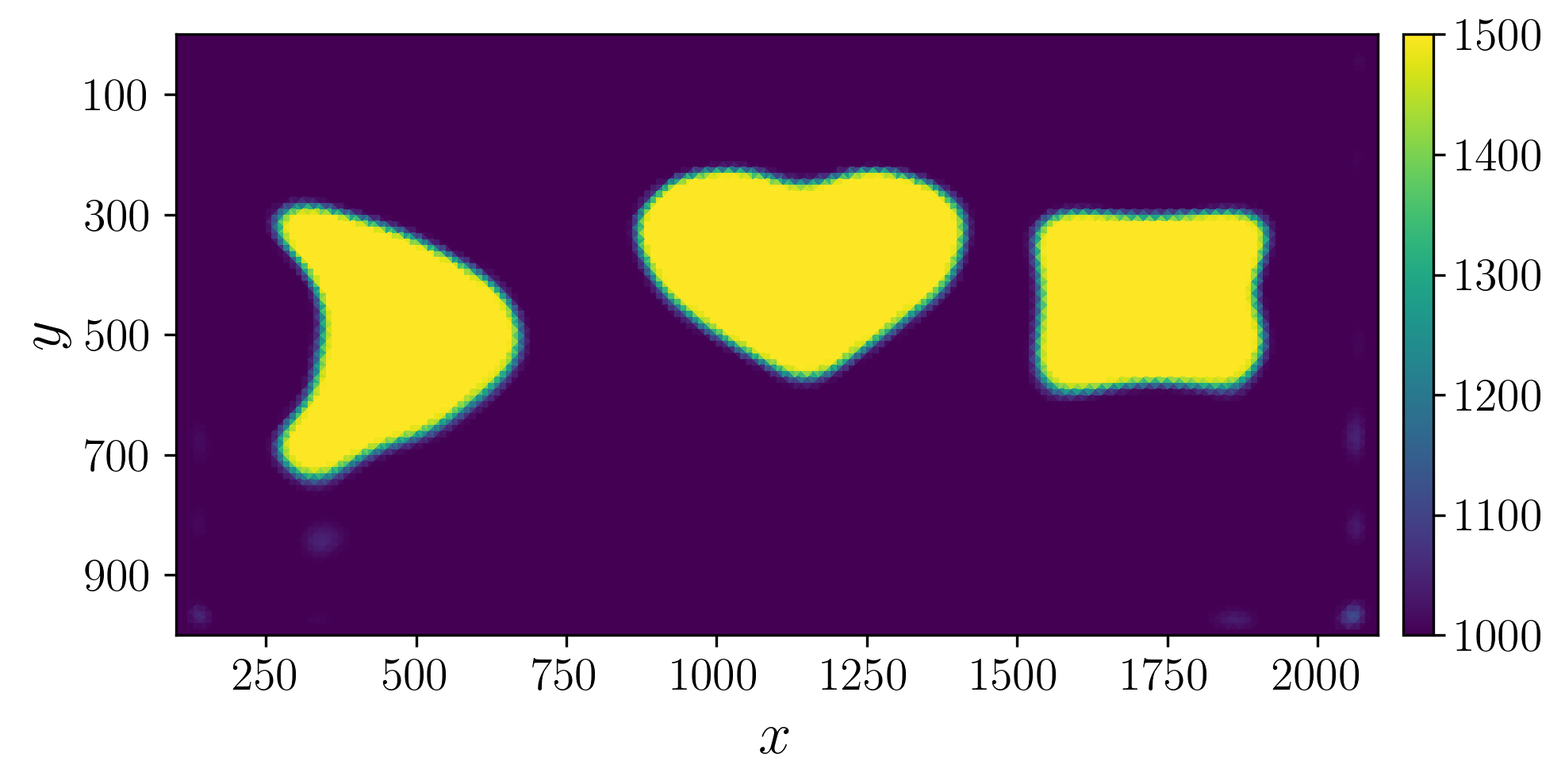}
\par\smallskip{\small (f) Velocity at $0$--$6\,\mathrm{Hz}$.\par}
\end{minipage}

\caption{Best retained MAP reconstructions at the three frequency levels in Example~\ref{example2}. The top and bottom rows show squared slowness and velocity, respectively; the columns correspond to $0$--$2$, $0$--$4$, and $0$--$6\,\mathrm{Hz}$.}
\label{fig:example2-frequency-stages}
\end{figure}

Figure~\ref{fig:example2-reconstruction} compares the true
squared-slowness and velocity fields with the corresponding best-MAP
and posterior-mean reconstructions. Although every
initial model contains only one circle, the reconstruction separates
into three distinct inclusions. Their locations and principal geometric
features are recovered, including the concave side of the kite and the
overall heart and rectangular shapes. The upper indentation of the
heart and the corners of the rectangle are smoothed, and the rectangular
edges are slightly curved. The mixture-mean squared-slowness and velocity fields are essentially identical to their MAP counterparts.

Figure~\ref{fig:example2-uncertainty} shows the corresponding pointwise
standard deviations. Variability is concentrated near the reconstructed
interfaces, most strongly along the concave side of the kite and parts
of the heart boundary. The rectangular boundary has comparatively lower
variability, while the inclusion interiors and most of the background
remain stable.

\begin{figure}[H]
\centering

\begin{minipage}[t]{0.33\textwidth}
\centering
\includegraphics[width=\linewidth]{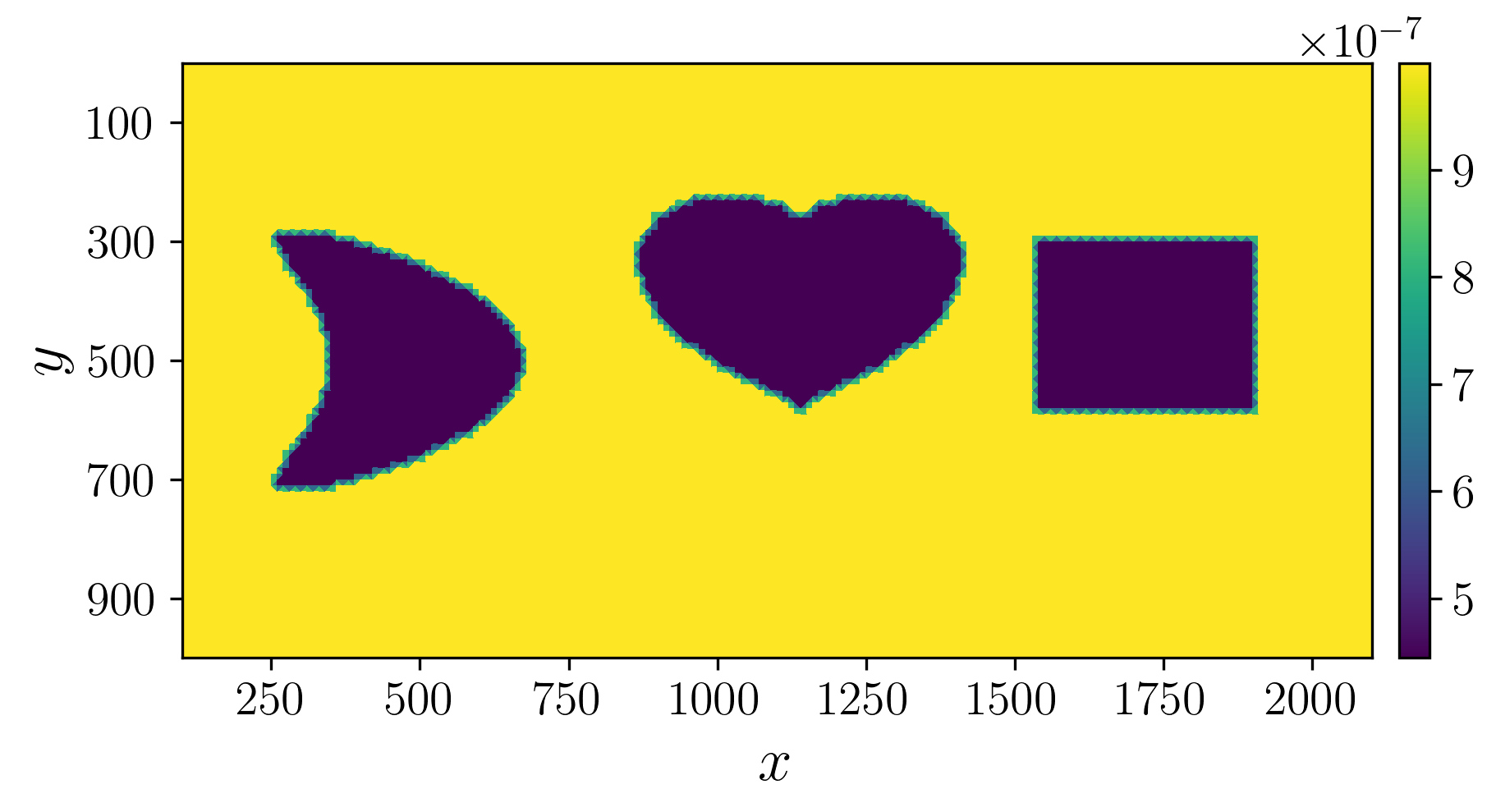}
\par\smallskip{\small (a) Exact squared slowness.\par}
\end{minipage}\hfill
\begin{minipage}[t]{0.33\textwidth}
\centering
\includegraphics[width=\linewidth]{Example3/Best_final_MAP_mux2.png}
\par\smallskip{\small (b) Best retained MAP squared slowness.\par}
\end{minipage}\hfill
\begin{minipage}[t]{0.33\textwidth}
\centering
\includegraphics[width=\linewidth]{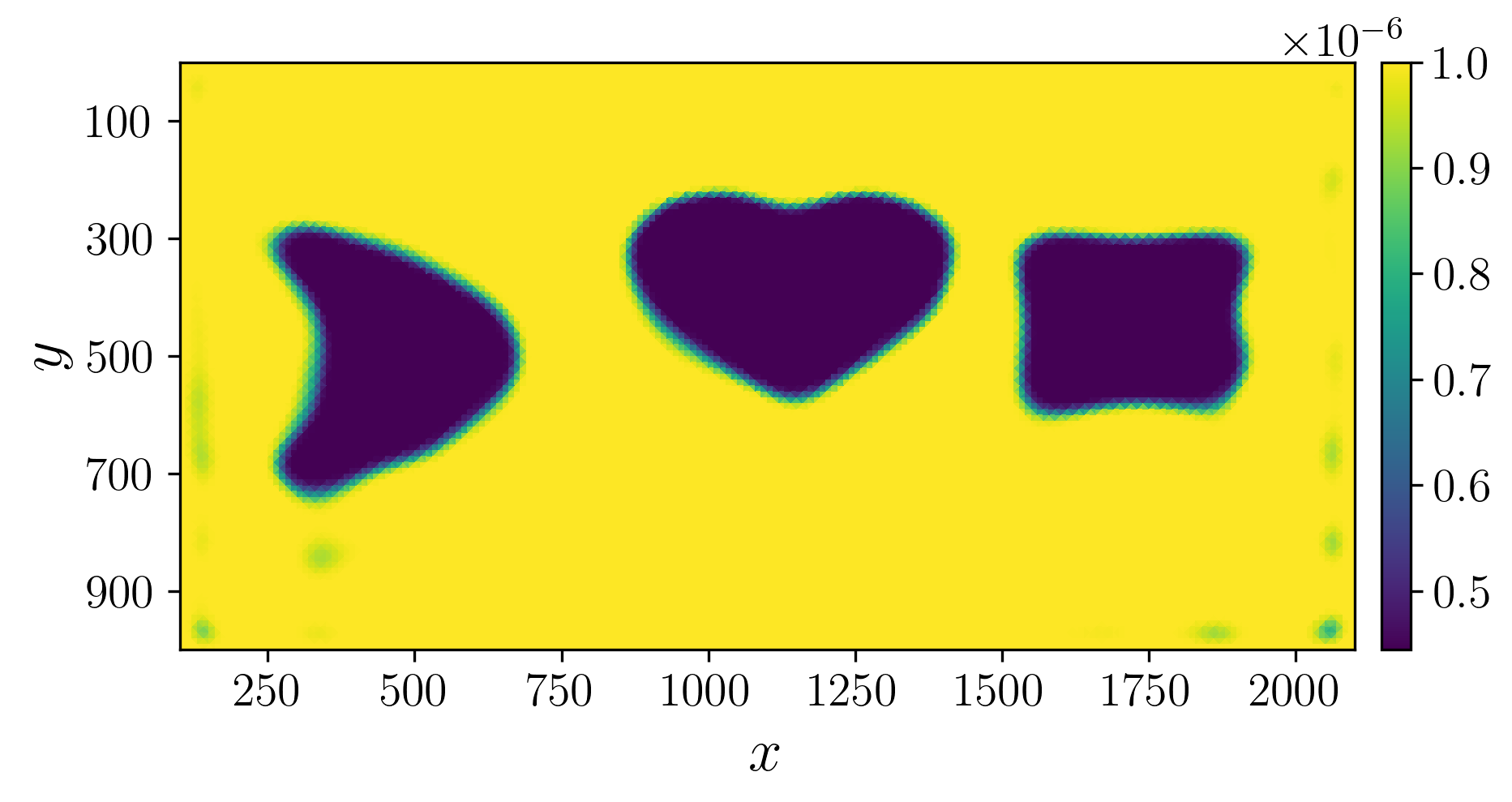}
\par\smallskip{\small (c) Mixture mean squared slowness.\par}
\end{minipage}

\par\medskip

\begin{minipage}[t]{0.33\textwidth}
\centering
\includegraphics[width=\linewidth]{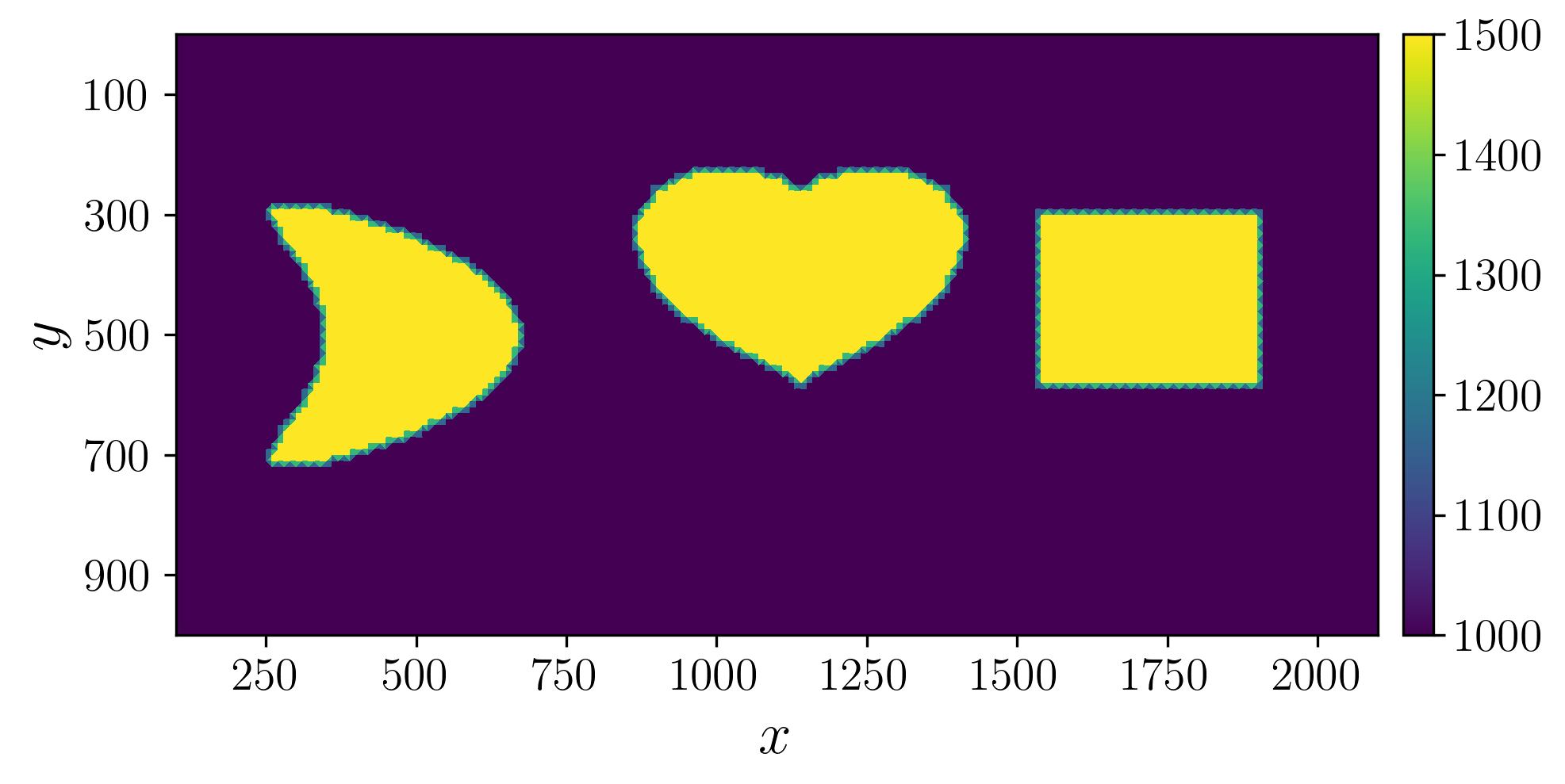}
\par\smallskip{\small (d) Exact velocity.\par}
\end{minipage}\hfill
\begin{minipage}[t]{0.33\textwidth}
\centering
\includegraphics[width=\linewidth]{Example3/Best_final_MAP_velx2.png}
\par\smallskip{\small (e) Best retained MAP velocity.\par}
\end{minipage}\hfill
\begin{minipage}[t]{0.33\textwidth}
\centering
\includegraphics[width=\linewidth]{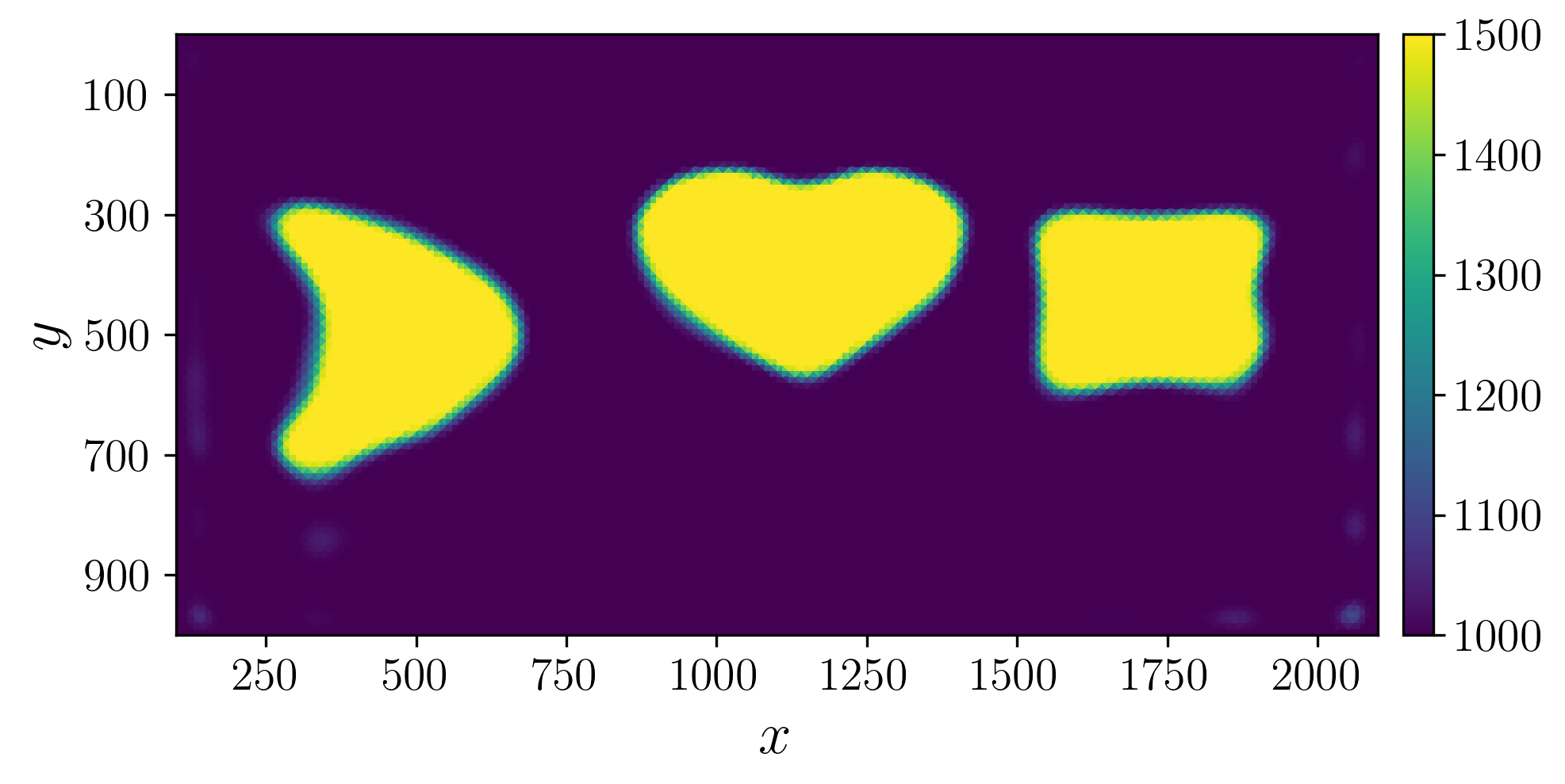}
\par\smallskip{\small (f) Mixture mean velocity.\par}
\end{minipage}

\caption{Exact, best retained MAP, and mixture mean fields for squared slowness (top row) and velocity (bottom row) in Example~\ref{example2}.}
\label{fig:example2-reconstruction}
\end{figure}

\begin{figure}[H]
\centering
\begin{minipage}[t]{0.33\textwidth}
\centering
\includegraphics[width=\linewidth]{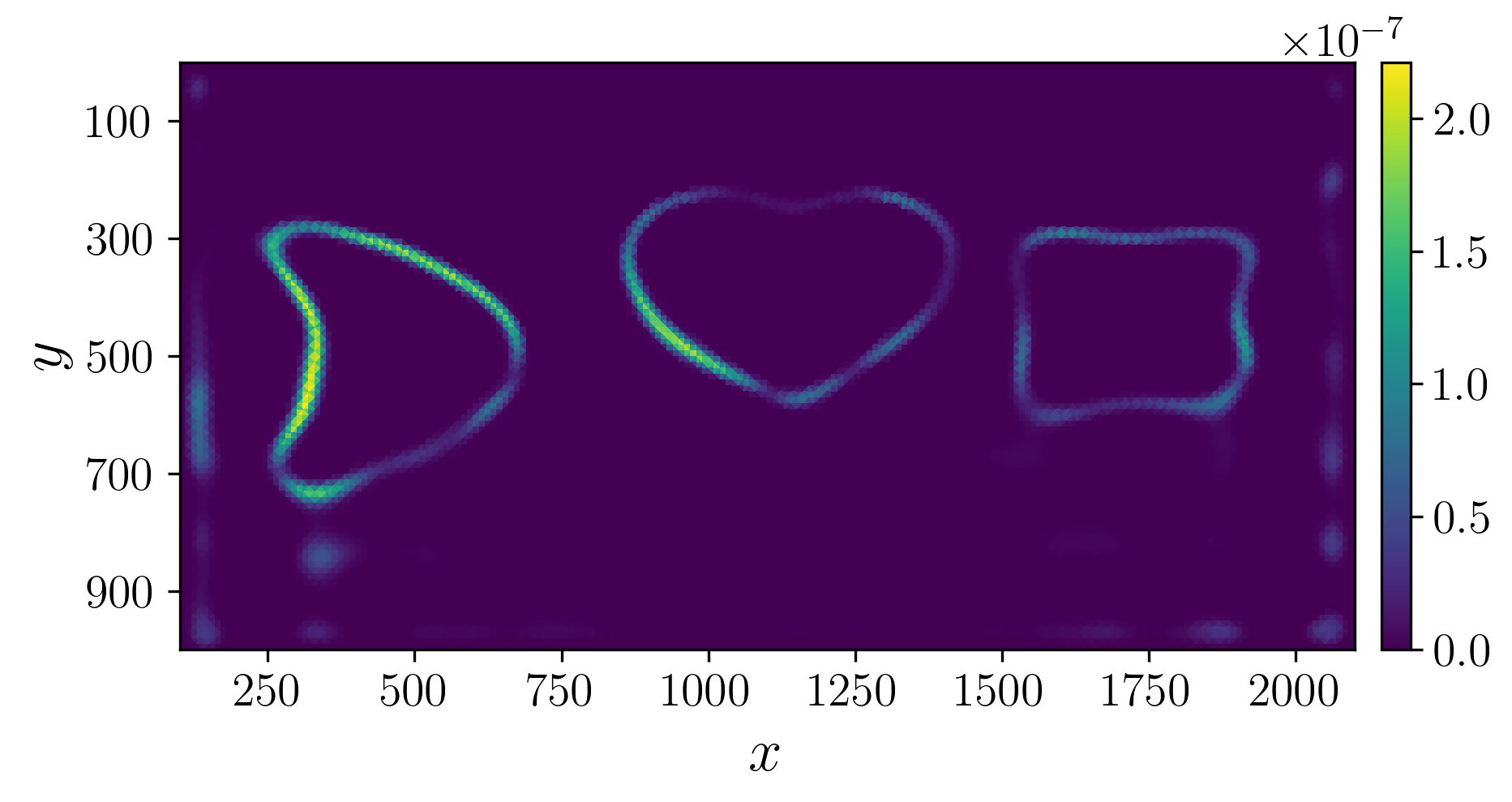}
\par\smallskip{\small (a) Squared slowness standard deviation.\par}
\end{minipage}
\begin{minipage}[t]{0.33\textwidth}
\centering
\includegraphics[width=\linewidth]{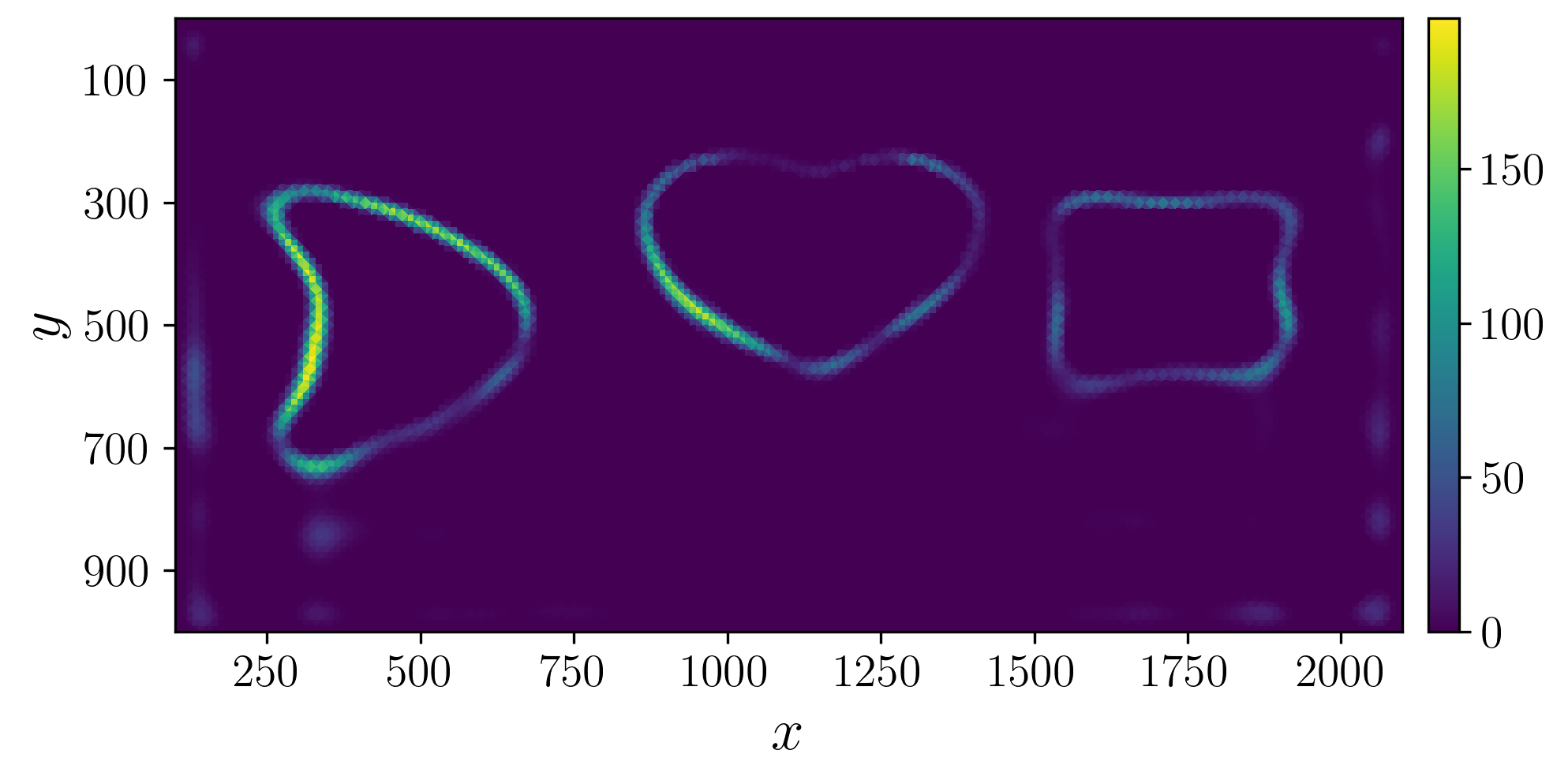}
\par\smallskip{\small (b) Velocity standard deviation.\par}
\end{minipage}
\caption{Pointwise standard deviations of squared slowness (a) and velocity (b) under the constructed Gaussian mixture. (Example~\ref{example2}).}
\label{fig:example2-uncertainty}
\end{figure}

We compare the proposed method with the standard single-component
Laplace approximation and gPCN-MCMC. The former provides the direct
baseline for assessing the multiscale multimodal extension, while gPCN
serves as a sampling-based reference. For the boundary uncertainty
metrics, all three methods use $r=20\,\mathrm m$, $\alpha=0.05$, and
$\eta=20\,\mathrm m$.

The gPCN chain is run for $100000$ iterations, with the first $20000$
discarded and every tenth subsequent state retained, yielding $8000$
samples. Of the proposed moves, $12498$ are accepted, corresponding to
an acceptance rate of $12.498\%$. The retained states are used to
compute the posterior means and standard deviations. Since no separate
MAP optimization is performed, the Hausdorff distance is evaluated
using the visited state with the smallest posterior objective.

Figure~\ref{fig:ex3-gpcn-comparison} shows that the best gPCN state
separates all three inclusions, although the interfaces and rectangular
corners remain smooth. The posterior means have broader transition
regions, particularly around the kite and rectangle. Their standard
deviations extend into the interiors of these two inclusions, whereas
the variability around the heart is more closely confined to its
boundary. Compared with the proposed method, the gPCN reconstruction is
more diffuse and its uncertainty is less sharply localized near the
interfaces.

For the single-component Laplace approximation, the same
$401$-coefficient representation, likelihood scale, and prior are used. The MAP candidate is computed using $1000$ Adam epochs with learning rate $1.0$.
Figure~\ref{fig:ex3-single-laplace} shows that the three inclusions are recovered and that the posterior means are essentially identical to the
MAP squared-slowness and velocity fields. Small spurious patches remain
near the ROI margins, especially along the lower boundary. The
standard deviations are comparatively small and are confined to
selected interface segments and the spurious patches, reflecting the
local character of the single-component approximation.

\begin{figure}[H]
\centering

\begin{minipage}[t]{0.33\linewidth}
\centering
\includegraphics[width=\linewidth]{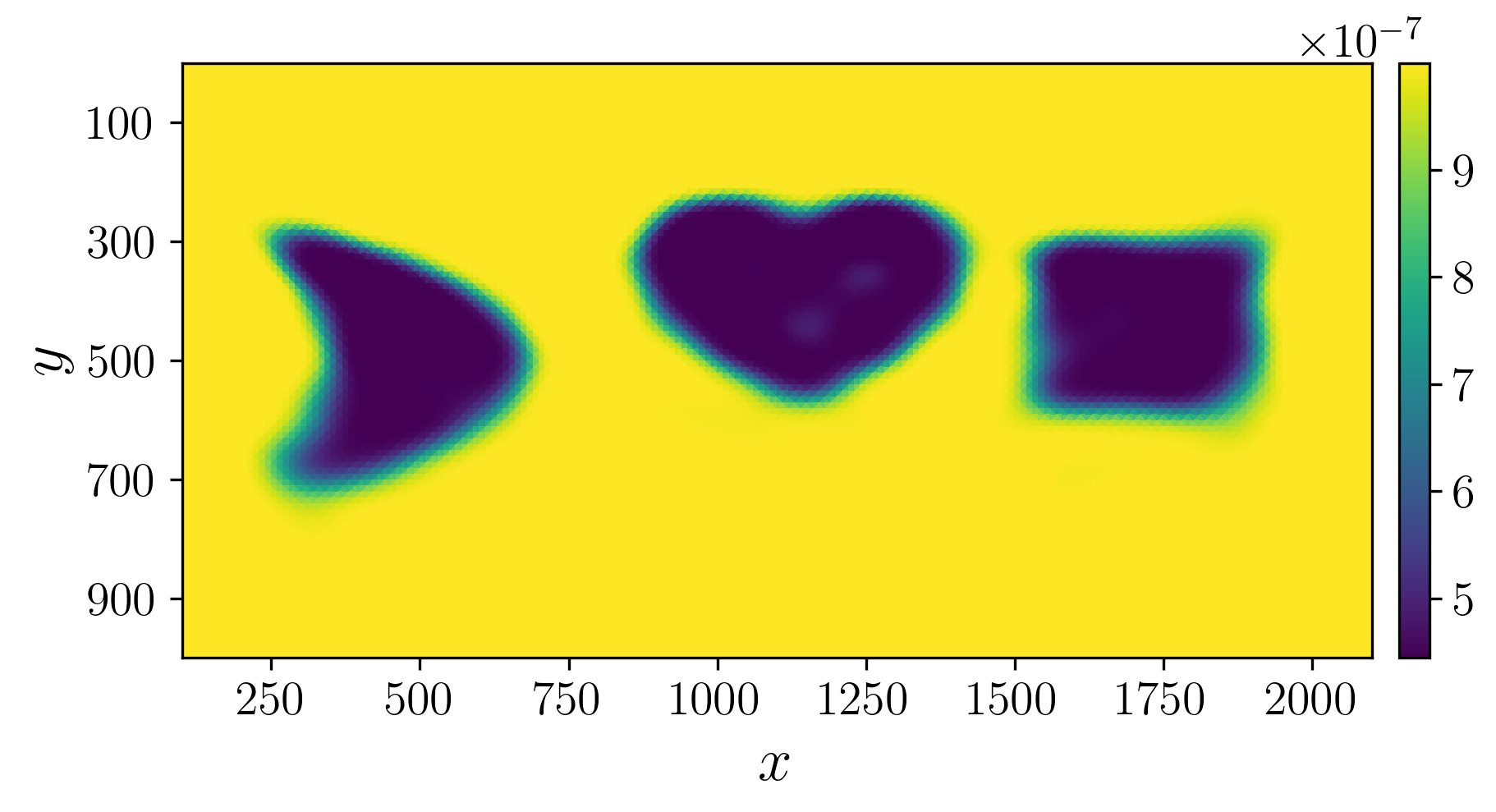}
\par\smallskip{\small (a) Loweat-objective squared slowness.\par}
\end{minipage}\hfill
\begin{minipage}[t]{0.33\linewidth}
\centering
\includegraphics[width=\linewidth]{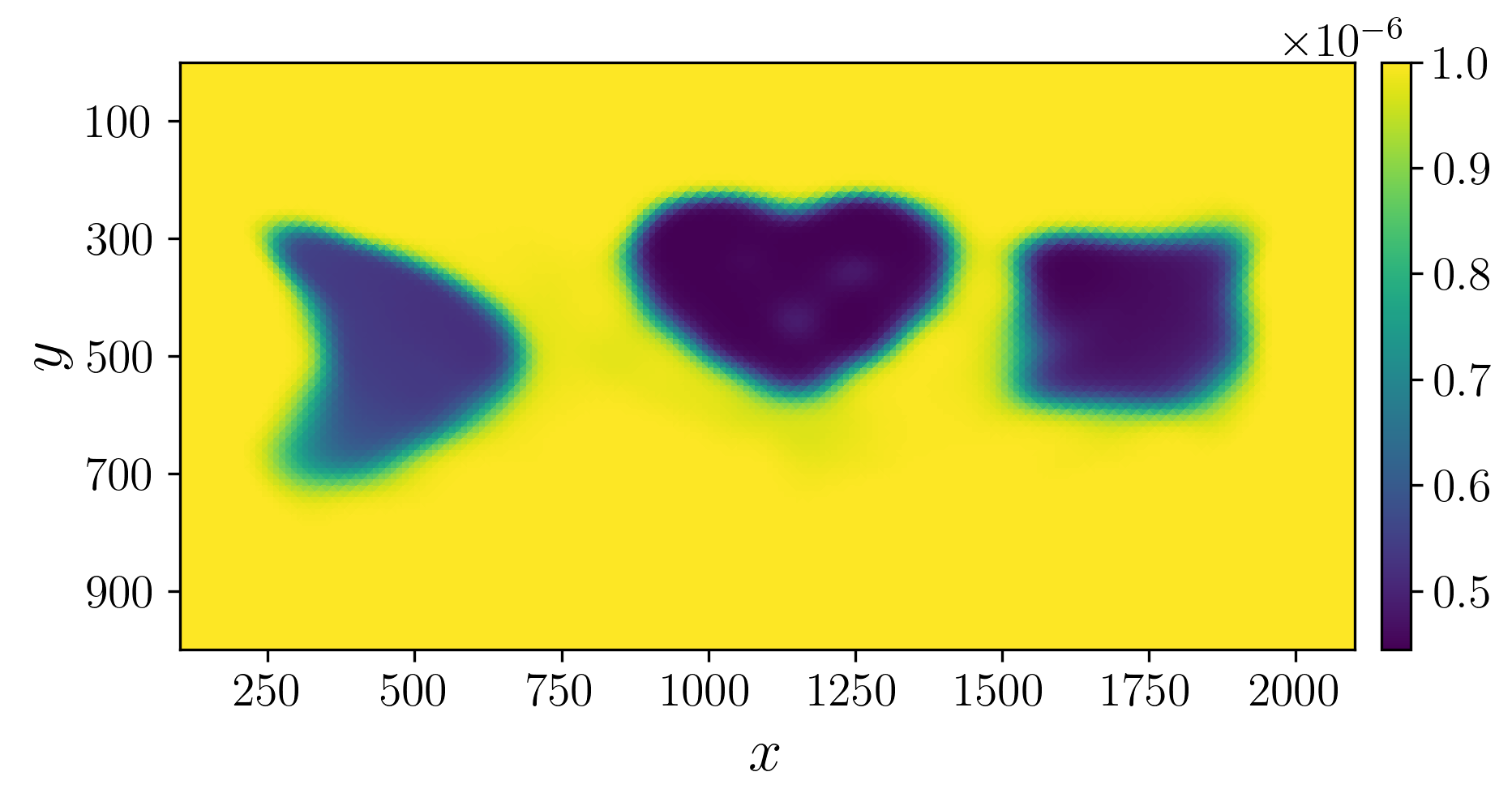}
\par\smallskip{\small (b) Posterior mean of squared slowness.\par}
\end{minipage}\hfill
\begin{minipage}[t]{0.33\linewidth}
\centering
\includegraphics[width=\linewidth]{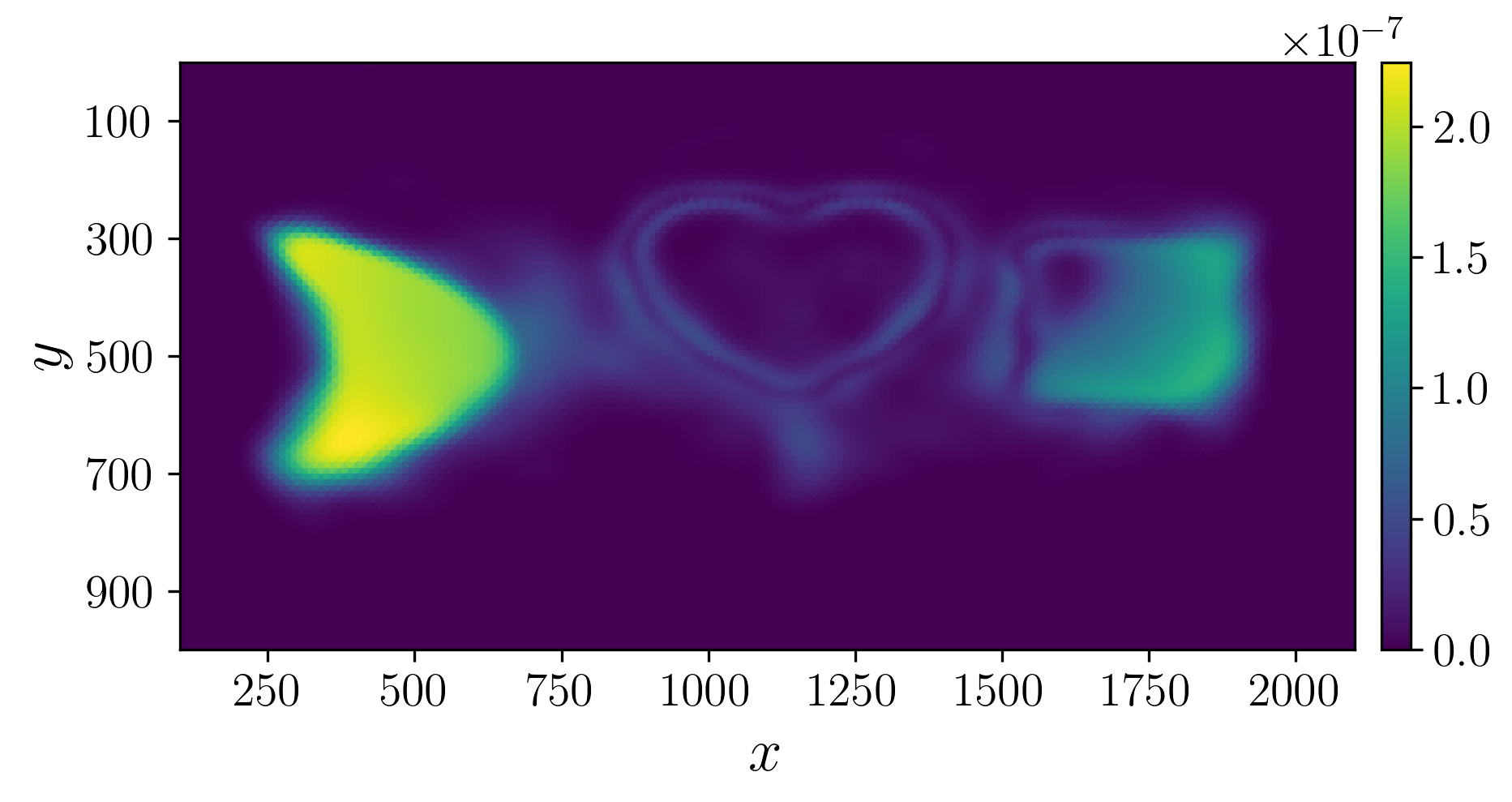}
\par\smallskip{\small (c) Posterior standard deviation of squared slowness.\par}
\end{minipage}

\par\medskip

\begin{minipage}[t]{0.33\linewidth}
\centering
\includegraphics[width=\linewidth]{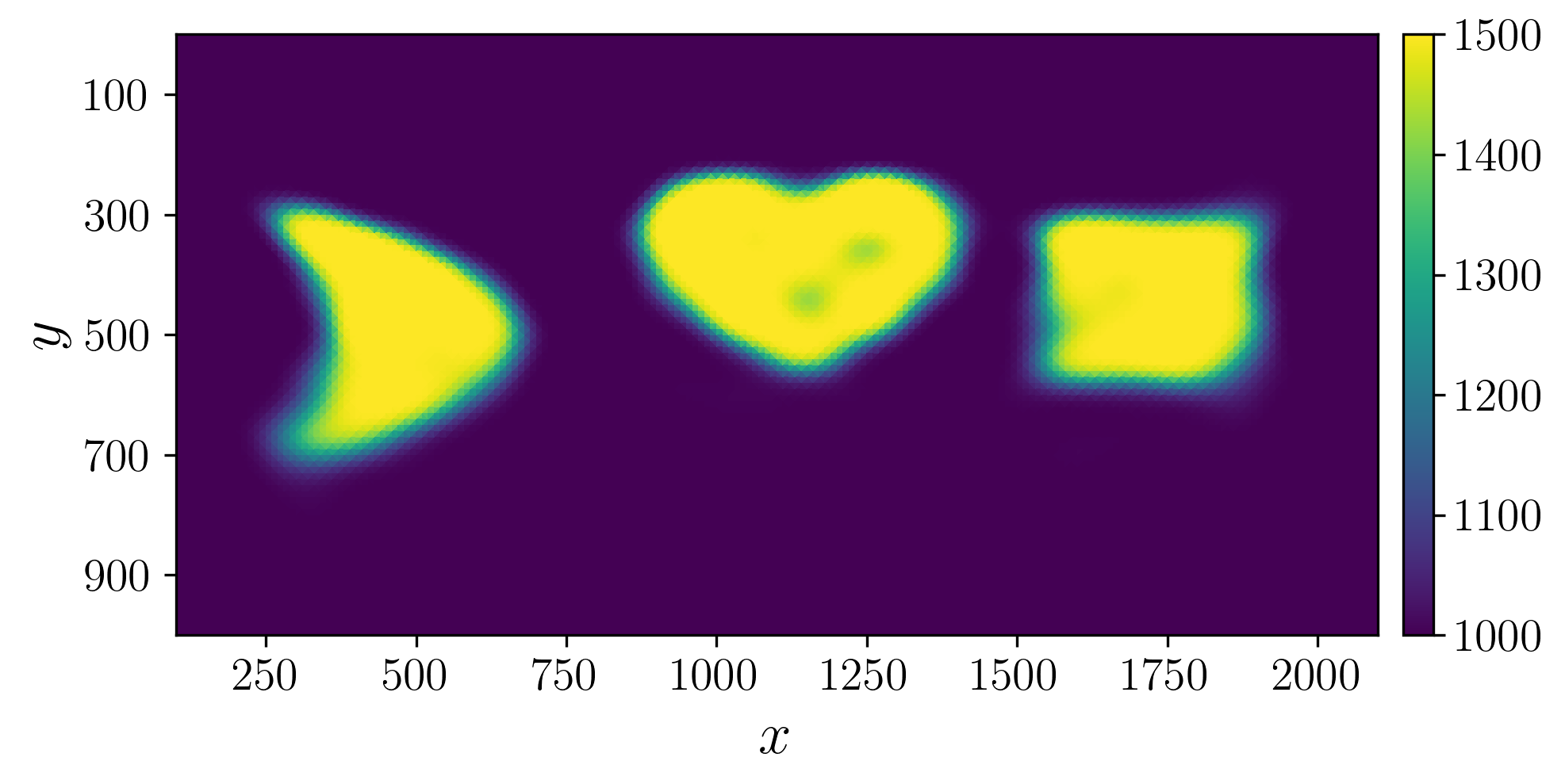}
\par\smallskip{\small (d) Loweat-objective velocity.\par}
\end{minipage}\hfill
\begin{minipage}[t]{0.33\linewidth}
\centering
\includegraphics[width=\linewidth]{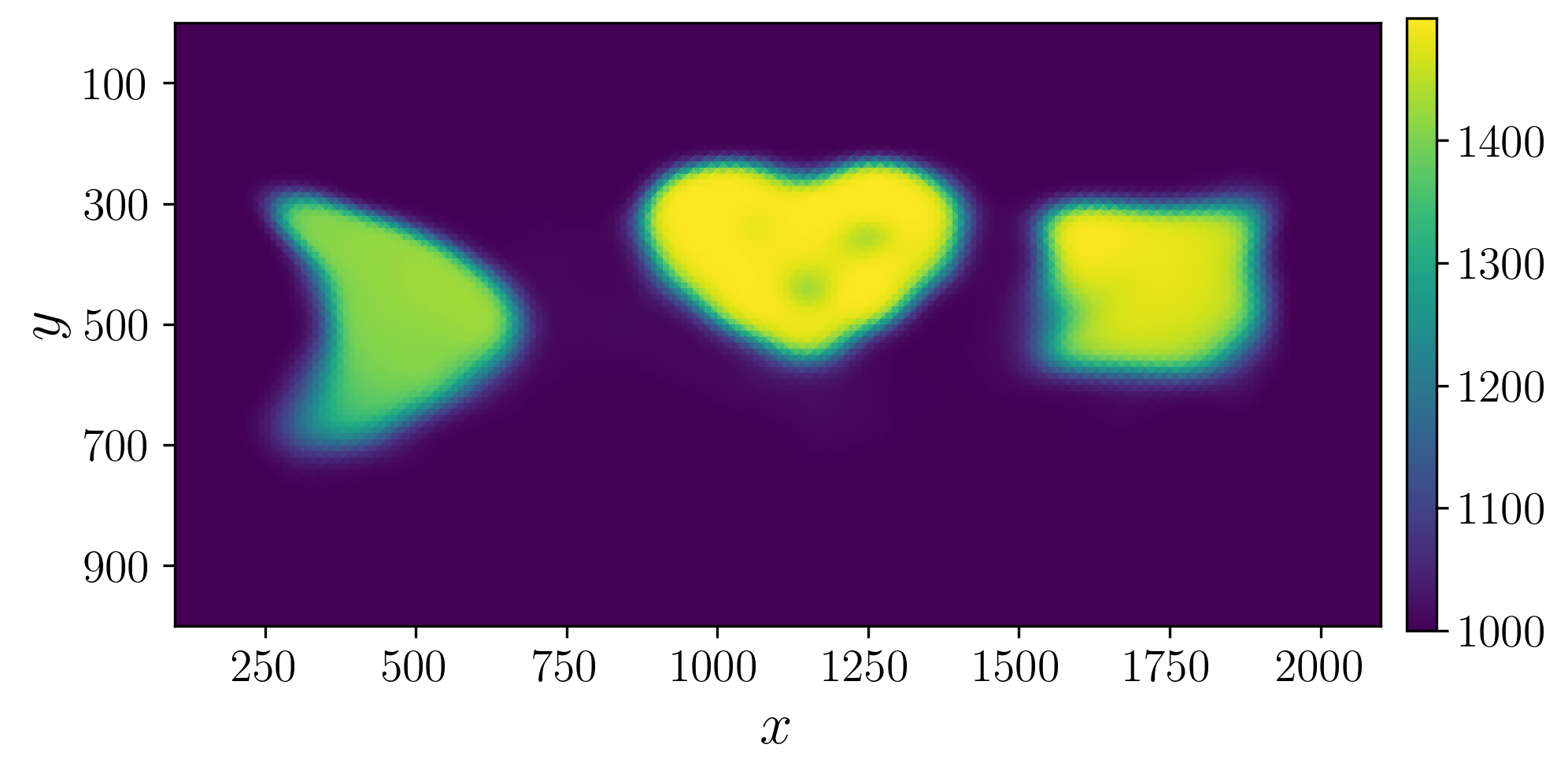}
\par\smallskip{\small (e) Posterior mean of velocity.\par}
\end{minipage}\hfill
\begin{minipage}[t]{0.33\linewidth}
\centering
\includegraphics[width=\linewidth]{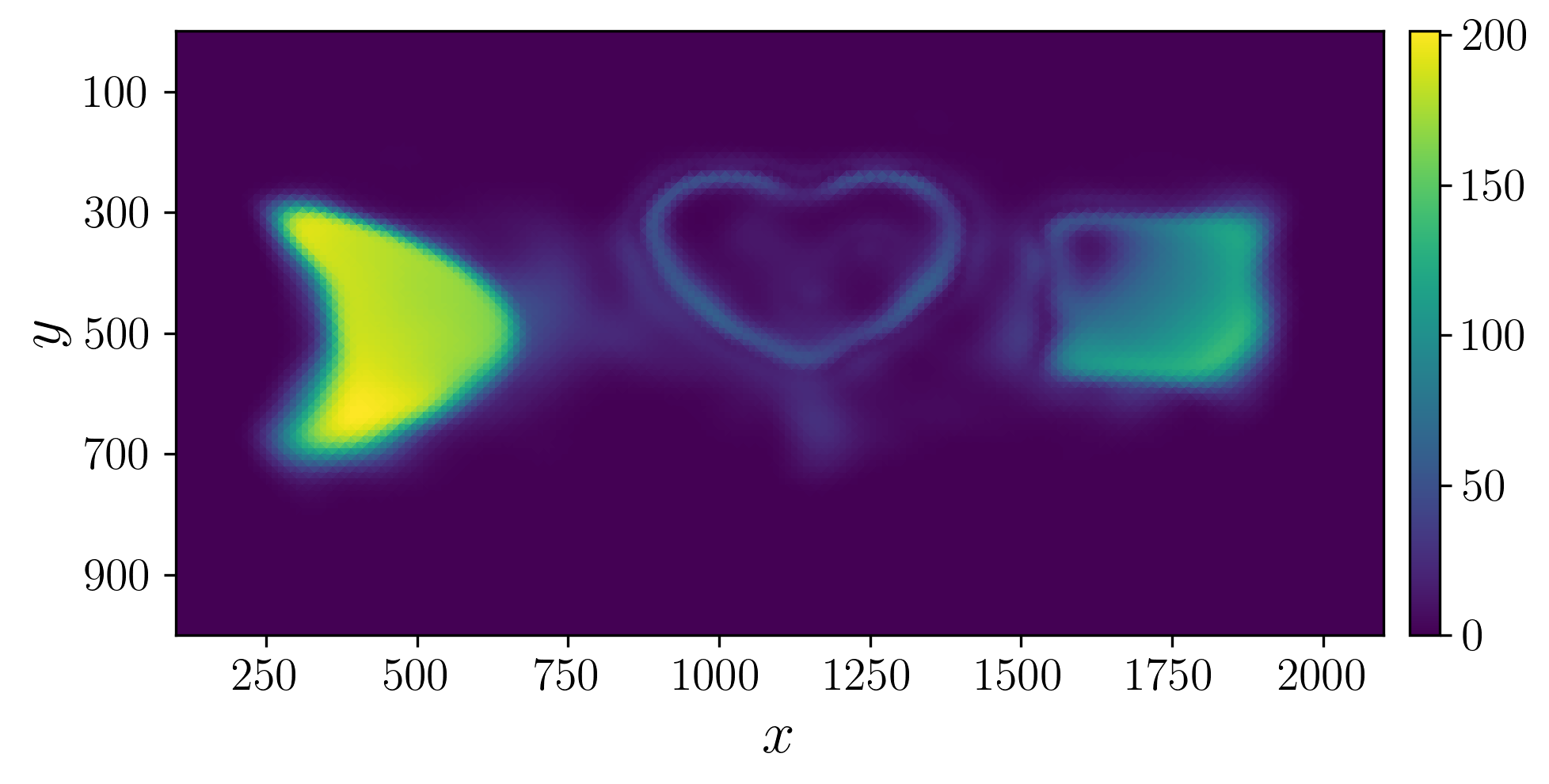}
\par\smallskip{\small (f) Posterior standard deviation of velocity.\par}
\end{minipage}

\caption{gPCN-MCMC results for the three-inclusion experiment (Example~\ref{example2}). The top and bottom rows show squared slowness and velocity, respectively. From left to right: the retained state with the smallest posterior objective, posterior mean, and pointwise posterior standard deviation.}
\label{fig:ex3-gpcn-comparison}
\end{figure}

\begin{figure}[H]
\centering

\begin{minipage}[t]{0.33\linewidth}
\centering
\includegraphics[width=\linewidth]{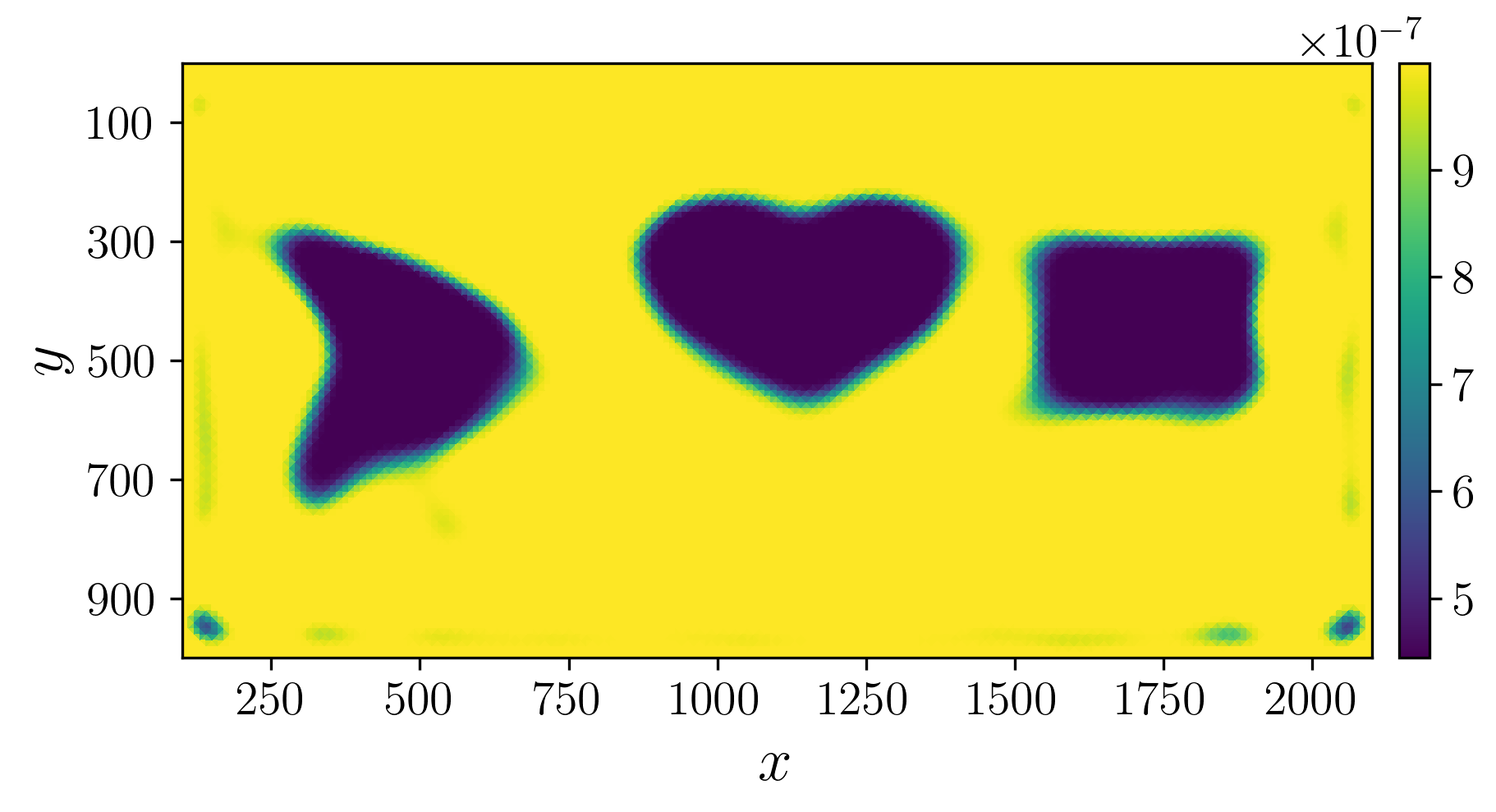}
\par\smallskip{\small (a) MAP-candidate squared slowness.\par}
\end{minipage}\hfill
\begin{minipage}[t]{0.33\linewidth}
\centering
\includegraphics[width=\linewidth]{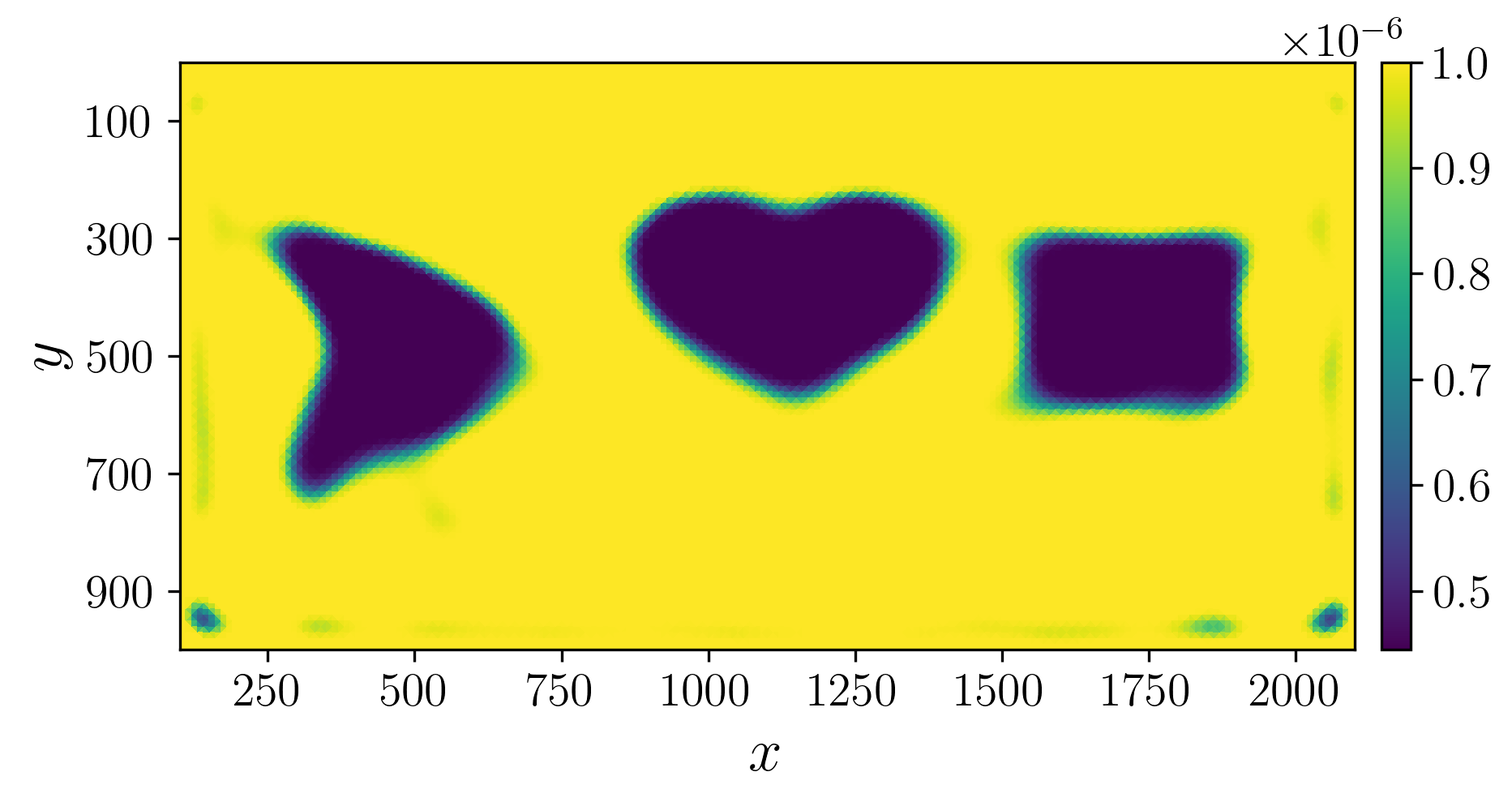}
\par\smallskip{\small (b) Posterior mean of squared slowness.\par}
\end{minipage}\hfill
\begin{minipage}[t]{0.33\linewidth}
\centering
\includegraphics[width=\linewidth]{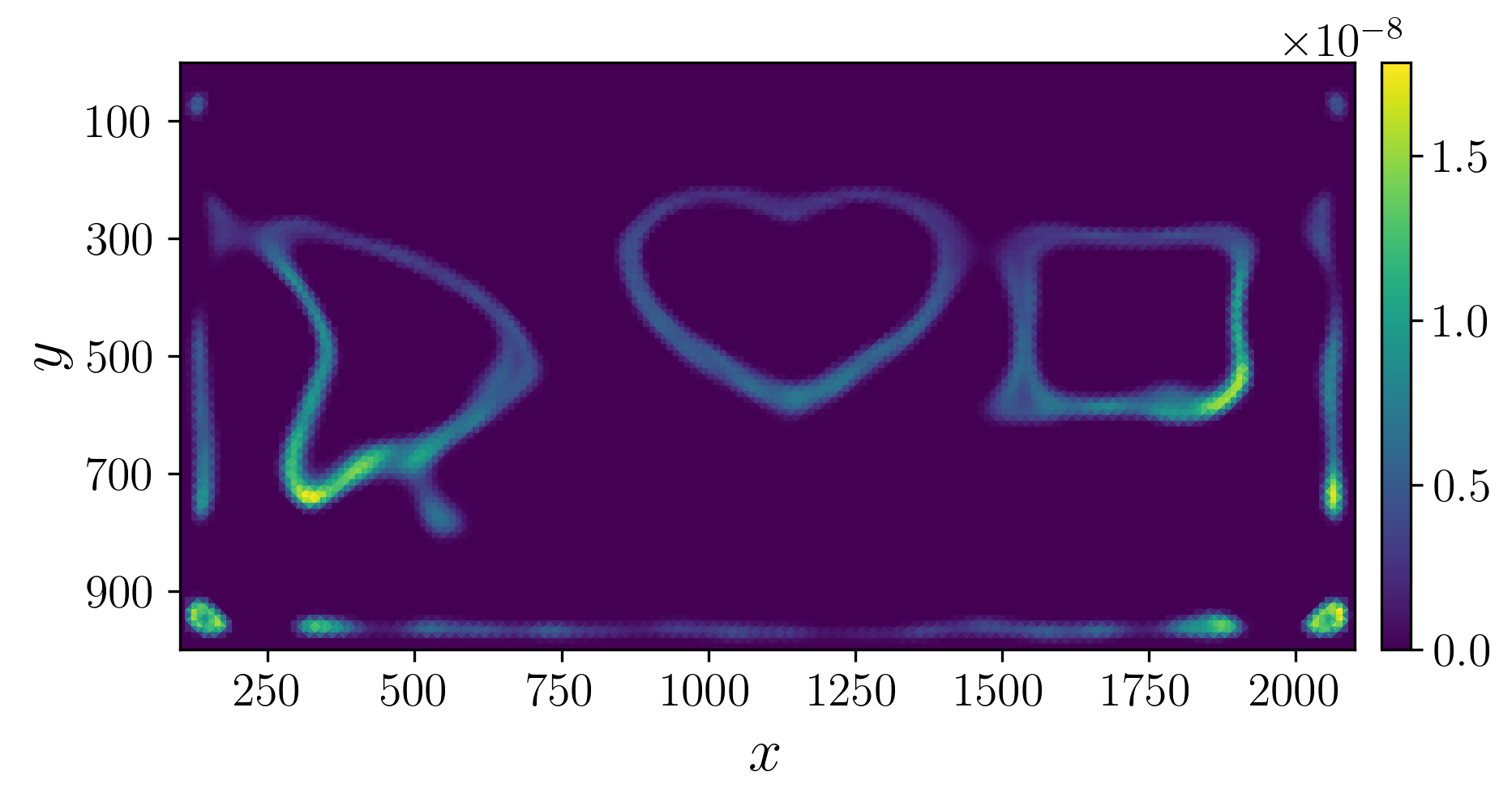}
\par\smallskip{\small (c) Posterior standard deviation of squared slowness.\par}
\end{minipage}

\par\medskip

\begin{minipage}[t]{0.33\linewidth}
\centering
\includegraphics[width=\linewidth]{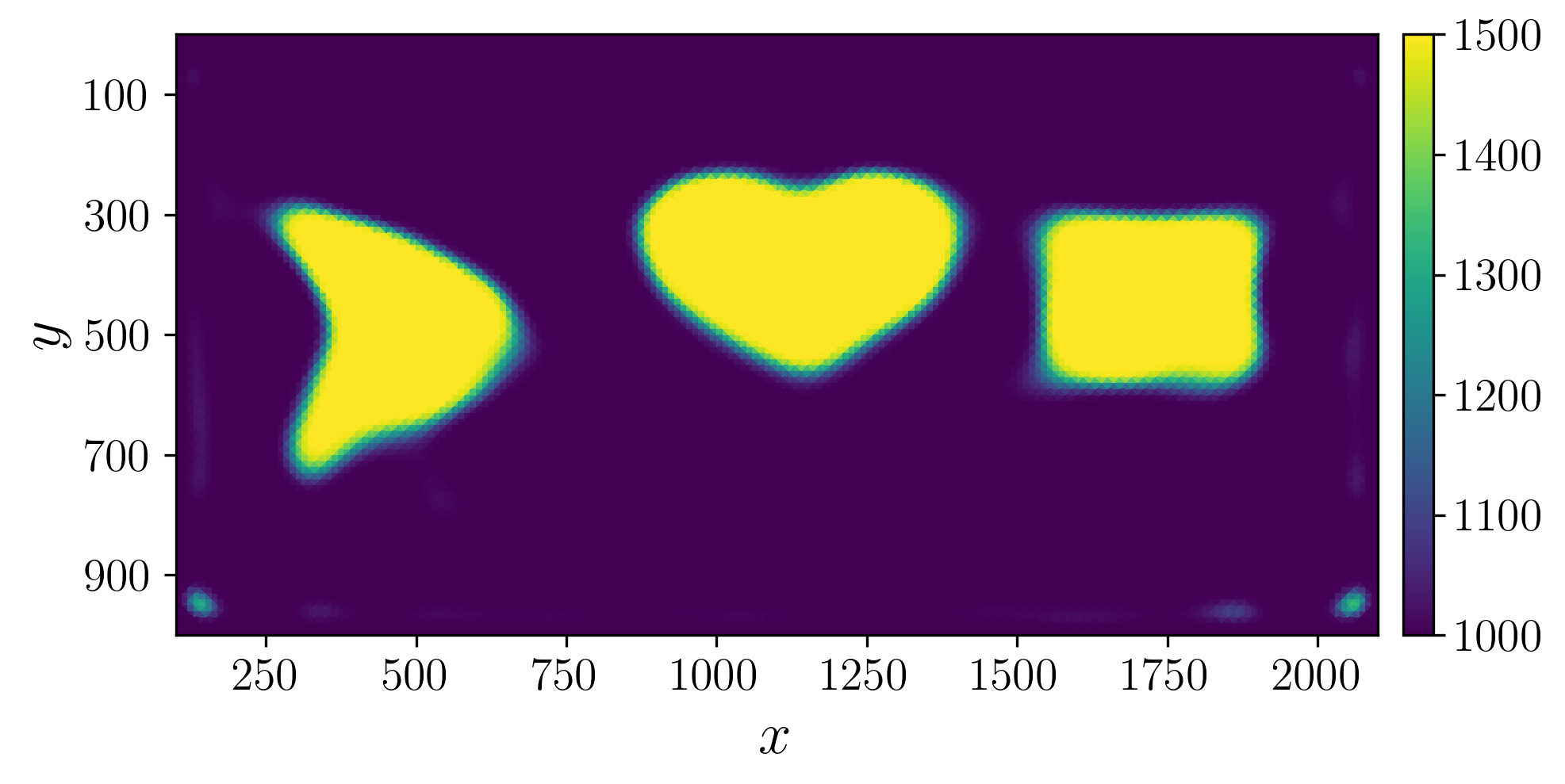}
\par\smallskip{\small (d) MAP-candidate velocity.\par}
\end{minipage}\hfill
\begin{minipage}[t]{0.33\linewidth}
\centering
\includegraphics[width=\linewidth]{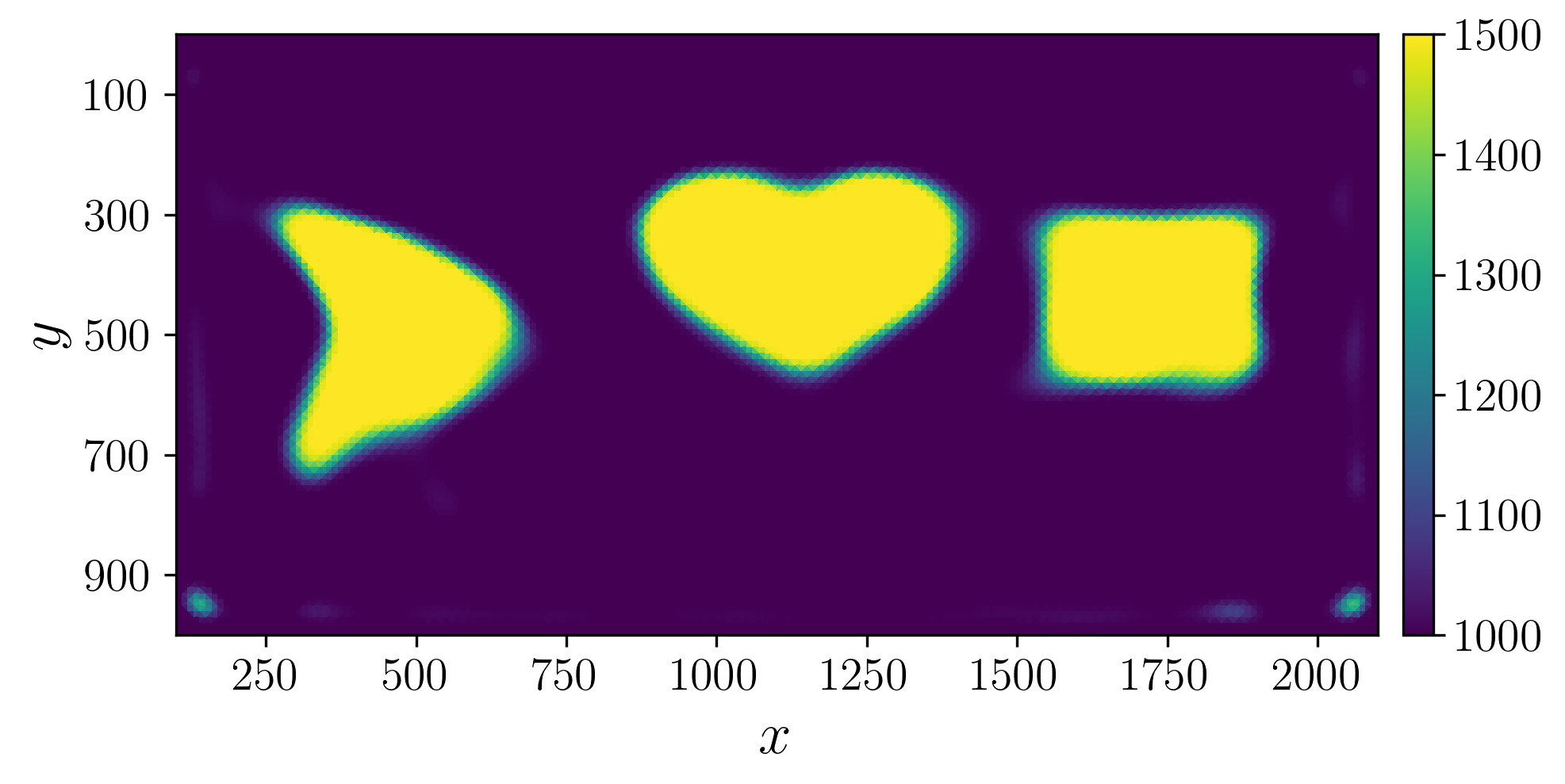}
\par\smallskip{\small (e) Posterior mean of velocity.\par}
\end{minipage}\hfill
\begin{minipage}[t]{0.33\linewidth}
\centering
\includegraphics[width=\linewidth]{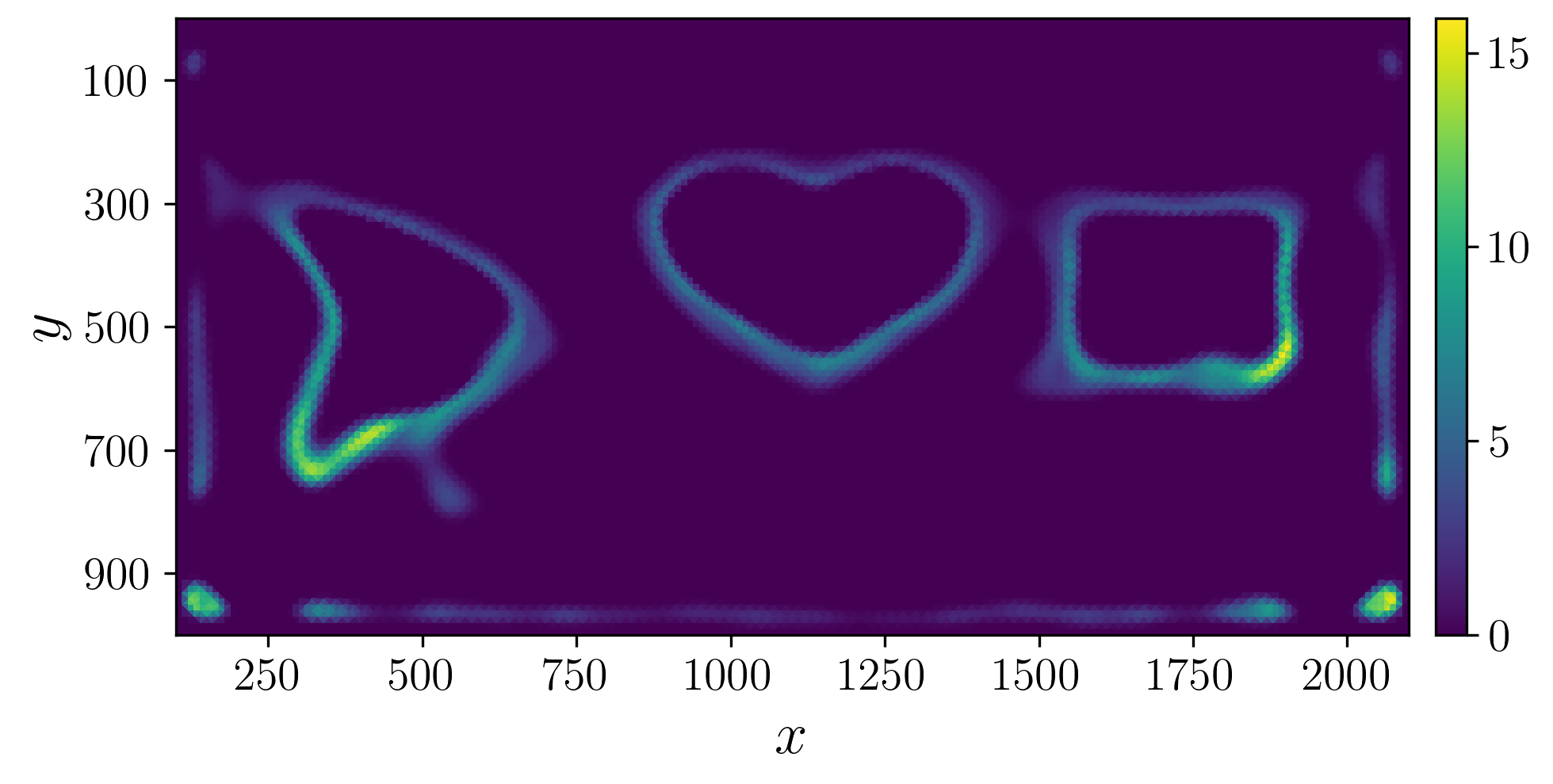}
\par\smallskip{\small (f) Posterior standard deviation of velocity.\par}
\end{minipage}

\caption{Standard single-component Laplace results for the
three-inclusion experiment (Example~\ref{example2}). The top and bottom rows show squared slowness and velocity, respectively. From left to right: the selected MAP candidate, posterior mean, and pointwise posterior standard deviation.}
\label{fig:ex3-single-laplace}
\end{figure}

\begin{table}[H]
\centering
\caption{Quantitative comparison of gPCN (MCMC), the standard
single-component Laplace approximation, and the proposed multiscale
multimodal Laplace method for the nonconvex-inclusion experiment
(Example~\ref{example2}). The best value in each row is shown in bold..}
\label{tab:ex3-three-method-comparison}
\begingroup
\small
\renewcommand{\arraystretch}{1.20}
\setlength{\tabcolsep}{5pt}
\renewcommand{\tabularxcolumn}[1]{m{#1}}
\begin{tabularx}{\linewidth}{
@{}>{\raggedright\arraybackslash}m{0.43\linewidth}
*{3}{>{\centering\arraybackslash}X}@{}
}
\toprule
Evaluation metric
& \shortstack{gPCN\\(MCMC)}
& \shortstack{Single\\Laplace}
& \shortstack{Multiscale\\multimodal\\Laplace} \\
\midrule
Squared-slowness mean error $E_{\mathrm{rel}}^{\mu}$ (\%)
& 6.4498 & 4.4569 & \textbf{4.0938} \\

Velocity mean error $E_{\mathrm{rel}}^{v}$ (\%)
& 5.6935 & 3.8054 & \textbf{3.0816} \\

Hausdorff distance $d_H$ (m)
& 42.43 m & 412.31 m & \textbf{30.00 m} \\

Boundary concentration $R_{\mathrm{bdry}}(20\,\mathrm m)$ (\%)
& 29.3508 & 62.9747 & \textbf{95.3956} \\

Node-based coverage
$\mathrm{BCov}^{\mathrm{node}}_{0.05,\,20\,\mathrm m}$ (\%)
& \textbf{89.5719} & 19.8683 & 81.4490 \\

Elapsed time (s)
& 37498.61 & \textbf{712.69} & 1013.90 \\
\bottomrule
\end{tabularx}
\endgroup
\end{table}

Table~\ref{tab:ex3-three-method-comparison} summarizes the accuracy, uncertainty measure, and computational cost of the three methods. Relative to the single-component Laplace approximation, the proposed method reduces the squared-slowness and velocity mean errors from
$4.4569\%$ and $3.8054\%$ to $4.0938\%$ and $3.0816\%$, respectively. The Hausdorff distance decreases from $412.31\,\mathrm m$ to
$30.00\,\mathrm m$; the large single-component value is consistent with
the isolated patches near the ROI margins in
Figure~\ref{fig:ex3-single-laplace}. Boundary concentration increases
from $62.9747\%$ to $95.3956\%$, and node-based coverage from
$19.8683\%$ to $81.4490\%$.

The proposed method also yields smaller errors and Hausdorff
distance than gPCN, together with a substantially higher boundary
concentration ratio. Although gPCN attains a higher node-based coverage
of $89.5719\%$, its low concentration ratio and diffuse standard
deviation fields indicate a broader, less localized uncertainty region.
The recorded runtimes are $712.69\,\mathrm s$, $1013.90\,\mathrm s$,
and $37498.61\,\mathrm s$ for the single-component, proposed, and gPCN
methods, respectively. The proposed method therefore costs about
$1.42$ times as much as the single-component approximation but is
approximately $37$ times faster than gPCN.
}
\end{example}

\begin{example}\label{example3}\rm{
This experiment reconstructs a salt-shaped anomaly adapted from the
2004 BP velocity estimation benchmark model ~\cite{Billette2005BP}. The binary mask is rescaled while preserving its aspect ratio and translated so that its lower-left corner coincides
with the lower-left corner of the ROI. This construction retains the characteristic salt geometry but replaces the spatially varying background by a constant velocity. Thus, the experiment uses a
BP2004-derived two-phase model rather than the original benchmark velocity field.

Let $D_{\mathrm{true}}$ denote the resulting salt region. The exact velocity and squared slowness are
\begin{equation*}
\begin{aligned}
v_{\mathrm{exact}}(\boldsymbol{x})
&=\begin{cases}
1500\,\mathrm{m/s}, & \boldsymbol{x}\in D_{\mathrm{true}},\\
1000\,\mathrm{m/s}, & \boldsymbol{x}\notin D_{\mathrm{true}},
\end{cases}
&\qquad
\mu_{\mathrm{exact}}(\boldsymbol{x})
&=v_{\mathrm{exact}}(\boldsymbol{x})^{-2}.
\end{aligned}
\end{equation*}
Both phase velocities remain fixed during inversion; only the anomaly geometry is inferred. The target combines a broad salt body, a deep left-hand indentation, irregular upper boundaries, and narrow right-hand extensions. The salt outline is used only to generate the observations, not to construct the initial models or constrain the recovered interface.

The computational domain, acquisition geometry, observation setup, noise level, and level-set representation are the same as those in Example~\ref{example1}. The only change is the RBF length scale, which is set to $\ell=60\,\mathrm{m}$; all other prior and inversion parameters remain unchanged. Three fitted circular initializations are used in the frequency-continuation procedure, with cut-offs of $2$, $4$, and $6\,\mathrm{Hz}$. Adam uses $300$, $300$, and $600$ epochs, respectively, with learning rate $1.0$.

Four Gaussian components are retained at the final frequency level. The final recorded weights are
\begin{equation*}
(\pi_1,\pi_2,\pi_3,\pi_4)
\approx
(0.506004,\ 0.295830,\ 0.132759,\ 0.065408).
\end{equation*}
The first component has the largest weight, but the other three together carry about $49.4\%$ of the mixture mass. The means, standard deviations, and sample-based evaluation metrics use $100$ draws from this mixture.

\begin{figure}[H]
\centering
\begin{minipage}[t]{0.33\textwidth}
\centering
\includegraphics[width=\linewidth]{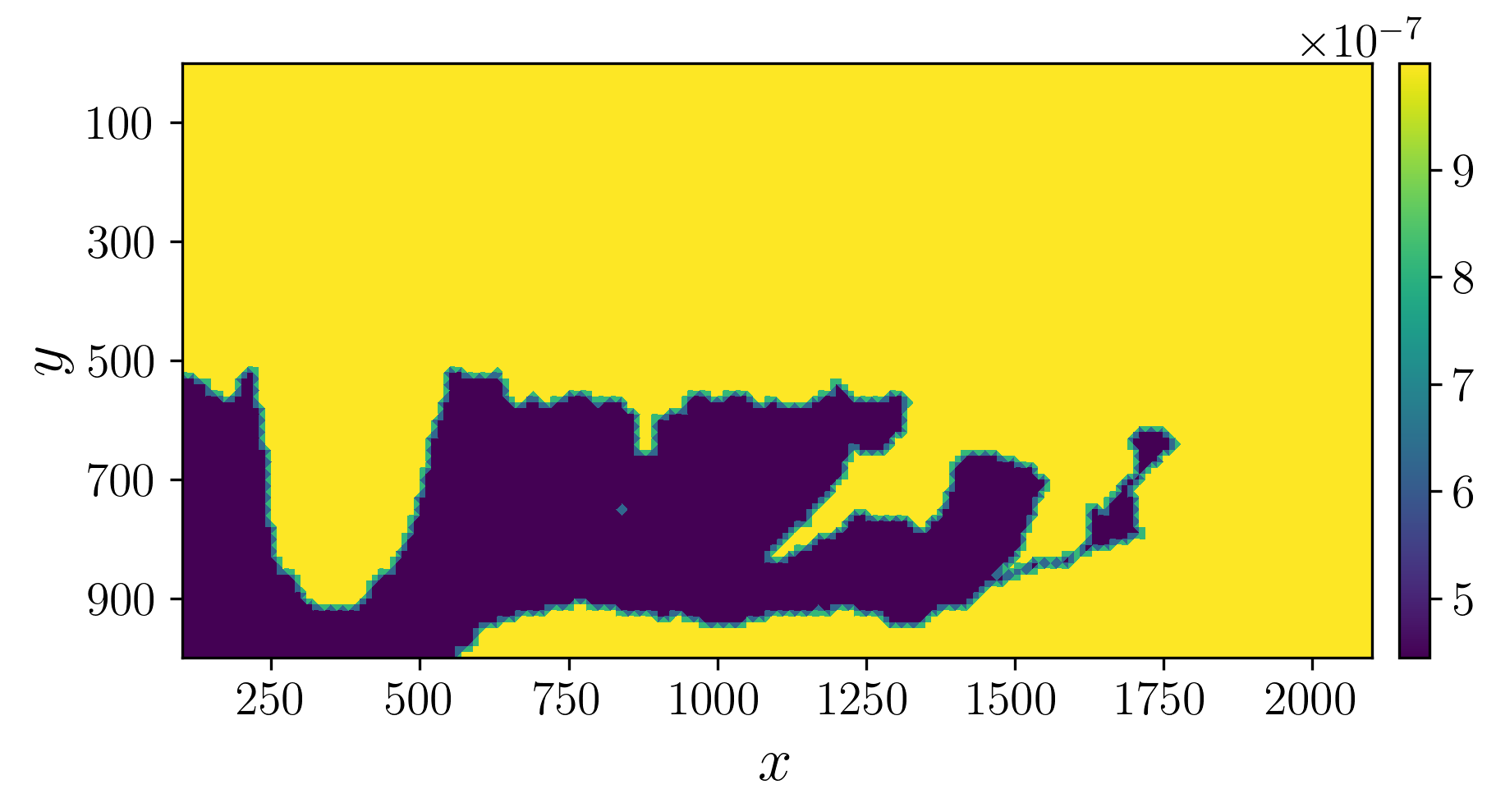}
\par\smallskip{\small (a) Exact squared slowness.\par}
\end{minipage}\hfill
\begin{minipage}[t]{0.33\textwidth}
\centering
\includegraphics[width=\linewidth]{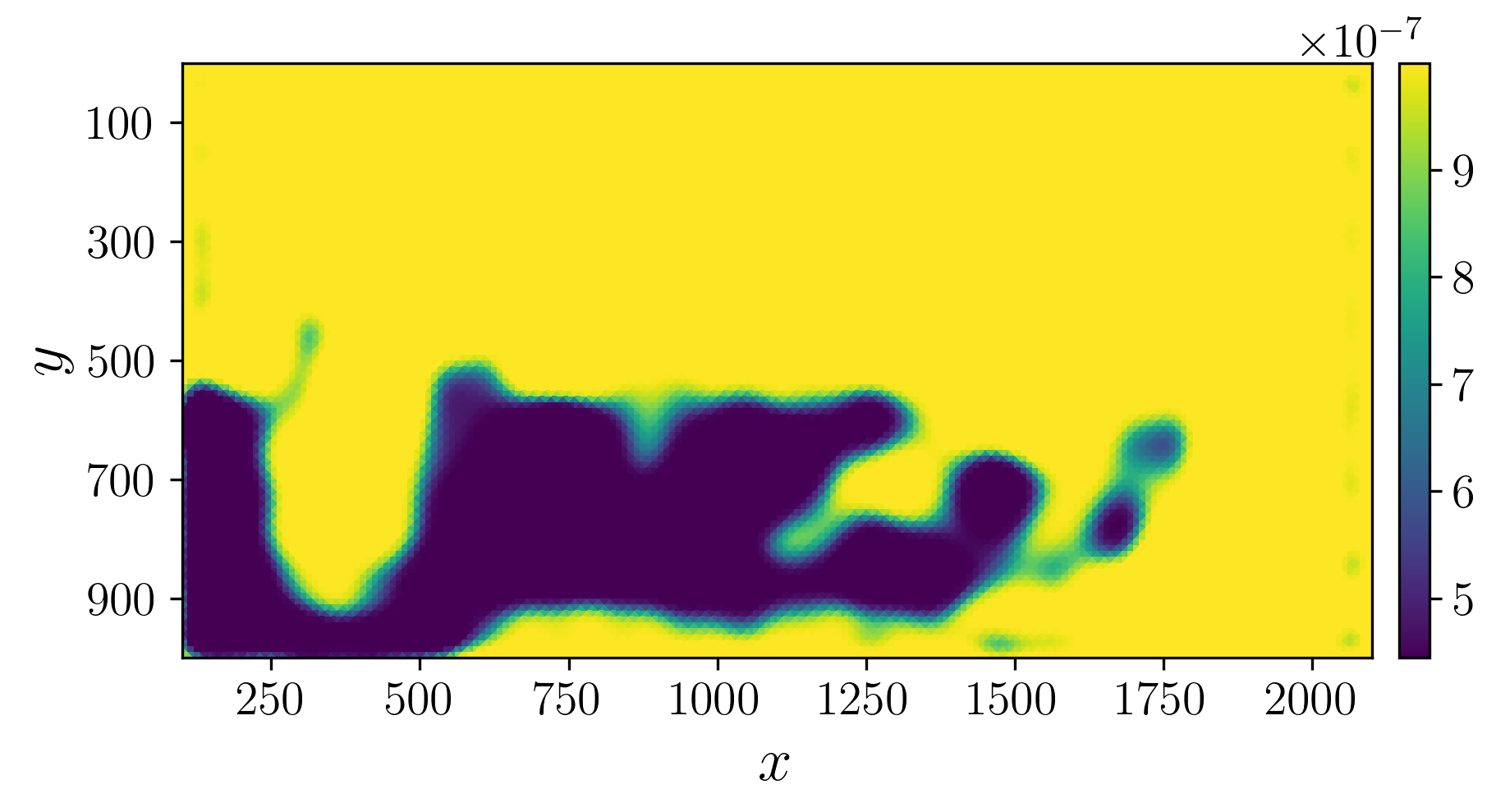}
\par\smallskip{\small (b) Best retained MAP squared slowness.\par}
\end{minipage}\hfill
\begin{minipage}[t]{0.33\textwidth}
\centering
\includegraphics[width=\linewidth]{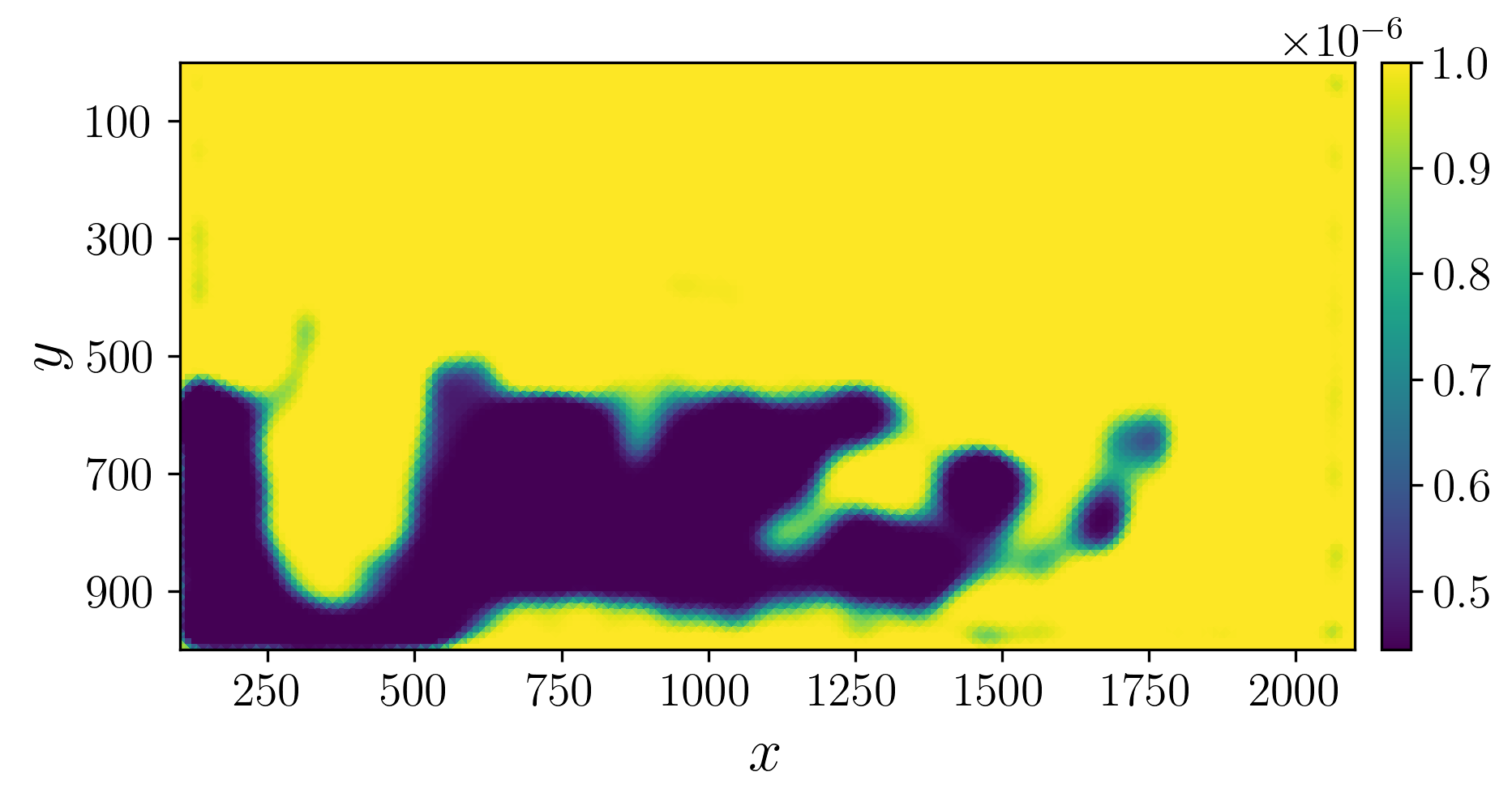}
\par\smallskip{\small (c) Mixture mean squared slowness.\par}
\end{minipage}

\par\medskip

\begin{minipage}[t]{0.33\textwidth}
\centering
\includegraphics[width=\linewidth]{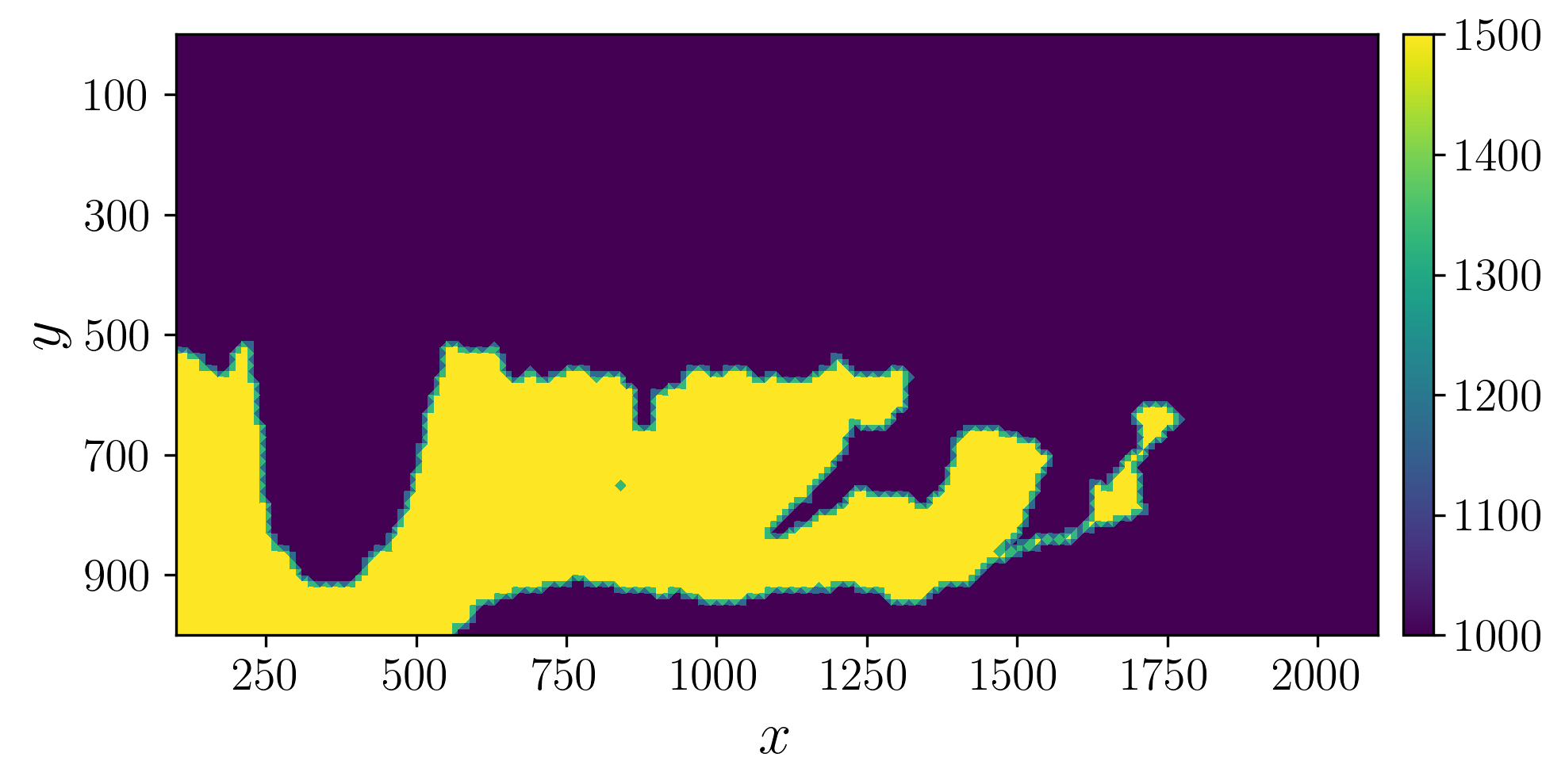}
\par\smallskip{\small (d) Exact velocity.\par}
\end{minipage}\hfill
\begin{minipage}[t]{0.33\textwidth}
\centering
\includegraphics[width=\linewidth]{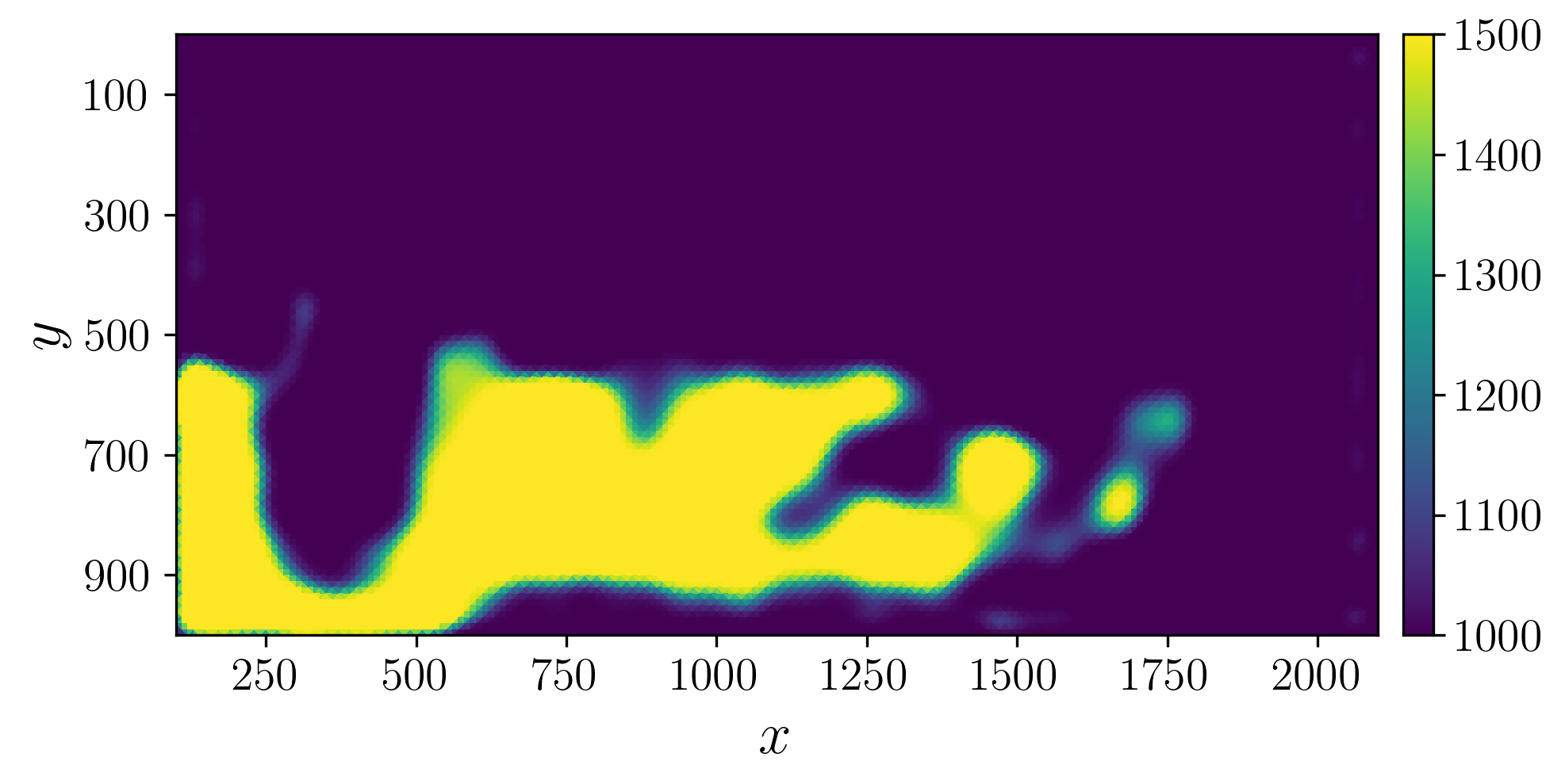}
\par\smallskip{\small (e) Best retained MAP velocity.\par}
\end{minipage}\hfill
\begin{minipage}[t]{0.33\textwidth}
\centering
\includegraphics[width=\linewidth]{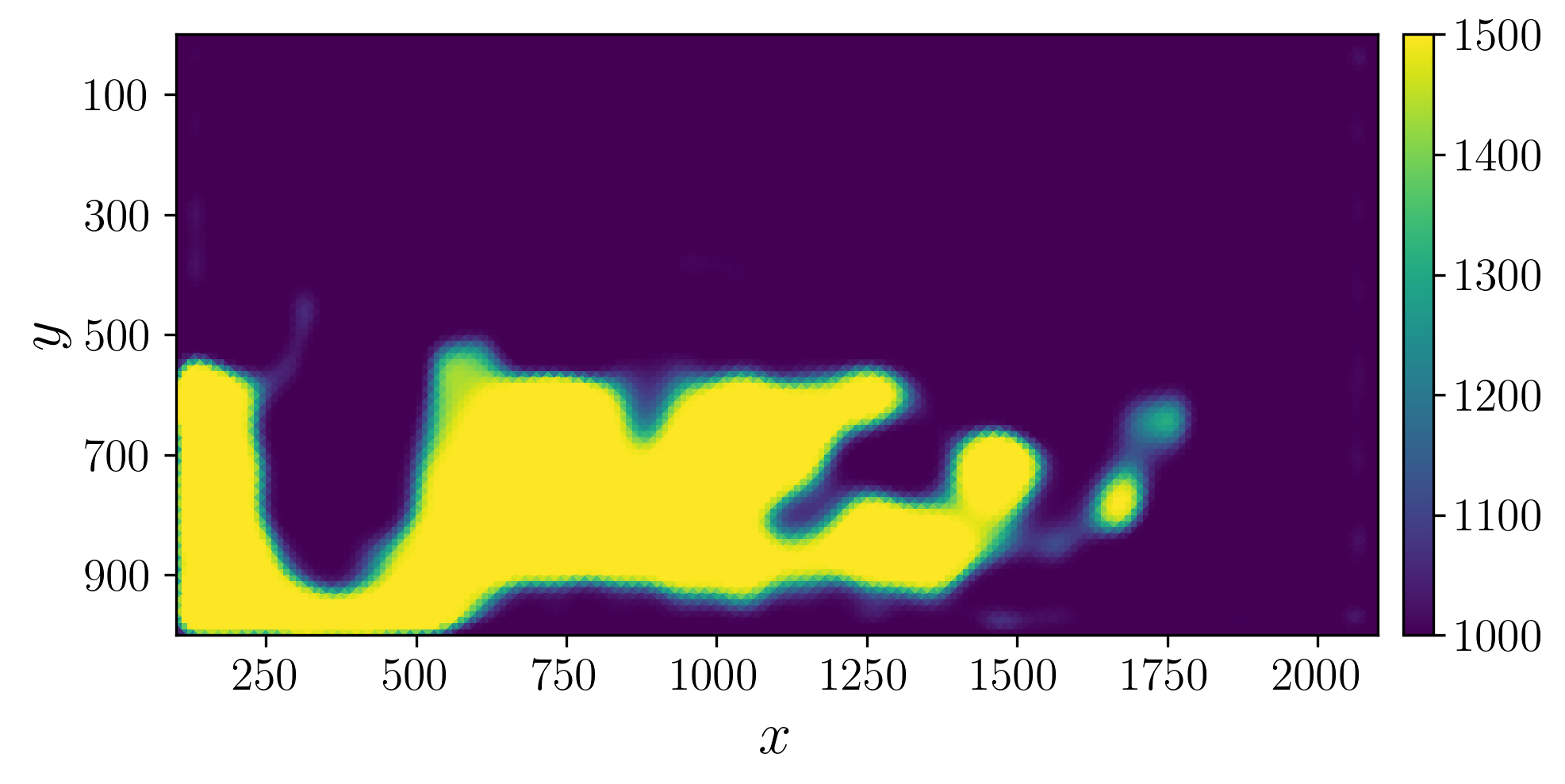}
\par\smallskip{\small (f) Mixture mean velocity.\par}
\end{minipage}
\caption{Exact, best retained MAP, and mixture mean fields for squared slowness (top row) and velocity (bottom row) in Example~\ref{example3}.}
\label{fig:example4-reconstruction}
\end{figure}

Figure~\ref{fig:example4-reconstruction} compares the exact fields with the best retained MAP reconstruction and the mixture means.  Both reconstructions successfully recover the main geometry and spatial extent of the BP2004-derived salt body. In particular, the pronounced left-hand indentation, the broad central structure, and the rightward extension are all well identified, while the reconstructed interfaces closely follow the overall shape of the exact model. The corresponding velocity fields also reproduce the prescribed two-phase contrast clearly and consistently. 
The mixture means remain highly consistent with the MAP reconstruction and preserve the principal geometric features of the salt body, indicating that the multimodal posterior representation retains the dominant structural information. Only minor smoothing is observed near some narrow notches and small-scale boundary features, without affecting the overall reconstruction quality.

\begin{figure}[H]
\centering
\begin{minipage}[t]{0.42\textwidth}
\centering
\includegraphics[width=\linewidth]{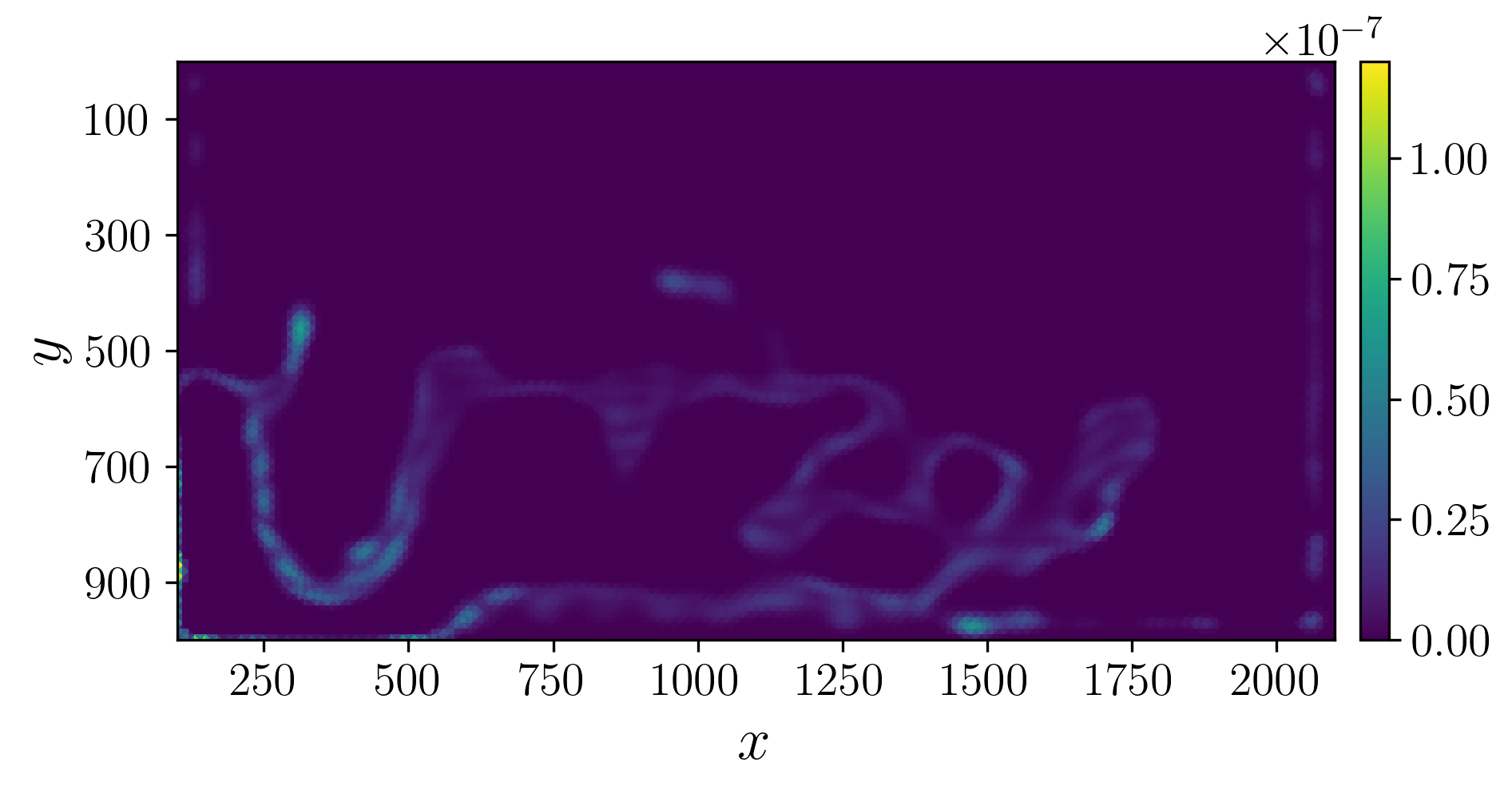}
\par\smallskip{\small (a) Squared-slowness standard deviation.\par}
\end{minipage}\hspace{0.03\textwidth}%
\begin{minipage}[t]{0.42\textwidth}
\centering
\includegraphics[width=\linewidth]{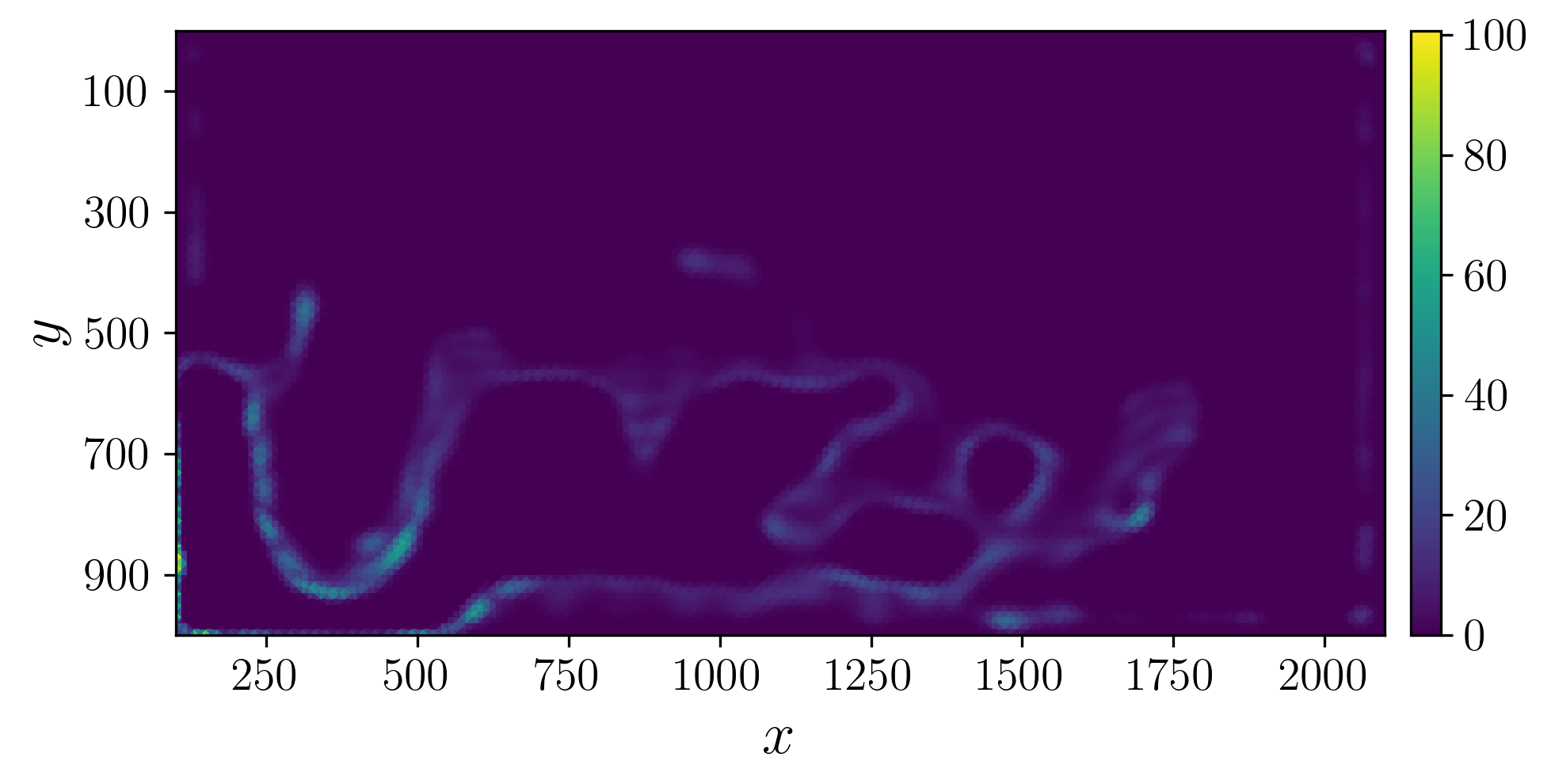}
\par\smallskip{\small (b) Velocity standard deviation.\par}
\end{minipage}
\caption{Pointwise standard deviations of squared slowness (a) and velocity (b) under the constructed Gaussian mixture. (Example~\ref{example3}).}
\label{fig:example4-uncertainty}
\end{figure}

Figure~\ref{fig:example4-uncertainty} shows the pointwise standard deviations. Variability follows much of the reconstructed interface, particularly the concave left boundary, parts of the lower boundary, and the narrow right-hand extension. The salt interior and most of the background have much smaller variability. Isolated uncertain patches
also occur away from the principal interface and near the ROI margins. 

\begin{table}[H]
\centering
\caption{Reconstruction accuracy, uncertainty diagnostics, and recorded
multiscale-stage runtime for Example~\ref{example3}. Sample-based
quantities reuse the same $100$ mixture samples.}
\label{tab:example4-metrics}
\begingroup
\small
\renewcommand{\arraystretch}{1.20}
\setlength{\tabcolsep}{6pt}
\begin{tabularx}{0.94\linewidth}{
@{}>{\raggedright\arraybackslash}X
>{\centering\arraybackslash}p{0.21\linewidth}@{}}
\toprule
Evaluation metric & Value\\
\midrule
Squared-slowness mean error $E_{\mathrm{rel}}^{\mu}$ (\%) & 6.1213\\
Velocity mean error $E_{\mathrm{rel}}^{v}$ (\%) & 4.9512\\
MAP-boundary Hausdorff distance $d_H$ (m) & 126.49\\
Boundary concentration $R_{\mathrm{bdry}}(20\,\mathrm m)$ (\%) & 83.6836\\
Node-based coverage $\mathrm{BCov}^{\mathrm{node}}_{0.05,\,20\,\mathrm m}$ (\%) & 54.5604\\
Multiscale MAP/Laplace time (s) & 1882.96\\
\bottomrule
\end{tabularx}
\endgroup
\end{table}

Table~\ref{tab:example4-metrics} reports relative mean-field errors of $6.1213\%$ for squared slowness and $4.9512\%$ for velocity. The MAP-boundary Hausdorff distance is $126.49\,\mathrm m$, indicating that some local boundary discrepancies remain appreciable despite recovery of the main body. Since this is a maximum-distance metric, it need not represent the typical displacement along the reconstructed interface. The boundary concentration ratio reaches $83.6836\%$, showing that most of the squared-slowness sample variance is concentrated within $20\,\mathrm m$ of the true boundary. The node-based coverage is $54.5604\%$, indicating that the uncertainty band still captures more
than half of the extracted boundary nodes under this strict tolerance. Together, these results show that the Gaussian-mixture approximation localizes the dominant uncertainty around the main interfaces.  
}
\end{example}

\section{Conclusions}\label{sec:conclusions}
In this paper, we have presented a Bayesian level-set framework for geometric full-waveform inversion, with a focus on reconstructing sharp interfaces that separate media with prescribed material properties. We have established the regularity of solutions to the time-domain forward wave problem, derived uniform bounds over the admissible medium coefficients, and shown that level-set-to-observation map is continuous almost surely under the Gaussian prior. Together, these results establish the well-posedness of the posterior measure and its local Lipschitz continuity in the Hellinger metric, providing a rigorous foundation for uncertainty quantification despite the lack of global continuity of the parameter-to-observation map.

To computationally approximate the posterior and capture multiple modes corresponding to distinct plausible geometries, we developed a multiscale multimodal Laplace method. Low-frequency information first guides the recovery of the large-scale location and overall geometry of the anomaly, while progressively higher frequencies refine its boundary details. Multiple posterior modes are identified and propagated through the frequency levels. At the final level, local Gaussian approximations are constructed around the retained modes and combined into a mixture whose nonnegative weights are determined by least-squares fitting to the target posterior. Numerical experiments demonstrate the effectiveness of the proposed method for shape reconstruction and uncertainty quantification.

\section*{Acknowledgments}
The work of Xiong-Bin Yan is supported by NSFC Grant No. 12401563. The work of Guang-Hui Zheng is supported by NSFC Grant No. 12271551. The work of Ting Wei is supported by NSFC Grant No. 12571455. 
\bibliographystyle{elsarticle-num}
\bibliography{reference-final}

\appendix
\section*{Appendix}
\setcounter{section}{1} 
\setcounter{subsection}{0} 
\renewcommand{\thesubsection}{\Alph{section}.\arabic{subsection}} 
\appendix
\renewcommand{\theequation}{\Alph{section}-\arabic{equation}}

\section{Regularized Parametrization and Gradient Computations}
\label{FAP}

This appendix provides the explicit definitions of the regularized coefficient-to-field and forward observation maps, the multiscale objective functionals, and the corresponding adjoint-based derivative formulations used in Section~\ref{sec:Inversion_algorithm}. Throughout this appendix, all vector spaces are real, and adjoints/transposes are defined with respect to standard Euclidean inner products. 

\subsection{Level-set discretization and smoothed medium parameterization}
\label{app:regularized-level-set-forward}

Let $x_1,\ldots,x_N$ be the spatial nodes. The nodal representation associated
with \eqref{eq:finite_levelset_representation} is
\begin{equation}\label{eq:LevelSet_D}
B_{ij}:=\Psi_j(x_i),
\quad\mathbf m_{0,h}:=\bigl(m_0(x_i)\bigr)_{i=1}^{N},
\quad\boldsymbol\phi_h(\mathbf q):=\mathbf m_{0,h}+B\mathbf q,
\quad B\in\mathbb R^{N\times L}.
\end{equation}
Prescribed level-set values are incorporated by setting the corresponding
rows of $B$ to zero and assigning those values to the corresponding entries of $\mathbf m_{0,h}$. These modified arrays are used throughout when such constraints are present.

In the continuous formulation, the squared-slowness coefficient is defined by
\begin{equation}\label{A1:indictor_function}
\mu_\phi=\mu_{\mathrm{in}}\mathbf 1_{\{\phi\geq0\}}
+\mu_{\mathrm{out}}\mathbf 1_{\{\phi<0\}}.
\end{equation}
Equivalently, introducing the discontinuous Heaviside function
\begin{equation*}
H(s)=
\begin{cases}
0, & s<0,\\
1, & s\geq0,
\end{cases}
\end{equation*}
we may write
\begin{equation}\label{eq:app-sharp-levelset-map}
\mu_\phi=\mu_{\mathrm{out}}+(\mu_{\mathrm{in}}-\mu_{\mathrm{out}})H(\phi).
\end{equation}
This representation is equivalent to the indicator-function. However, the map from the level-set coefficients
to $\mu_\phi$ is discontinuous when the zero level set crosses a spatial node. Consequently, it is unsuitable for the gradient-based optimization and local Laplace approximation used in the numerical algorithm. We therefore approximate the Heaviside function when constructing the discrete parameter-to-data map.

For a fixed smoothing width $\varepsilon>0$, replace the discontinuous
Heaviside function by the nondecreasing $C^2$ function
\begin{equation}\label{eq:alg-regularized-heaviside}
H_\varepsilon(s)=
\begin{cases}
0,&s<-\varepsilon,\\[1mm]
\displaystyle\frac12\left[
1+\frac{s}{\varepsilon}
+\frac1\pi\sin\!\left(\frac{\pi s}{\varepsilon}\right)\right],
&|s|\leq\varepsilon,\\[2mm]
1,&s>\varepsilon.
\end{cases}
\end{equation}
Its derivative is
\begin{equation}
\delta_\varepsilon(s):=H_\varepsilon'(s)=
\begin{cases}
\displaystyle\frac{1}{2\varepsilon}
\left[1+\cos\!\left(\frac{\pi s}{\varepsilon}\right)\right],
&|s|\leq\varepsilon,\\[2mm]
0,&|s|>\varepsilon.
\end{cases}
\label{eq:alg-regularized-delta}
\end{equation}
The nodal squared-slowness field is then
\begin{equation}\label{eq:alg-smoothed-medium}
\boldsymbol\mu_{\varepsilon,h}(\mathbf q)
=\mu_{\mathrm{out}}\mathbf1_N
 +(\mu_{\mathrm{in}}-\mu_{\mathrm{out}})
 H_\varepsilon(\boldsymbol\phi_h(\mathbf q)),
\end{equation}
where $H_\varepsilon$ acts componentwise. Since $0\leq H_\varepsilon\leq1$, this construction preserves the positive bounds set by $\mu_{\mathrm{in}}$ and $\mu_{\mathrm{out}}$. Its sensitivity to the level-set values is confined to $|\boldsymbol\phi_h|\leq\varepsilon$. The same $H_\varepsilon$ and $\varepsilon$ are used in
all forward, derivative, and posterior-density evaluations.


\subsection{Source functions and discrete observation map}
\label{app:discrete-observation-map}

For the $s$-th source function, we set
\begin{equation}
S^{(s)}(\boldsymbol{x},t)=A_{\mathrm{src},s}\,
 \rho_{\varepsilon_s}(\boldsymbol{x}-\boldsymbol{x_s})\,w(t),
\qquad
\rho_{\varepsilon_s}(\boldsymbol{x}-\boldsymbol{x_s})
=\frac{1}{2\pi\varepsilon_s^2}
 \exp\!\left[-\frac{\|\boldsymbol{x}-\boldsymbol{x_s}\|^2}
                  {2\varepsilon_s^2}\right],
\label{eq:app-source-profile}
\end{equation}
where $A_{\mathrm{src},s}$ is the source amplitude, $\boldsymbol{x_s}$ denotes location, and $\varepsilon_s>0$ its spatial width. The temporal waveform is the Ricker pulse
$w(t)=\left[1-2\pi^2 f_0^2(t-t_0)^2\right]\exp[-\pi^2 f_0^2(t-t_0)^2]$,
with dominant frequency $f_0$ and time delay $t_0$. The sources and damping are prescribed independently of $\mathbf q$.

On a grid with $N=N_xN_y$ nodes, let $\mathcal U[\mathbf q]$ denote the discrete wavefield histories computed for the coefficient vector $\mathbf q$ using the forward solver summarized in Section~\ref{subsec:experimental-settings}.
After all source batches have been processed, the observation operator
$\mathcal R_{\mathrm{obs}}$ samples or interpolates the wavefield at the
receiver locations, retains the prescribed $N_t$ recording times, and
collects all $N_s$ source experiments. For one coefficient vector,
\begin{equation}
\begin{aligned}
\mathcal G_{L,h}^{\varepsilon}(\mathbf q)
&=\mathcal R_{\mathrm{obs}}\mathcal U[\mathbf q]
\in\mathbb R^{N_s\times N_r\times N_t},
\\
G_{L,h}^{\varepsilon}(\mathbf q)
&=\operatorname{vec}\!\left(
\mathcal G_{L,h}^{\varepsilon}(\mathbf q)\right)\in\mathbb R^M,
\qquad M=N_sN_rN_t.
\end{aligned}
\label{eq:app-numerical-forward-map}
\end{equation}
The receiver ordering, source ordering, and time ordering used by $\operatorname{vec}$ agree with those of the observation vector $d$.

\subsection{Multiscale objective functionals and their derivatives}
\label{app:mml-derivatives}

\paragraph{Frequency filtering and residual injection.}
At frequency level $k$, the discrete Fourier transform is applied along the time dimension of each source--receiver trace. A binary frequency mask is then used to retain the selected frequencies, after which the filtered trace is transformed back to the time domain.

Let $\mathcal F_t$ denote the discrete Fourier transform along the time
dimension and let $\mathcal{M}_k$ denote the binary frequency mask at level
$k$. We define the corresponding time-domain filtering operator by
$P_k =\mathcal F_t^{-1}\mathcal{M}_k\mathcal F_t.$
Thus, for any time-dependent data vector $\mathbf v$,
\begin{equation*}\label{eq:alg-fourier-filtering}
P_k\mathbf v=\mathcal F_t^{-1}\left(\mathcal{M}_k\mathcal F_t\mathbf v\right).
\end{equation*}

The mask $\mathcal{M}_k$ retains conjugate frequency pairs, so that
$P_k\mathbf v$ is real whenever $\mathbf v$ is real. We assume that the retained frequency sets are nested and that the final mask retains all frequencies. Consequently,
\begin{equation*}
P_k^\top=P_k,\qquad P_k^2=P_k,\qquad
\operatorname{ran}P_k\subseteq\operatorname{ran}P_{k+1},
\qquad P_{N_f}=I_M.
\end{equation*}

Let $\mathbf e(\mathbf q)=G_{L,h}^{\varepsilon}(\mathbf q)-d$
be the full time-domain residual. Since $P_k$ is linear, the filtered
residual can be written as
\begin{align*}\label{eq:alg-filtered-residual}
\mathbf e_k(\mathbf q)=P_k\mathbf e(\mathbf q)=P_kG_{L,h}^{\varepsilon}(\mathbf q)-P_kd.
\end{align*}
For a positive definite observation covariance $\Gamma$, define the
covariance of the filtered noise by
\begin{equation*}\label{eq:alg-filtered-covariance}
\Gamma_k =P_k\Gamma P_k^\top.
\end{equation*}
The negative log-posterior at frequency level $k$ is then
\begin{equation*}\label{eq:frequency-posterior}
J_k(\mathbf q)=\frac12 \mathbf e_k(\mathbf q)^\top
\Gamma_k^\dagger \mathbf e_k(\mathbf q)
+J_{\mathrm{prior},L,h}(\mathbf q),
\end{equation*}
where $\Gamma_k^\dagger$ denotes the Moore--Penrose inverse. 

The weighted time-domain residual used in the adjoint calculation is
\begin{equation*}\label{eq:alg-weighted-residual}
\mathbf z_k(\mathbf q)=P_k^\top\Gamma_k^\dagger P_k\mathbf e(\mathbf q).
\end{equation*}
Accordingly, the gradient of the data-misfit term is
\begin{equation*}
\nabla J_{\mathrm{data},k}(\mathbf q)
=
DG_{L,h}^{\varepsilon}(\mathbf q)^\top
\mathbf z_k(\mathbf q).
\label{eq:alg-filtered-data-gradient}
\end{equation*}

For the independent-noise model $\Gamma=\delta^2I_M$, we have
$\Gamma_k=\delta^2P_k,\, \Gamma_k^\dagger=\delta^{-2}P_k.$
Therefore,
\begin{equation*}
J_k(\mathbf q)
=\frac{1}{2\delta^2}\left\|P_k\mathbf e(\mathbf q)\right\|_2^2
+J_{\mathrm{prior},L,h}(\mathbf q),
\end{equation*}
and the weighted residual reduces to
$\mathbf z_k(\mathbf q)=\delta^{-2}P_k\mathbf e(\mathbf q).$

\paragraph{Coefficient gradient.}
Let $\mathbf g_{\phi,k}$ denote the Euclidean gradient of the data-misfit term with respect to the nodal level-set values. The coefficient gradient follows from \eqref{eq:LevelSet_D}:
\begin{equation}
\nabla J_k(\mathbf q)
=B^\top\mathbf g_{\phi,k}
 +\mu_{\mathrm{prior}}(K_{L,h}\mathbf q+\mathbf b_{L,h}),
\label{eq:total-gradient}
\end{equation}
where the $K_{L,h},~\mathbf{b}_{L,h}$ are defined in Appendix~\ref{app:mml-prior}.

\paragraph{Jacobian matrix.}
For a given candidate coefficient vector $\mathbf{q}_a \in \mathbb{R}^L$ and a perturbation direction $\delta\mathbf{q} \in \mathbb{R}^L$, the directional derivative of the squared slowness is given by the chain rule:
\begin{equation}
    \dot{\boldsymbol{\mu}} = (\mu_{\mathrm{in}} - \mu_{\mathrm{out}}) \, \delta_\varepsilon(\boldsymbol{\phi}_h) \odot (B\delta\mathbf{q}), \quad \text{with } \dot{\boldsymbol{\phi}}_h = B\delta\mathbf{q}.
    \label{eq:mu-tangent}
\end{equation}
Linearizing the discrete wave equation at candidate $\mathbf{q}_a$ along the $j$-th canonical direction $\mathbf{e}_j$ yields the tangent wavefield history $\dot{\mathcal{U}}_a[\mathbf{e}_j]$. The $j$-th column of the projected Jacobian $G_{k,a}$ at frequency level $k$ is then
\begin{equation}
    G_{k,a}^{(:,j)} = P_k \operatorname{vec}\bigl(\mathcal{R}_{\mathrm{obs}}\dot{\mathcal{U}}_a[\mathbf{e}_j]\bigr), \quad j = 1, \dots, L.
    \label{eq:jacobian-column}
\end{equation}

\paragraph{Local curvature and factorization.}
Let $\mathbf{f}_k = P_k G_{L,h}^\varepsilon$ and $\mathbf{e}_k(\mathbf{q}) = \mathbf{f}_k(\mathbf{q}) - P_k y$. The Hessian of the frequency-dependent objective $J_k$ expands as
\begin{equation}
    \nabla^2 J_k(\mathbf{q})
    = D\mathbf{f}_k(\mathbf{q})^\top \Gamma_k^\dagger D\mathbf{f}_k(\mathbf{q}) + \Lambda_{\mathrm{prior},h}
    + \sum_{i=1}^M \bigl[\Gamma_k^\dagger \mathbf{e}_k(\mathbf{q})\bigr]_i \nabla^2 f_{k,i}(\mathbf{q}).
    \label{eq:app-exact-hessian}
\end{equation}
Neglecting the residual-weighted second derivatives at a candidate mode $\mathbf{q}_{k,a}^{\mathrm{MAP}}$ yields the Gauss--Newton approximation
$H_{k,a}^{\mathrm{GN}} = G_{k,a}^\top \Gamma_k^\dagger G_{k,a} + \Lambda_{\mathrm{prior},h}$,
which simplifies to \eqref{eq:alg-gn-hessian} under independent noise.
To ensure numerical stability during Cholesky factorization, we form the regularized curvature and local covariance:
\begin{equation}
    H_{k,a} = \frac{1}{2} \bigl( H_{k,a}^{\mathrm{GN}} + (H_{k,a}^{\mathrm{GN}})^\top \bigr) + \epsilon_{H,k,a} I_L,
    \qquad C_{k,a} = H_{k,a}^{-1},
    \label{eq:alg-local-covariance}
\end{equation}
where $\epsilon_{H,k,a} \ge 0$ is a small diagonal shift applied only when factorization fails (and is zero otherwise).

\section{Prior and posterior-mixture implementation}
\label{app:mml-implementation}

This appendix describes the discrete prior and the conversion of coefficient samples into field and geometric uncertainty. 


\subsection{Construction of the discrete Gaussian prior}
\label{app:mml-prior}
Let $\{\varphi_i\}_{i=1}^{N}$ be the continuous piecewise-linear nodal basis
on the triangular mesh used to evaluate the prior. Define
\[
K_{ij}:=\int_\Omega\nabla\varphi_i\cdot\nabla\varphi_j\,dx,
\qquad K=K^\top\in\mathbb R^{N\times N},
\]
and let $\phi_{h,\mathbf q}
=\sum_{i=1}^{N}[\boldsymbol\phi_h(\mathbf q)]_i\varphi_i$ be the
interpolated nodal field. Then
\begin{equation*}
J_{\mathrm{prior},L,h}(\mathbf q)
=\frac{\mu_{\mathrm{prior}}}{2}
\int_\Omega|\nabla\phi_{h,\mathbf q}|^2\,dx
=\frac{\mu_{\mathrm{prior}}}{2}
\boldsymbol\phi_h(\mathbf q)^\top K\boldsymbol\phi_h(\mathbf q).
\label{eq:prior_discrete_phi}
\end{equation*}


Introduce
\begin{equation*}\label{eq_appendix:KB_lambda}
K_{L,h}:=B^\top KB,
\, \mathbf b_{L,h}:=B^\top K\mathbf m_{0,h},
\, \Lambda_{\mathrm{prior},h}:=\mu_{\mathrm{prior}}K_{L,h}.
\end{equation*}
Expanding \eqref{eq:prior_discrete_phi} gives
\begin{equation*}
J_{\mathrm{prior},L,h}(\mathbf q)
=\frac{\mu_{\mathrm{prior}}}{2} \left(\mathbf q^\top K_{L,h}\mathbf q +2\mathbf b_{L,h}^\top\mathbf q\right)+c_h,
\qquad c_h=\frac{\mu_{\mathrm{prior}}}{2}
\mathbf m_{0,h}^\top K\mathbf m_{0,h}.
\end{equation*}
Completing the square yields \eqref{eq:mml-prior-potential}, with
\begin{equation*}
\bar{\mathbf q}_{L,h}=-K_{L,h}^{-1}\mathbf b_{L,h},\quad
C_{\mathbf q,h}=\Lambda_{\mathrm{prior},h}^{-1},
\quad\widetilde c_h
=c_h-\frac{\mu_{\mathrm{prior}}}{2}
\mathbf b_{L,h}^\top K_{L,h}^{-1}\mathbf b_{L,h}.
\label{eq:app-prior-mean-covariance}
\end{equation*}
The corresponding derivatives are
\begin{equation*}
\begin{aligned}
\nabla J_{\mathrm{prior},L,h}(\mathbf q)
&=\mu_{\mathrm{prior}}(K_{L,h}\mathbf q+\mathbf b_{L,h})
=\mu_{\mathrm{prior}}B^\top K\boldsymbol\phi_h(\mathbf q),\\
\nabla^2J_{\mathrm{prior},L,h}(\mathbf q)
&=\Lambda_{\mathrm{prior},h}.
\end{aligned}
\label{eq:alg-prior-derivatives}
\end{equation*}

\subsection{Posterior Sampling and Uncertainty Quantification}
\label{app:mml-sampling}

Independent samples from the final Gaussian mixture are generated by
\begin{equation*}
\begin{aligned}
r_\ell&\sim\operatorname{Categorical}(w_1,\ldots,w_R),\\
\mathbf q^{(\ell)}&=\mathbf q_{r_\ell}^{\mathrm{MAP}}
+L_{r_\ell}\boldsymbol\xi^{(\ell)},
\qquad
\boldsymbol\xi^{(\ell)}\sim\mathcal N(0,I_L),
\qquad
L_{r_\ell}L_{r_{\ell}}^\top=C_{r_\ell},
\end{aligned}
\label{eq:alg-mixture-sampling}
\end{equation*}
where the component labels and Gaussian vectors are sampled
independently for $\ell=1,\ldots,N_{\mathrm{sample}}$ and $C_{r_{\ell}}:=C_{N_f},r_{\ell}$ given by \eqref{eq:alg-local-covariance}.

Each coefficient sample induces the nodal fields
\begin{equation}
\begin{aligned}
\boldsymbol\phi_h^{(\ell)}
&=
\mathbf m_{0,h}+B\mathbf q^{(\ell)},
\\
\boldsymbol\mu_{\varepsilon,h}^{(\ell)}
&=
\mu_{\mathrm{out}}\mathbf 1_N
+
(\mu_{\mathrm{in}}-\mu_{\mathrm{out}})
H_\varepsilon(\boldsymbol\phi_h^{(\ell)}),
\\
\boldsymbol v_h^{(\ell)}
&=
\left(
\boldsymbol\mu_{\varepsilon,h}^{(\ell)}
\right)^{-1/2}.
\end{aligned}
\label{eq:alg-field-samples}
\end{equation}
Here $H_\varepsilon$ is defined in
\eqref{eq:alg-regularized-heaviside}, and all nonlinear operations are applied componentwise.
The means and pointwise standard deviations of squared slowness and
velocity are estimated from the transformed samples in
\eqref{eq:alg-field-samples}.

\section{Evaluation metrics}
\label{app:evaluation-metrics}

We evaluate the reconstruction using four metrics: relative errors of the
posterior mean fields, the Hausdorff distance between the reconstructed and
exact interfaces, the concentration of posterior variance near the exact
interface, and the coverage of the exact interface by the posterior
uncertainty band. 

Let $\{\mathbf q^{(\ell)}\}_{\ell=1}^{N_{\mathrm{sample}}}$ be samples drawn from
the final posterior approximation. The corresponding level set, squared
slowness, and velocity fields are
\begin{equation*}
\begin{aligned}
\phi^{(\ell)}(\boldsymbol{x})
&=\phi(\boldsymbol{x};\mathbf q^{(\ell)}),\\
\mu^{(\ell)}(\boldsymbol{x})
&=\mu_2+(\mu_1-\mu_2)
  H_\varepsilon\!\left(\phi^{(\ell)}(\boldsymbol{x})\right),\\
v^{(\ell)}(\boldsymbol{x})
&=\frac{1}{\sqrt{\mu^{(\ell)}(\boldsymbol{x})}}.
\end{aligned}
\end{equation*}
Their empirical posterior means are
\begin{equation*}
\overline{\mu}(\boldsymbol{x})
=\frac{1}{N_{\mathrm{sample}}}
\sum_{\ell=1}^{N_{\mathrm{sample}}}\mu^{(\ell)}(\boldsymbol{x}),
\qquad
\overline{v}(\boldsymbol{x})
=\frac{1}{N_{\mathrm{sample}}}
\sum_{\ell=1}^{N_{\mathrm{sample}}}v^{(\ell)}(\boldsymbol{x}).
\end{equation*}

\subsection{Relative errors of the posterior means}

Let $\mu_{\mathrm{exact}}$ and $v_{\mathrm{exact}}$ denote the exact squared
slowness and velocity fields, respectively. Their relative posterior mean
errors are defined by
\begin{equation*}
E_{\mathrm{rel}}^{\mu}
=
\frac{\|\overline{\mu}-\mu_{\mathrm{exact}}\|_{L^2(\Omega)}}
     {\|\mu_{\mathrm{exact}}\|_{L^2(\Omega)}},
\qquad
E_{\mathrm{rel}}^{v}
=\frac{\|\overline{v}-v_{\mathrm{exact}}\|_{L^2(\Omega)}}
     {\|v_{\mathrm{exact}}\|_{L^2(\Omega)}}.
\label{eq:num-relative-errors}
\end{equation*}

\subsection{Hausdorff distance between interfaces}

The discrepancy between the reconstructed and exact interfaces is measured
by the symmetric Hausdorff distance \cite{Burago2001}:
\begin{equation}
d_H(\Gamma_{\mathrm{MAP}},\Gamma_{\mathrm{true}})
=\max\Bigl\{
\sup_{\boldsymbol{x}\in\Gamma_{\mathrm{MAP}}}
\inf_{\boldsymbol{y}\in\Gamma_{\mathrm{true}}}
\|\boldsymbol{x}-\boldsymbol{y}\|_2,
\sup_{\boldsymbol{y}\in\Gamma_{\mathrm{true}}}
\inf_{\boldsymbol{x}\in\Gamma_{\mathrm{MAP}}}
\|\boldsymbol{x}-\boldsymbol{y}\|_2
\Bigr\}.
\label{eq:num-continuous-hausdorff}
\end{equation}
A smaller value indicates closer agreement between the two interfaces.

Let
\begin{equation*}
\mathcal B_{\mathrm{MAP}}
=\{\boldsymbol{x}_i^{\mathrm{MAP}}\}_{i=1}^{N_{\mathrm{MAP}}},
\qquad\mathcal B_{\mathrm{true}}
=\{\boldsymbol{x}_j^{\mathrm{true}}\}_{j=1}^{N_{\mathrm{true}}}
\end{equation*}
be the boundary point sets extracted within the ROI. For these finite point
sets, the Hausdorff distance reduces to the standard max--min form used for
comparing sampled shapes and image boundaries
\cite{Huttenlocher1993}:
\begin{equation}
d_H(\mathcal B_{\mathrm{MAP}},\mathcal B_{\mathrm{true}})
=\max\Bigl\{
\max_{1\leq i\leq N_{\mathrm{MAP}}}
\min_{1\leq j\leq N_{\mathrm{true}}}
\|\boldsymbol{x}_i^{\mathrm{MAP}}
-\boldsymbol{x}_j^{\mathrm{true}}\|_2,
\max_{1\leq j\leq N_{\mathrm{true}}}
\min_{1\leq i\leq N_{\mathrm{MAP}}}
\|\boldsymbol{x}_j^{\mathrm{true}}
-\boldsymbol{x}_i^{\mathrm{MAP}}\|_2
\Bigr\}.
\label{eq:num-discrete-hausdorff}
\end{equation}
The resulting distance is reported in metres.

\subsection{Concentration of posterior variance near the interface}
\label{app:variance-interface}

For $N_{\mathrm{sample}}>1$, the empirical pointwise variance of the
squared slowness is defined by
\begin{equation*}
\operatorname{Var}_{\mu}(\boldsymbol{x})
=\frac{1}{N_{\mathrm{sample}}-1}
\sum_{s=1}^{N_{\mathrm{sample}}}\left(
\mu^{(s)}(\boldsymbol{x})-\overline{\mu}(\boldsymbol{x})\right)^2,
\qquad
\overline{\mu}(\boldsymbol{x})
=\frac{1}{N_{\mathrm{sample}}}
\sum_{s=1}^{N_{\mathrm{sample}}}
\mu^{(s)}(\boldsymbol{x}).
\end{equation*}
The corresponding empirical standard deviation is
\begin{equation*}
\operatorname{Std}_{\mu}(\boldsymbol{x})
=\operatorname{Var}_{\mu}(\boldsymbol{x})^{1/2}.
\end{equation*}
The velocity variance and standard deviation are defined analogously.

For piecewise-constant geometric reconstructions, posterior variability is expected to be concentrated primarily near regions where the interface location is uncertain. To quantify this localization, we introduce the
following diagnostic. For a prescribed radius $r>0$, let
\begin{equation*}
B_r(\Gamma_{\mathrm{true}})
=\left\{\boldsymbol{x}\in\Omega:
\operatorname{dist}(\boldsymbol{x},\Gamma_{\mathrm{true}})\le r
\right\}
\end{equation*}
denote the $r$-neighborhood of the exact interface. We define the
interface variance concentration ratio by
\begin{equation}
R_{\mathrm{bdry}}(r)
=\displaystyle\int_{B_r(\Gamma_{\mathrm{true}})}
\operatorname{Var}_{\mu}(\boldsymbol{x})\,d\boldsymbol{x}
\bigl/
\displaystyle\int_{\Omega}
\operatorname{Var}_{\mu}(\boldsymbol{x})\,d\boldsymbol{x}.
\label{eq:boundary-variance-ratio}
\end{equation}
Whenever the denominator is nonzero,
$R_{\mathrm{bdry}}(r)\in[0,1]$.
A value close to one indicates that most of the posterior variance is localized within distance $r$ of the exact interface, whereas a smaller value indicates that a larger fraction of the uncertainty is distributed away from the interface.

In the numerical implementation, a small stabilization parameter
$\varepsilon_{\mathrm{stab}}>0$ is added to the denominator of
\eqref{eq:boundary-variance-ratio} when necessary to avoid division by a numerically vanishing total variance. This parameter is unrelated to the Heaviside smoothing parameter $\varepsilon$.

\subsection{Coverage by the posterior uncertainty band}

Each posterior sample determines an anomaly region
\begin{equation*}
D^{(s)}
=
\left\{
\boldsymbol{x}\in\Omega:
\phi^{(s)}(\boldsymbol{x})>0
\right\},
\qquad
s=1,\ldots,N_{\mathrm{sample}}.
\end{equation*}
The posterior probability that $\boldsymbol{x}$ belongs to the anomaly is
estimated by
\begin{equation*}
p_D(\boldsymbol{x})
=\mathbb P(\boldsymbol{x}\in D^{(s)})
\approx
\frac{1}{N_{\mathrm{sample}}}
\sum_{s=1}^{N_{\mathrm{sample}}}
\mathbf 1_{\{\phi^{(s)}(\boldsymbol{x})>0\}}.
\label{eq:num-inclusion-probability}
\end{equation*}
Values near zero or one indicate agreement among the posterior samples,
whereas intermediate values indicate uncertainty in the local phase
assignment.

For $\alpha\in(0,1/2)$, define the pointwise uncertainty band
\begin{equation*}
\mathcal U_{\alpha}
=\left\{
\boldsymbol{x}\in\Omega:
\alpha\leq p_D(\boldsymbol{x})\leq1-\alpha
\right\}.
\label{eq:num-uncertainty-band}
\end{equation*}
Because both the exact interface and the uncertainty band are represented
on a discrete grid, their direct intersection can be sensitive to grid
alignment. We therefore measure whether each point of the exact interface
lies within a prescribed distance $\eta>0$ of the uncertainty band. The
continuous coverage measure is
\begin{equation}
\operatorname{BCov}_{\alpha,\eta}
=\mathcal H^1\!\left(
\left\{\boldsymbol{x}\in\Gamma_{\mathrm{true}}:
\operatorname{dist}(\boldsymbol{x},\mathcal U_{\alpha})
\leq\eta\right\}\right)/
\mathcal H^1(\Gamma_{\mathrm{true}}),
\label{eq:num-continuous-coverage}
\end{equation}
where $\mathcal H^1$ denotes arc length.

Let $\Omega_h$ be the set of computational grid nodes and define
\begin{equation*}
\mathcal U_{\alpha,h}
=
\left\{
\boldsymbol{x}_j\in\Omega_h:
\alpha\leq p_D(\boldsymbol{x}_j)\leq1-\alpha
\right\}.
\end{equation*}
If the exact interface is represented by
$\mathcal B_{\mathrm{true}}
=\{\boldsymbol{x}_i^{\mathrm{true}}\}_{i=1}^{N_{\mathrm{true}}}$,
an arc-length discretization of \eqref{eq:num-continuous-coverage} is
\begin{equation*}
\operatorname{BCov}_{\alpha,\eta}
\approx
\displaystyle
\sum_{i=1}^{N_{\mathrm{true}}}
\Delta s_i\,
\mathbf 1_{\{
\operatorname{dist}(
\boldsymbol{x}_i^{\mathrm{true}},\mathcal U_{\alpha,h})
\leq\eta\}} / 
\displaystyle
\sum_{i=1}^{N_{\mathrm{true}}}\Delta s_i,
\label{eq:num-weighted-coverage}
\end{equation*}
where $\Delta s_i$ is the arc-length weight associated with the $i$th boundary point.

The numerical implementation assigns equal weight to the extracted
boundary points. The reported quantity is therefore
\begin{equation*}
\operatorname{BCov}^{\mathrm{node}}_{\alpha,\eta}
= \frac{1}{N_{\mathrm{true}}}
\sum_{i=1}^{N_{\mathrm{true}}}
\mathbf 1_{\{
\operatorname{dist}(
\boldsymbol{x}_i^{\mathrm{true}},\mathcal U_{\alpha,h})
\leq\eta\}}.
\label{eq:num-node-coverage}
\end{equation*}
Here the distance is the minimum Euclidean distance from
$\boldsymbol{x}_i^{\mathrm{true}}$ to the nodes in
$\mathcal U_{\alpha,h}$. We set this distance to $+\infty$ when
$\mathcal U_{\alpha,h}$ is empty. 
\end{document}